\documentclass[oneside,final,11pt]{amsart}
\usepackage[T1]{fontenc}
\usepackage[latin9]{inputenc}
\usepackage{geometry}
\usepackage{graphicx}
\usepackage{soul}
\usepackage{xcolor}
\usepackage{amssymb}

\newtheorem{theorem}{Theorem}[section]
\newtheorem{lemma}[theorem]{Lemma}
\newtheorem{proposition}[theorem]{Proposition}
\newtheorem{corollary}[theorem]{Corollary}
{ \theoremstyle{definition}
\newtheorem{definition}[theorem]{Definition}}
{ \theoremstyle{definition}
}
{ \theoremstyle{remark}
\newtheorem{remark}[theorem]{Remark}}
\usepackage{placeins}
\usepackage{cite}
\usepackage{caption}
\usepackage{enumerate}
\usepackage{enumitem}
\usepackage{bmpsize}
\usepackage{stmaryrd}
\usepackage{mathtools}

\newcommand\eqd{\stackrel{\mathclap{\normalfont d}}{=}}

\begin{document}

\title[Transversal fluctuations of the ASEP]{Transversal fluctuations of the ASEP, stochastic six vertex model, and Hall-Littlewood Gibbsian line ensembles}

\author{Ivan Corwin and Evgeni Dimitrov}

\begin{abstract}
We consider the ASEP and the stochastic six vertex model started with step initial data. After a long time, $T$, it is known that the one-point height function fluctuations for these systems are of order $T^{1/3}$. We prove the KPZ prediction of $T^{2/3}$ scaling in space. Namely, we prove tightness (and Brownian absolute continuity of all subsequential limits) as $T$ goes to infinity of the height function with spatial coordinate scaled by $T^{2/3}$  and fluctuations scaled by $T^{1/3}$.
The starting point for proving these results is a connection discovered recently by Borodin-Bufetov-Wheeler between the stochastic six vertex height function and the Hall-Littlewood process (a certain measure on plane partitions). Interpreting this process as a line ensemble with a Gibbsian resampling invariance, we show that the one-point tightness of the top curve can be propagated to the tightness of the entire curve.
\end{abstract}
\maketitle

\setcounter{tocdepth}{2}
\tableofcontents

\section{Introduction}\label{Section1}

\subsection{The ASEP and S6V model}\label{Section1.1}
In this paper we prove, as Theorem \ref{ASEPTight} and Corollary \ref{SVTight}, the long predicted transversal $2/3$ exponent for the asymmetric simple exclusion process (ASEP) \cite{Spi, Lig2} and the stochastic six vertex (S6V) model \cite{Gwa} -- two (closely related) $1+1$ dimensional random interface growth models / interacting particle systems in the Kardar-Parisi-Zhang (KPZ) universality class. We work with step initial data for both models and demonstrate that their height functions, scaled in space by $T^{2/3}$ and in fluctuation size by $T^{1/3}$, are tight as spatial processes as time $T$ goes to infinity (we use $T$ for time since $t\in (0,1)$ will be reserved for the Hall-Littlewood parameter). We also show as Corollary \ref{ACC}, that all subsequential limits of the scaled height function (shifted by a parabola) have increments, which are absolutely continuous with respect to a Brownian bridge measure. Conjecturally the limit process should be the Airy$_2$ process and we provide further evidence for this conjecture by uncovering a Gibbsian line ensemble structure behind these models, which formally limits to that of the Airy line ensemble \cite{CorHamA}.

\subsubsection{Background and literature}
In 1986, Kardar, Parisi, and Zhang \cite{KPZ} predicted that a large class of growth models in one-spatial dimension subject to space-time independent random forcing, and lateral growth would all demonstrate the same universal scaling properties in long time (see also \cite{KSB}). In particular, drawing on the earlier 1977 work of Forster, Nelson and Stephen \cite{FNS} (which involved a non-rigorous dynamic renormalization group study of the stochastic Burgers equation -- a continuum interacting particle system), \cite{KPZ} predicted that these growth models would have fluctuations of order $T^{1/3}$ in the direction of growth, and have non-trivial spatial correlations in the $T^{2/3}$ transversal scale (the exact nature of this correlation structure was not understood until later). Here time $T$ is assumed to be large. This $3:2:1$ scaling of time : space : fluctuation (known now as KPZ scaling) immediately caught the imagination of physicists and then in the late 90s, mathematicians. These researchers generally sought to refine and expand the scope and notion of this KPZ universality class through numerics, experiments, non-rigorous methods, and in some limited cases mathematical proofs. Much of this progress and a broader discussion about KPZ universality can be found in surveys and books such as \cite{CU2,QS, HT, BS} and references given therein.

Rigorous results concerning the KPZ class generally come in two flavors -- those mainly reliant upon delicate underlying integrable structures in certain exactly solvable models (see, for example, the surveys \cite{Cor14,BPIP,BPIP2,BorGorIP, BorCor}), and those mainly reliant upon more flexible probabilistic methods like couplings or resolvent equations (see for example \cite{BalS, QV,BQS, QV2}). The integrable results provide the highest resolution and have led to exact calculations of statistics for KPZ fluctuations. The more probabilistic methods are more widely applicable and have met with some success in extending KPZ universality outside the realm of integrable models.

Recently, there have been a few hybrid works which have recast certain elements of the integrable model structures into more probabilistic language, and consequently provided new tools in asymptotics. Examples of such works are \cite{CorHamA, CorHamK}, which introduce a method to show tightness of Gibbsian line ensembles from one-point tightness, and use that to probe the fluctuations of the Airy$_2$ process and the KPZ equation. We will have more to say on this method in Section \ref{sec.HLGLE}, since our results ultimately rely on a variant of it. Another such hybrid method is that of continuum statistics \cite{CQR} (see also \cite{MFQR, BorCorRem, QR}) which recasts the TASEP multipoint fluctuation formulas in terms of a kernel which solves a simple boundary value problem. It turns out that this recast kernel admits a simple limit as the number of points of interest grows. Recently, \cite{MQR} extends this method to take limits of general initial data formulas of \cite{Sas05, BFPS} so as to give the full transition probabilities for the KPZ fixed point.

Returning to the $2/3$ KPZ transversal exponent, our results are not the first for KPZ class models (though they are for the ASEP and S6V model). The TASEP (or exponential/geometric last passage percolation (LPP) and the longest increasing subsequence (LIS) of random permutations) is solvable via Schur / determinantal point processes. Using this, \cite{Spohn,Joh02} extract multipoint limit theorems in the $T^{2/3}$ transversal scale and show that from step initial data (particles start packed to one side of the origin and empty to the other), the TASEP height function converges to the Airy$_2$ process (introduced in those works). Under the same scaling, but for other initial data (e.g. periodic, or Bernoulli) other limit processes arise \cite{BFPS,BFP, CFP ,MQR}.

For LPP, LIS, and directed polymer models there is another closely related notion of transversal fluctuation which relates to the wandering of the maximizing (or polymer) path. That exponent is also $2/3$ as was first demonstrated for the LIS in \cite{Joh00b} and for directed LPP with geometric weights in \cite{BDMM}. The KPZ relation predicts that twice the transversal exponent should equal one plus the fluctuation exponent -- in our case $2(2/3) = 1+ 1/3$. That relation has been shown to hold if one makes very strong assumptions on existence of exponents \cite{ Chatt,Dam14,Dam12}. The ASEP and S6V model are not mappable to polymers, so these results do not apply here. For that matter, even if they did, the notion of existence of exponents in those works are very strong and to our knowledge have not yet been verified for any models, even those that are exactly solvable.

Whereas the TASEP is solvable via determinantal / Schur process methods, the ASEP and S6V model are not. There are two main algebraic structures which produce tools for the analysis of integrable KPZ class models (including the TASEP, ASEP and S6V model) -- Macdonald processes and quantum integrable systems. In fact, there are now some bridges between these structures which suggest that they will eventually be joined together. Both of these structures produce moment formulas, which in principle entirely characterize the distribution of the probabilistic systems in question at a given time. However, it remains a significant challenge to extract meaningful asymptotics from these formulas. So far, besides in the very special determinantal models, it is only for one-point asymptotics that this has been achieved. Since formulas are similar in the Macdonald and quantum integrable systems cases, there is a fairly established route now to one-point asymptotics.

Focusing on step initial data, the ASEP one-point $T^{1/3}$ fluctuations were established first in \cite{TWASEP2} and then for other related methods in \cite{BCS,Cor12,BorO}. Analogous asymptotic results are proved for the S6V model in \cite{BCG14, Bor16}, KPZ equation in \cite{CorQ,BCF}, semi-discrete directed polymer in \cite{BorCor}, inverse-gamma polymer in \cite{BCR,KQ}, and $q$-TASEP in \cite{FerVet,Bar15}. There are other models which fall into these hierarchies of integrability for which similar results have been demonstrated (see, for example, \cite{Cor14} and some references therein), and there are some other types of initial data which have been dealt with (though not yet flat -- see however \cite{OQR} for preasymptotic progress on this).

From exact formulas (e.g. \cite{NZ, BorCor,BCS,BP16 }), multipoint fluctuations for the ASEP/S6V model and other non-determinantal models have proved elusive so far. There have been some non-rigorous attempts (e.g. \cite{ PSpohn, PSpohn2, Dot13,Dot14,  ISS13}) for related models (KPZ equation and inverse-gamma polymer) by use of certain uncontrolled approximations. It is unclear whether the resulting formulas are correct. From more probabilistic methods, \cite{Sep12,Sep10} demonstrate the $2/3$ transversal exponent in terms of the typical scale fluctuations of the polymer path for the inverse-gamma and semi-discrete Brownian directed polymer models (with stationary boundary conditions).

Employing the hybrid approach, by associating the narrow wedge initial data solution to the KPZ equation to a Gibbsian line ensemble, \cite{CorHamK} gives the first rigorous proof of the $2/3$ transversal exponent for the KPZ equation. In particular, they show that the spatial fluctuations are tight in this transversal scale and under fluctuation scaling by exponent $1/3$, and, moreover, that subsequential limits are locally absolutely continuous with respect to Brownian motion. The starting point for this result is the remarkable fact that the KPZ equation solution at a fixed time can be realized as the top indexed curve of an infinite ensemble of curves (called the KPZ line ensemble -- see also \cite{OCW, Nica}) which interact with nearest neighbor curves through an exponential energy. At the heart of this existence is the relation between the semi-discrete directed polymer and the Whittaker process and quantum Toda lattice Hamiltonian, as facilitated by the geometric Robinson-Schensted-Knuth (RSK) correspondence \cite{OC} (see also \cite{CCSZ, CSZ, NZ}).

Until recently, it did not seem that the ASEP or S6V model enjoyed such a relationship with Gibbsian line ensembles. The work of \cite{BBW} shows that the S6V height function (and that of the ASEP through a limit transition) arises as the top curve of a line ensemble related to the Hall-Littlewood process. In fact, \cite{BM} demonstrates a Hall-Littlewood variant of the RSK correspondence which makes this relationship all the more transparent. As we discuss below in Sections \ref{sec.HLGLE} and \ref{Section3}, the Hall-Littlewood line ensemble has a slightly more involved Gibbs property which requires us to develop some new techniques beyond those of \cite{CorHamK} in order to demonstrate our tightness results.\\

\subsubsection{Our main results}
We now state our main results concerning the ASEP. Precise definitions of this model and further discussion can be found in Section \ref{Section2.3}. We forgo stating the S6V model result until the main text -- Corollary \ref{SVTight} -- since it requires more notation to define the model.

In the ASEP, particles occupy sites indexed by $\mathbb{Z}$ with at most one particle per site (the exclusion rule) and jump according to independent exponential clocks to the right and left with rates $R$ and $L$ respectively ($R > L$ is assumed). Jumps that would violate the exclusion rule are suppressed. Step initial data means that particles start at every site in $\mathbb{Z}_{\leq 0}$ (and no particles start elsewhere). The height function $\mathfrak{h}_T(x)$ records the number of particles at or to the right of position $x\in \mathbb{Z}$ at time $T$. For $x\notin \mathbb{Z}$ we linearly interpolate to make the height function continuous. With this notation we can state our main theorem (Theorem \ref{ASEPTight} and Corollary \ref{ACC} in the main text).
\begin{theorem}\label{TMain}
Suppose $r > 0$, $R = 1$, $L \in (0,1)$, $\gamma = R- L$ and fix $\alpha \in (0, 1)$. For $s \in [-r,r]$ set
\begin{equation}
f^{ASEP}_N(s) =  \sigma_\alpha^{-1}N^{-1/3} \left(f_3(\alpha)N + f'_3(\alpha)s N^{2/3} + (1/2) s^2 f_3''(\alpha)N^{1/3} - \mathfrak{h}_{N/ \gamma} \left( \alpha N + sN^{2/3}\right) \right),
\end{equation}
The constants above are given by $\sigma_\alpha = 2^{-4/3}(1 - \alpha^2)^{2/3}$, $f_3(\alpha) = \frac{(1-\alpha)^2}{4}$, $f'_3(\alpha) = -\frac{1 - \alpha}{2}$, $f''_3(\alpha) = \frac{1}{2}$. If $\mathbb{P}_N$ denotes the law of $f^{ASEP}_N(s)$ as a random variable in $(C[-r,r], \mathcal{C})$ --- the space of continuous functions on $[-r,r]$ with the uniform topology and Borel $\sigma$-algebra $\mathcal{C}$ (see e.g. Chapter 7 in \cite{Bill}) --- then the sequence $\mathbb{P}_N$ is tight. 

Moreover, if $\mathbb{P}_\infty$ denotes any subsequential limit of $\mathbb{P}_N$ and $f^{ASEP}_\infty$ has law $\mathbb{P}_\infty$, then $g^{ASEP}_\infty$ defined by 
$$g^{ASEP}_\infty(x) =\sigma_{\alpha} f^{ASEP}_\infty(x) -  \frac{x^2 f_3''(\alpha)}{2 }, \mbox{ for $x \in [-r,r]$}, $$
is absolutely continuous with respect to a Brownian bridge of variance  $-2r   f'_3(\alpha) [1 + f'_3(\alpha)]$ in the sense of Definition \ref{DACB}.
\end{theorem}

Our approach for proving Theorem \ref{TMain} is to (1) embed the ASEP height function into a line ensemble, which enjoys a certain `Hall-Littlewood Gibbs' resampling property, and (2) use the known one-point tightness in the $T^{1/3}$ fluctuation scale to obtain the $T^{2/3}$ transversal tightness. These two points will be discussed more extensively in the section below. Here we mention that the Gibbs property implies that conditional on the second curve in the line ensemble, the top curve (i.e. the height function) has a law expressible in terms of an explicit Radon-Nikodym derivative with respect to the trajectory of a random walk. By controlling this Radon-Nikodym derivative as $T$ goes to infinity, we are able to control quantities like the maximum, minimum and modulus of continuity of the prelimit continuous curves, which translates into a tightness statement in the space of continuous curves. By exploiting a strong coupling of random walk bridges and Brownian bridges we can further deduce the absolute continuity of subsequential limits with respect to Brownian bridges of appropriate variance.
\subsection{Hall-Littlewood Gibbsian line ensembles}\label{sec.HLGLE} 

\subsubsection{Line ensembles and resampling} The central objects that we study in this paper are discrete line ensembles, which satisfy what we call the Hall-Littlewood Gibbs property. In what follows we describe the general setup informally, and refer the reader to Section \ref{Section3.1} for the details. 

A discrete line ensemble is a finite collection of up-right paths $\{L_i\}_{i = 1}^k$ drawn on the integer lattice, which we assume to be weakly ordered, meaning that $L_i(x) \geq L_{i+1}(x)$ for $i = 1,..., k-1$, and all $x$. The up-right paths $L_i$ are understood to be continuous curves on some interval $I = [a,b]$, and to be piecewise constant or have slope $1$ (see Figure \ref{S1_1} for examples). 
\begin{figure}[h]
\centering
\scalebox{0.45}{\includegraphics{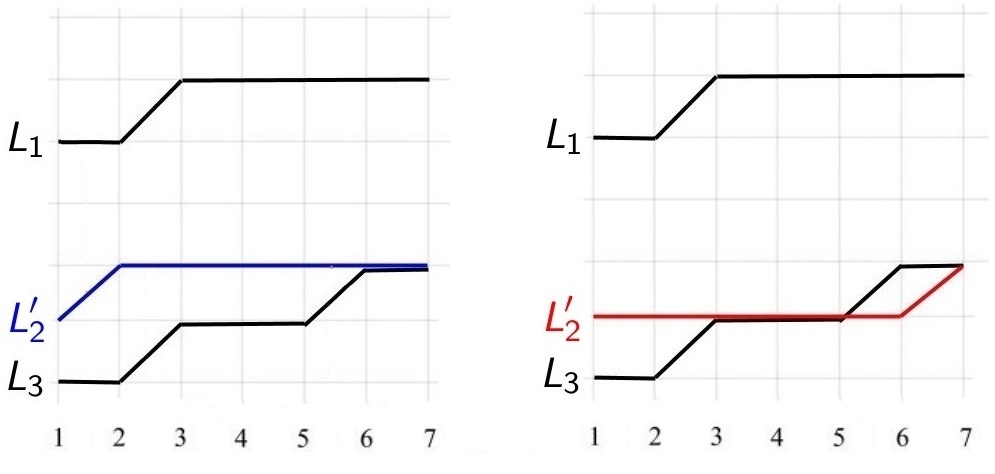}}
\caption{The black lines are a sample from a discrete line ensemble $\{L_i\}_{i = 1}^k$ with $k = 3$ ($L_2$ is not drawn and coincides with the blue line above). Each line is a continuous curve on $I = [1,7]$ that is piecewise constant or has slope $1$. The red and blue lines are uniformly sampled up-right paths connecting the endpoints $(1,1)$ and $(7, 2)$ of $L_2$. \\
}
\label{S1_1}
\end{figure}
Suppose we are given a probability distribution $\mu$ on the set of ensembles $\{L_i\}_{i = 1}^k$. We will consider the following resampling procedure. Fix any $i \in \{1,\dots,k-1\}$ and denote by $f = L_{i-1}$ and $g = L_{i+1}$ with the convention that $L_0 = +\infty$. Sample $\{L_i\}_{i = 1}^k$ according to $\mu$ and afterwards erase the line $L_i$, between its endpoints $A = L_i(a)$ and $B = L_i(b)$. Sample a new path $L_i'$, connecting the points $(a,A)$ and $(b,B)$ from the uniform distribution on all up-right paths that connect these points, and independently accept the path with probability $W_t(L_i', f,g)$. If the new path is not accepted the same procedure is repeated until a path is accepted. We say that $\mu$ has the Hall-Littlewood Gibbs property with parameter $t \in (0,1)$ if given $\{L_i\}_{i = 1}^k$ distributed according to $\mu$, the random path ensemble obtained from the above resampling procedure again has distribution $\mu$. The acceptance probability is given by
\begin{equation}\label{AP}
W_t(L_i', f,g) = \prod_{s = a+1}^b  \hspace{-2mm} \left(1 -  {\bf 1}_{\{\Delta^+(s-1) - \Delta^+(s) = 1\}}\cdot t^{\Delta^+(s-1)} \right) \hspace{-1mm} \cdot \hspace{-1mm}\left(1 -  {\bf 1}_{\{\Delta^-(s-1) - \Delta^-(s) = 1\}}\cdot t^{\Delta^-(s-1)} \right) ,
\end{equation}
where $ \Delta^+(s) = f(s) - L'_i(s)$ and $\Delta^-(s) = L'_i(s) -g(s)$. The above expression can be understood as follows. Follow the path $L'_i$ from left to right and any time $f - L'_i $ decreases from $\Delta^+$ to $\Delta^+ - 1$ at location $s-1$ we multiply by a new factor $1 - t^{\Delta^+(s-1)}$. Similarly,  any time $ L'_i - g$ decreases from $\Delta^-$ to $\Delta^- - 1$ at location $s-1$ we multiply by a new factor $1 - t^{\Delta^-(s-1)}$. Observe that by our assumption on $t$ we have that $W_t(L_i',f,g) \in [0, 1]$, which is why we can interpret it as a probability. 

We make a couple of additional observations about the acceptance probability $W_t(L_i',f,g)$. By assumption $f(a) \geq L'_i(a) \geq g(a)$ and $f(b) \geq L'_i(b) \geq g(b)$. If for some $s$ we fail to have $f(s) \geq L'_i(s) \geq g(s)$, we see that one of the factors in $W_t(L_i',f,g)$ is zero and we will never accept such a path. Consequently, the resampling procedure always maintains the relative order of lines in the ensemble. An additional point we make is that if $L'_i$ is very well separated from $f$ and $g$ (in particular, when $f = +\infty$) we have that $\Delta^{\pm}$ is very large and so the factors in the definition of $W_t(L_i', f,g)$ are close to $1$. In this sense, we can interpret $W_t(L_i',f,g)$ as a deformed indicator function of the paths $f, L_i',g$ having the correct order, the deformation being very slight if the paths are well-separated.
\vspace{2mm}

{\raggedleft \bf Example:} We give a short example of resampling $L_2$ to explain the resampling procedure, using Figure \ref{S1_1} as a reference. We will calculate the acceptance probability if the uniform path we sampled is the red or blue one in Figure \ref{S1_1}. If $L_{red}'$ denotes the red line, we have $W_t(L_{red}', L_1, L_3) = 0$ because the lines $L_3$ and $L_{red}'$ go out of order. In particular, we see that $\Delta^-(s-1) = 0$ and $\Delta^-(s) = -1$ when $s = 6$, which means that the factor $\left(1 -  {\bf 1}_{\{\Delta^-(s-1) - \Delta^-(s) = 1\}}\cdot t^{\Delta^-(s-1)} \right)$ is zero. Such a path is never accepted in the resampling procedure.

If $L_{blue}'$ denotes the blue line, we have $W_t(L_{blue}', L_1, L_3) = (1-t)(1-t^2)(1-t^3).$ To see the latter notice that $\Delta^+$ decreases at location $1$ from $3$ to $2$, producing the factor $(1-t^3)$. On the other hand, $\Delta^-$ decreases from $2$ to $1$ and from $1$ to $0$ at locations $2$ and $5$ respectively, producing factors $(1-t^2)$ and $(1-t)$. The product of all these factors equals $W_t(L_{blue}', L_1, L_3)$ and with this probability we accept the new path.\\

The main result we prove for the Hall-Littlewood Gibbsian line ensembles appears as Theorem \ref{PropTightGood} in the main text. It is a general result showing how one-point tightness for the top curve of a sequence of Hall-Littlewood Gibbsian line ensembles translates into tightness for the entire top curve. This theorem can be considered the main technical contribution of this work, and we deduce tightness statements for different models like the ASEP by appealing to it. It is possible that under some stronger (than tightness) assumptions, one might be able to extend the results of that theorem to tightness of the entire ensemble (i.e. all subsequent curves too) -- but since we do not need this for our applications, we do not pursue it here.

This idea of using the Gibbs property to propagate one-point tightness to tightness of the entire ensemble was developed in \cite{CorHamA,CorHamK}. In those works, the Gibbs property was either non-intersecting or an exponential repulsion. In other words, curves are penalized by either an infinite energetic cost or an exponential energetic cost for moving out of their indexed order. Those works rely fundamentally upon certain stochastic monotonicity enjoyed by such Gibbsian line ensemble. Namely, if you consider a given curve and either shift the starting/ending points of that curve up, or shift the above/below curves up, then the conditional measure of the given curve will stochastically shift up too. Since the Hall-Littlewood Gibbs property relies on not just the distance between curves, but on their relative slope (or derivative of the distance), this type of monotonicity is lost. Indeed, it is not just the proof of the monotonicity, but the actual result which no longer holds true in our present setting (see Remark \ref{MonCoup}).

Faced with the loss of the above form of monotonicity, we had to find a weak enough variant of it which would actually be true, while being strong enough to allow us to rework various types of arguments from \cite{CorHamA,CorHamK}. Lemma \ref{LemmaMon} (and its corollaries) ends up fitting this need. In essence, it says that the acceptance probability of the top curve increases (though only in terms of its expected value and up to a factor of $c(t) = \prod_{i=1}^{\infty} (1-t^i)$) as the curve is raised. Informally, this result is a weaker version of the stochastic monotonicity of \cite{CorHamA,CorHamK} in that pointwise inequalities are replaced with ones that hold on average and upto an additional factor. Armed with this result, we are able to redevelop a route to prove tightness of the entire top line of the ensemble from its one-point tightness. Our approach should apply for more general Gibbs properties which rely upon not just the relative separation of lines, but also their relative slopes. Indeed, the constant $c(t)$ arises in our case as a relatively crude estimate needed to handle the dependence of our weights on the derivative of the distance between the top two curves. If the dependence of the weights becomes different, one should be able to reproduce the same arguments, with only the constant $c(t)$ changing its value.

\subsubsection{The homogeneous ascending Hall-Littlewood process} The prototypical model behind the Hall-Littlewood Gibbsian line ensemble of the previous section is the (homogeneous ascending) {\em Hall-Littlewood process} (HAHP). The HAHP (a special case of the ascending Macdonald processes \cite{BorCor}) is a probability distribution on interlacing sequences $\varnothing \prec \lambda(1) \prec \lambda(2) \prec \cdots \prec \lambda(M)$, where $\lambda(i)$ are partitions (see the beginning of Section \ref{Section2.1} for some background on partitions, Young diagrams etc.). It depends on two positive integers $M$ and $N$ as well as two parameters $t, \zeta \in (0,1)$. We will provide a careful definition in terms of symmetric functions in Section \ref{Section2.1} later, but here we want to give a more geometric interpretation of this measure. 
In what follows we will describe a measure on interlacing sequences of partitions $\varnothing \prec \lambda^{-M+1} \prec \cdots \prec \lambda^0 \succ \lambda^1 \succ \cdots \succ \lambda^{N-1} \succ \varnothing$. The HAHP is then recovered by restriction to the first $M$ partitions of this sequence. The description we give dates back to \cite{Vul2}, and we emphasize it here as it is the origin of the Hall-Littlewood Gibbs property that we use.

We can associate an interlacing sequence of partitions with a boxed plane partition or 3d Young diagram, which is contained in the $M \times N$ rectangle --  Figure \ref{S1_2} provides an illustration of this correspondence. Consequently, measures on interlacing sequences are equivalent to measures on boxed plane partitions and we focus on the latter. If a plane partition $\pi$ is given, we define its weight by
\begin{equation}\label{Apit}
W(\pi) = A_{\pi}(t) \times \zeta^{diag(\pi)},
\end{equation}
where $diag(\pi)$ denotes the sum of the entries on the main diagonal of $\pi$ (alternatively this is the sum of the parts of $\lambda^0$ or the number of cubes on the diagonal $x= y$ in the 3d Young diagram). The function $ A_{\pi}(t)$ depends on the geometry of $\pi$ and is described in Figure \ref{S1_2} (see also Section 1 of \cite{ED} for a more detailed explanation). With the above notation, we have that the probability of a plane partition is given by the weight $W(\pi)$, divided by the sum of the weights of all plane partitions. 
\begin{figure}[h]
\centering
\scalebox{0.6}{\includegraphics{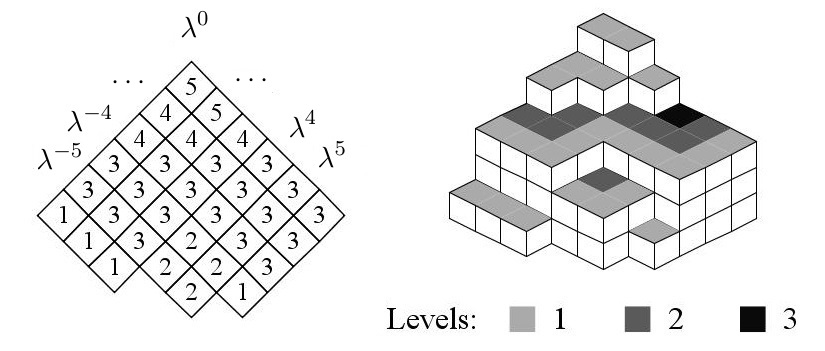}}
\caption{If given a sequence $\varnothing \prec \lambda^{-M+1} \prec \cdots \prec \lambda^0 \succ  \cdots \succ \lambda^{N-1} \succ \varnothing$ we write the parts of $\lambda^i$  downward -- in this way we obtain a plane partition. The left part of the figure shows how to do this when $\lambda^{-5} = (3)$, $\lambda^{-4} = (3,1)$, $\lambda^{-3}= (3,3,1)$ and so on. In this example $N = M = 6$. The right part of the figure shows the corresponding 3d Young diagram. The entry in a cell of the plane partition corresponds to the number of cubes in a vertical stack of the 3d diagram.\\
For the above diagram we have $diag(\pi) = 5 + 4 + 3 +2 + 2 = 16$.\\
To find $A_{\pi}(t)$ we do the coloring in the right part of the figure. Each cell gets a level, which measures the distance of the cell to the boundary of the terrace on which it lies. We consider connected components (formed by cells of the same level that share a side) and for each one we have a factor $(1 - t^i)$, where $i$ is the level of the cells in the component. The product of all these factors is $A_{\pi}(t)$. For the example above we have $7$ components of level $1$, $3$ of level $2$ and one of level $3$ -- thus $A_{\pi}(t)= (1-t)^7 (1-t^2)^3(1-t^3)$.
}
\label{S1_2}
\end{figure}

Let us denote $\lambda(i) = \lambda^{i - M}$ for $i = 1,...,M$. Then the HAHP is the probability distribution induced from the weights (\ref{Apit}) and projected to the first $M$ terms $\varnothing \prec \lambda(1) \prec \lambda(2) \prec \cdots \prec \lambda(M)$. Denoting by $\lambda'$ the transpose of a partition $\lambda$ we observe that $\{ \lambda'_j( \cdot) \}_{j = 1}^N$ defines a discrete line ensemble on the interval $[0,M]$. In the above geometric setting, the lines in the discrete line ensemble $\{ \lambda'_j( \cdot) \}_{j = 1}^N$ can be associated to the level lines of $\pi$ (in particular, $\lambda'_1(\cdot)$ corresponds to the bottom slice of the plane partition $\pi$). The important point we emphasize is that the geometric interpretation of $A_{\pi}(t)$ above can be seen to be equivalent with the statement that the line ensemble $\{ \lambda'_j( \cdot) \}_{j = 1}^N$ satisfies the Hall-Littlewood Gibbs property of the previous section. The latter is proved in Proposition \ref{HLGPP} in the main text.

The main result we prove for the HAHP is that as $M,N$ tend to infinity the top line $\lambda'_1(\cdot)$ (or alternatively the bottom slice of $\pi$), appropriately shifted and scaled, is tight -- this is Theorem \ref{HLTight} in the text. The bottom slice of a similar (though slightly different) Hall-Littlewood random plane partition was recently investigated in \cite{ED}, using ideas from \cite{BorCor}, where it was shown that the one-point marginals are governed by the Tracy-Widom distribution. In Theorem \ref{HLCT} we combine arguments from that paper as well as \cite{BCG14} to show that the same is true for the model we described above. This convergence implies in particular one-point tightness for the top line of the ensemble $\{ \lambda'_j( \cdot) \}_{j = 1}^N$. Once the one-point tightness and Hall-Littlewood Gibbs property are established we enter the setup Theorem \ref{PropTightGood}, from which Theorem \ref{HLTight} is deduced.
\subsubsection{Connection to the ASEP and S6V model} In this section we explain how the ASEP and S6V model fit into the setup of Hall-Littlewood Gibbsian line ensembles.

For the S6V model, the key ingredient comes from the remarkable recent work in \cite{BBW}. In particular, Theorem 4.1 in \cite{BBW} (recalled as Theorem \ref{SVHL} in the main text),  shows that the top curve $\lambda'_1$ of the line ensemble $\{ \lambda'_j( \cdot) \}_{j = 1}^N$ of the previous section has the same distribution as the height function on a horizontal slice of the S6V model, with appropriately matched parameters. This equivalence relies on the use of the $t$-Boson vertex model, as well as the infinite volume limit of the Yang-Baxter equation (as developed, for instance, in \cite{Kor13,Bor14,BWh}). Alternatively, \cite{BM} relates this distributional equality to a Hall-Littlewood version of the RSK correspondence. Through this identification one deduces the predicted transversal $2/3$ exponent for the height function of the S6V model as a corollary of the HAHP result Theorem \ref{HLTight} -- the exact statement is given in Corollary \ref{SVTight} in the text.

We now explain how to relate the ASEP to our line ensemble framework. Recall from Section \ref{Section1.1} that $\mathfrak{h}_T(x)$ denotes the height function of the ASEP with rates $R$ and $L$, started from step initial condition at time $T$. Set $R = 1$ and $L = t \in (0,1)$. Since we use linear interpolation to define $\mathfrak{h}_T(x)$ for non-integer $x$, one observes that $-\mathfrak{h}_T(x)$ either stays constant or goes up linearly with slope $1$ as $x$ increases, i.e. $-\mathfrak{h}_T(x)$ is an up-right path. In Proposition \ref{HLASEP} we show that for any $T > 0$ and $k,K \in \mathbb{N}$ there is a random discrete line ensemble $\{L^{ASEP}_i\}_{i = 1}^k$ on $I = [-K, K]$ such that (1) the law of $\{L^{ASEP}_i\}_{i = 1}^k$ satisfies the Hall-Littlewood Gibbs property and (2) $L^{ASEP}_1$ has the same law as $-\mathfrak{h}_T(x)$, restricted to $x \in [-K,K]$. The realisation of $-\mathfrak{h}_T(x)$ as the top line in a Hall-Littlewood Gibbsian line ensemble is an important step in our arguments and we will provide some details how this is accomplished in a moment. For now let us explain the implications of this fact.

Once we have that  $\{L^{ASEP}_i\}_{i = 1}^k$ satisfies the Hall-Littlewood Gibbs property, we can use Theorem \ref{PropTightGood} to reduce the spatial tightness of the top curve $L^{ASEP}_1$ (i.e. the negative height function $-\mathfrak{h}_T(\cdot)$) to the one-point tightness of its height function. The latter is a well-known fact -- it follows from the celebrated theorem of Tracy-Widom  \cite[Theorem 3]{TWASEP}, and is recalled as Theorem \ref{ASEPCT} in the main text. Consequently, once $-\mathfrak{h}_T(x)$ is understood as the top line of a discrete line ensemble with the Hall-Littewood Gibbs property, the general machinery of Theorem \ref{PropTightGood} takes over and produces the tightness statement of Theorem \ref{TMain}. 

Let us briefly explain how we construct the line ensemble $\{L^{ASEP}_i\}_{i = 1}^k$  from earlier -- see Proposition \ref{HLASEP} for the details. One starts from a sequence of HAHP with parameters $\zeta_N = 1 - \frac{1-t}{N}$. Under suitable shifts and truncations, these line ensembles give rise to a sequence of line ensembles $\{L^{N}_i\}_{i = 1}^k$, which one can show to be tight. One defines $\{L^{ASEP}_i\}_{i = 1}^k$ as a subsequential limit of this sequence.  Since the HAHP satisfies the Hall-Littlewood Gibbs property one deduces the same for $\{L^{ASEP}_i\}_{i = 1}^k$. The property that $L^{ASEP}_1$ has the same law as $-\mathfrak{h}_T(x)$ follows from the connection between the HAHP and the S6V model height function we discussed above and the convergence of the height function of the S6V model to $\mathfrak{h}_T(x)$. The fact that one can obtain the ASEP height function through a limit transition of the S6V model was suggested in \cite{Gwa, BCG14} with a complete proof given in \cite{Agg16}. \\

We end this section with a brief discussion on possible extensions of our results. In Theorem \ref{HLTight},  Corollary \ref{SVTight} and Theorem \ref{ASEPTight} we construct sequences of random continuous curves, which are tight in the space of continuous curves. We believe that the same sequences should converge to the Airy$_2$ process -- that is how the particular scaling constants in those results were chosen. The missing ingredient necessary to establish this is the convergence of several-point marginals of these curves (currently only one-point convergence is known). It is possible that such several point-convergence will come from integrable formulas for these models but we also mention here a possible alternative approach. One could try to enhance the arguments of this paper to show that the one-point convergence of the top line of a Hall-Littlewood Gibbsian line ensemble in fact implies tightness of the entire line ensemble (not just the top curve). This was done in a continuous setting in \cite{CorHamA,CorHamK}. If one achieves the latter and \cite[Conjecture 3.2]{CorHamA} were proved, this would provide a means to prove that the entire line ensemble corresponding to the ascending Hall-Littlewood process converges to the Airy line ensemble. In particular, this would demonstrate the Airy$_2$ process limit for the ASEP and S6V height functions too.

\subsubsection*{Outline} The introductory section above provided background context for our work and a general overview of the paper. In Section \ref{Section2} we define the HAHP, S6V model and the ASEP and supply some known one-point convergence results for the latter. Section \ref{Section3} introduces the necessary definitions in the paper, states the main technical result -- Theorem \ref{PropTightGood}, as well as the main results we prove about the HAHP, the S6V model and the ASEP in Theorem \ref{HLTight},  Corollary \ref{SVTight} and Theorem \ref{ASEPTight} respectively. Section \ref{Section4} summarizes the primary set of results we need to prove Theorem \ref{PropTightGood}. In Section \ref{Section5} we give the proof of Theorem \ref{PropTightGood} by reducing it to three key lemmas, whose proofs are presented in Section \ref{Section6}. In Section \ref{Section7} we demonstrate that all subsequential limits of the tight sequence of Theorem \ref{PropTightGood} are absolutely continuous with respect to Brownian bridges of appropriate variance. Section \ref{Section8} is an appendix, which contains the proof of a strong coupling between random walks and Brownian bridges, used in Section \ref{Section4}.

\subsubsection*{Acknowledgements}
I.C. would like to thank Alexei Borodin and Michael Wheeler for advanced discussions of their work \cite{BBW} at the workshop ``Quantum Integrable Systems, Conformal Field Theories and Stochastic Processes'' held at the Institut D'Etudes Scientifiques de Cargese (and funded partially by NSF DMS:1637087). I.C. was partially supported by the NSF through DMS-1208998 and DMS-1664650, the Clay Mathematics Institute through a Clay Research Fellowship and the Packard Foundation through a Packard Fellowship for Science and Engineering. E.D. would like to thank Alexei Borodin for numerous useful conversations.

\section{Three stochastic models}\label{Section2}
The results of our paper have applications to three different but related probabilistic objects -- the ascending Hall-Littlewood process, the stochastic six-vertex model in a quadrant and the ASEP. In this section we recall the definitions of these models, some known one-point convergence results about them and explain how they are connected.

\subsection{The ascending Hall-Littlewood process}\label{Section2.1}
In this section we briefly recall the definition of the Hall-Littlewood process (a special case of the Macdonald process \cite{BorCor}). We will isolate a particular case that will be important for us, which we call the {\em homogeneous ascending Hall-Littlewood process} (HAHP) and derive a certain one-point convergence result for it. We start by fixing terminology and notation, using \cite{Mac} as a main reference. 

A {\em partition} is a sequence $\lambda = (\lambda_1, \lambda_2,\cdots)$ of non-negative integers such that $\lambda_1 \geq \lambda_2 \geq \cdots$ and all but finitely many elements are zero. We denote the set of all partitions by $\mathbb{Y}$. The {\em length} $\ell (\lambda)$ is the number of non-zero $\lambda_i$ and the {\em weight} is given by $|\lambda| = \lambda_1 + \lambda_2 + \cdots$ . There is a single partition of $0$, which we denote by $\varnothing$. We write $m_i(\lambda)$ for the {\em multiplicity } of $i$ in $\lambda$, i.e. $m_i(\lambda) = |\{ j \in \mathbb{N}: \lambda_j = i\}|$.

A {\em Young diagram} is a graphical representation of a partition $\lambda$, with $\lambda_1$ left justified boxes in the top row, $\lambda_2$ in the second row and so on. In general, we do not distinguish between a partition $\lambda$ and the Young diagram representing it. The {\em conjugate} of a partition $\lambda$ is the partition $\lambda'$ whose Young diagram is the transpose of the diagram $\lambda$. In particular, we have the formula $\lambda_i' = |\{j \in \mathbb{N}: \lambda_j \geq i\}|$.

Given two diagrams $\lambda$ and $\mu$ such that $\mu \subset \lambda$ (as a collection of boxes), we call the difference $\theta = \lambda - \mu$ a {\em skew Young diagram}. A skew Young diagram $\theta$ is a {\em horizontal $m$-strip} if $\theta$ contains $m$ boxes and no two lie in the same column. If $\lambda - \mu$ is a horizontal $m$-strip for some $m\geq 0$, we write $\lambda \succ \mu$. Some of these concepts are illustrated in Figure \ref{S2_1}.

\begin{figure}[h]
\centering
\scalebox{0.45}{\includegraphics{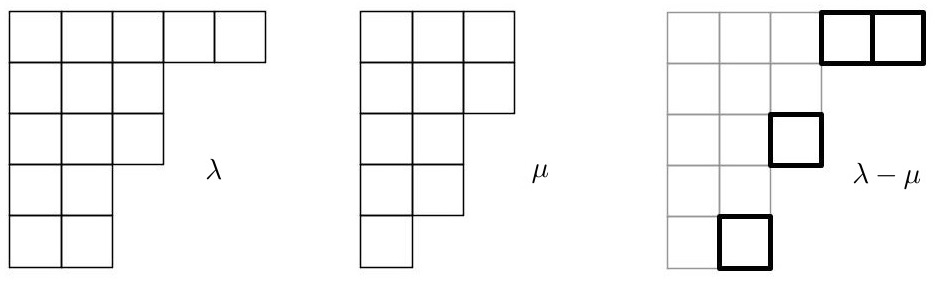}}
\caption{The Young diagram $\lambda = (5,3,3,2,2)$ and its transpose (not shown) $\lambda' = ( 5,5,3,1,1)$. The length $\ell(\lambda) = 5$ and weight $|\lambda| = 15$. The Young diagram $\mu = (3,3,2,2,1)$ is such that $\mu \subset \lambda$. The skew Young diagram $\lambda - \mu$ is shown in {\em black bold lines} and is a horizontal $4$-strip}
\label{S2_1}
\end{figure}

A {\em plane partition} is a two-dimensional array of nonnegative integers
$$\pi = (\pi_{i,j}), \hspace{3mm} i,j = 0,1,2,...,$$
such that $\pi_{i,j} \geq \max (\pi_{i,j+1}, \pi_{i+1,j})$ for all $i,j \geq 0$ and the {\em volume} $|\pi| = \sum_{i,j \geq 0} \pi_{i,j}$ is finite. Alternatively, a plane partition is a Young diagram filled with positive integers that form non-increasing rows and columns. A graphical representation of a plane partition $\pi$ is given by a {\em $3$-dimensional Young diagram}, which can be viewed as the plot of the function 
$$(x,y) \rightarrow \pi_{\lfloor x \rfloor, \lfloor y \rfloor} \hspace{3mm} x,y > 0.$$
Given a plane partition $\pi$ we consider its diagonal slices $\lambda^t$ for $t\in \mathbb{Z}$, i.e. the sequences
$$\lambda^t = (\pi_{i, i + t}) \hspace{3mm} \mbox{ for } i \geq \max(0, -t).$$
One readily observes that $\lambda^t$ are partitions and satisfy the following interlacing property
$$\cdots \prec \lambda^{-2} \prec \lambda^{-1} \prec \lambda^0 \succ \lambda^1 \succ \lambda^2 \succ \cdots.$$
Conversely, any (terminating) sequence of partitions $\lambda^{t}$, satisfying the interlacing property, defines a partition $\pi$ in the obvious way. Concepts related to plane partitions are illustrated in Figure \ref{S2_2}.\\
\begin{figure}[h]
\centering
\scalebox{0.45}{\includegraphics{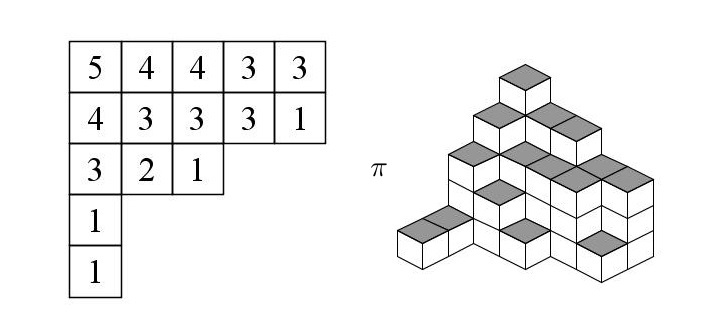}}
\caption{The plane partition $\pi = \varnothing \prec (1) \prec (1) \prec (3) \prec (4,2) \prec (5,3,1) \succ (4,3) \succ (4,3) \succ (3,1) \succ(3) \succ \varnothing$ . The volume $|\pi| = 41$}
\label{S2_2}
\end{figure}

Fix $t \in (0,1)$. For partitions $\mu, \lambda \in \mathbb{Y}$ we let $P_{\lambda/ \mu}$ and $Q_{\lambda/ \mu}$ denote the (skew) Hall-Littlewood symmetric functions with parameter $t$ (see e.g. Chapter 3 in \cite{Mac} for the definition and properties of these functions). Let us fix $M,N \in \mathbb{N}$ and suppose $(\mathcal{L}, \mathcal{M})$ is a pair of sequences of partitions $\mathcal{L} = \{ \lambda^k\}_{k = -M+1}^{N-1}$ and $\mathcal{M} = \{ \mu^k\}_{k = -M+1}^{N-2}$. Define the weight of such a pair as
\begin{equation}\label{shurweight}
\mathcal{W}(\mathcal{L}, \mathcal{M}) := \prod_{n = -M+1}^{N-1} P_{\lambda^n / \mu^{n-1}}(x_n)  Q_{\lambda^n / \mu^{n}}(y_n),
\end{equation}
where $x_i, y_i \in [0,1]$ for all $i \in \{-M+1,...,N-1\}$ and we have $\mu^{-M} = \mu^{N-1} = \varnothing$.  From $(5.8)$ and $(5.8')$ in Chapter III of \cite{Mac} we have
\begin{equation}\label{HLM}
\begin{split}
P_{\lambda/ \mu}(x;t) =  \psi_{\lambda/\mu}(t)x^{|\lambda| - |\mu|} &\mbox{ and } Q_{\lambda/ \mu}(x;t) = \phi_{\lambda / \mu}(t) x^{|\lambda| - |\mu|}, \mbox{ where }\\
 \psi_{\lambda/\mu}(t) =  {\bf 1}_{ \lambda \succ \mu}  \prod_{j \in J}(1 - t^{m_j(\mu)}) \hspace{3mm} &\mbox{ and }\phi_{\lambda / \mu}(t) ={\bf 1}_{ \lambda \succ \mu} \prod_{i \in I}(1 - t^{m_i(\lambda)});\\
I = \{ i \in \mathbb{N}: \lambda'_{i+1} = \mu'_{i+1} \mbox{ and }\lambda'_{i} > \mu'_{i}\} &\mbox{ and }J = \{ j \in \mathbb{N}: \lambda'_{j+1} > \mu'_{j+1} \mbox{ and }\lambda'_{j} = \mu'_{j}\}.
\end{split}
\end{equation}
Observe that the weights are non-negative (as $t \in (0,1)$) and provided $Z: = \sum_{\mathcal{L}, \mathcal{M}} \mathcal{W}(\mathcal{L}, \mathcal{M})$, is finite we have that $\mathbb{P}(\mathcal{L}, \mathcal{M}): =Z^{-1} \cdot \mathcal{W}(\mathcal{L}, \mathcal{M})$ defines probability measure on $(\mathcal{L}, \mathcal{M})$, which we call a {\em Hall-Littlewood process}.
\begin{remark}
One can generalize the approach we described above by considering the Macdonald symmetric functions, instead of the Hall-Littlewood ones, and by considering more general (than single variable) specializations. The resulting object is called the {\em Macdonald process} and was defined and studied in \cite{BorCor}.
\end{remark}

In this paper we will consider the following variable specialization 
\begin{equation}\label{schurspec}
\begin{split}
x_{n+1} = 1,\mbox{ } y_n = 0 \mbox{ if } n \leq -1  ; \mbox{ } x_{n+1} = 0, \mbox{ } y_n = \zeta \mbox{ if } 0 \leq n , \mbox{ where } \zeta \in (0,1) \mbox{ is fixed.}
\end{split}
\end{equation}
Using (\ref{HLM}) and Proposition 2.4 in \cite{BorCor} we conclude that for the above variables we have
\begin{enumerate}[label = \arabic{enumi}., leftmargin=1.5cm]
\item $Z = \left( \frac{1 - t\zeta}{1 - \zeta} \right)^{NM} < \infty$ so that the measure is well defined;
\item $\mu^n = \lambda^n$ for $n < 0$ and $\mu^n = \lambda^{n+1}$ for $n\geq 0$;
\item $\varnothing \prec \lambda^{-M + 1} \cdots \prec\lambda^{-1}  \prec \lambda^0 \succ \lambda^1 \succ \cdots \succ \lambda^{N-1} \succ \varnothing$.
\end{enumerate}
The last statement shows that $\mathcal{L}$ defines a plane partition $\pi$, whose base is contained in an $M \times N$ rectangle (i.e. such that $\pi_{i,j} = 0$ for $i \geq M$ or $j \geq N$). Denoting the set of such plane partitions by $\mathcal{P}(M,N)$ we see that the projection of the Hall-Littlewood process on $\mathcal{L}$ induces a measure on $\mathcal{P}(M,N)$. 

Substituting $P_{\lambda/ \mu}(x)$ and $Q_{\lambda/ \mu}(x)$ from (\ref{HLM}) one arrives at
$$\mathbb{P}(\mathcal{L}) = \left( \frac{1 - \zeta}{1 - t\zeta} \right)^{NM}\cdot  \zeta^{|\lambda^0|} \cdot B_{\mathcal{L}} (t) \mbox{, where } B_{\mathcal{L}} (t) =  \prod_{n = -M+1}^0\psi_{\lambda^n/\lambda^{n-1}}(t) \times \prod_{n = 1}^N \phi_{\lambda^{n-1}/ \lambda^{n}}(t).$$
What is remarkable is that if $\pi$ is the plane partition associated to $\mathcal{L}$, then $ B_{\mathcal{L}}(t) = A_\pi(t)$  from (\ref{Apit}), i.e. $B_{\mathcal{L}}$ admits the geometric interpretation from Figure \ref{S1_2}. The latter is very far from obvious from the definition of $B_{\mathcal{L}}$, since the functions $\phi$ and $\psi$ are somewhat involved, and we refer the reader to \cite{Vul} where this identification was first discovered.\\

The above formulation aimed to reconcile the definition of the Hall-Littlewood process in terms of symmetric functions with the geometric formulation given in Section \ref{sec.HLGLE}. In the remainder of the paper; however, we will be mostly interested in the projection of this measure to the partitions $\lambda^{-M+1},..., \lambda^0$. We perform a shift of the indices by $M$ and denote the latter by $\lambda(1),...,\lambda(M)$. Using results from Section 2.2 in \cite{BorCor} we have the following (equivalent) definition of the measure on these sequences, which we isolate for future reference.
\begin{definition}\label{HLPDef}
 Let $M,N \in \mathbb{N}$ and $\zeta \in(0,1)$.  The {\em homogeneous ascending Hall-Littlewood process} (HAHP) is a probability distribution on sequences of partitions $\varnothing \prec \lambda(1) \prec \lambda(2) \prec \cdots \prec \lambda(M)$ such that 
\begin{equation}\label{HLDefeq}
\mathbb{P}^{M,N}_\zeta (\lambda(1),...,\lambda(M)) = \left( \frac{1 - \zeta}{1 - t \zeta}\right)^{NM} \times \prod_{i = 1}^M P_{\lambda(i)/ \lambda({i-1})}(1) \times Q_{\lambda(M)}(\zeta^N),
\end{equation}
where we use the convention that $\lambda(0) = \varnothing$ is the empty partition and $\zeta^N$ denotes the specialization of $N$ variables to $\zeta$. We also write $\mathbb{E}^{M,N}_{\zeta}$ for the expectation with respect to $\mathbb{P}^{M,N}_{\zeta}$.
\end{definition}

We end this section with an important asymptotic statement for the measures $\mathbb{P}^{M,N}_\zeta $.
\begin{theorem}\label{HLCT} Let $r > 0$, $\zeta, t \in (0,1)$ be given and fix $\mu \in (\zeta, \zeta^{-1})$. Suppose $N,M \in \mathbb{N}$ are sufficiently large so that $\mu N > (r+2) N^{2/3}$ and $M > \mu N + (r+2) N^{2/3}$. Let $\lambda_1'(\cdot)$ be sampled from $\mathbb{P}^{M,N}_{\zeta}$ and set for $x\in [-r-1,r+1]$
\begin{equation}
f^{HL}_N(x) :=  \sigma_\mu^{-1}N^{-1/3}\left(\lambda_1'(\mu N + xN^{2/3}) - f_1(\mu)N - f_1'(\mu) x N^{2/3} -\frac{x^2}{2} f_1''(\mu) N^{1/3}\right),
\end{equation}
where we define $\lambda_1'$ at non-integer points by linear interpolation. The constants above are given by  $ \sigma_\mu = \frac{(\zeta\mu)^{1/6} \left(1 - \sqrt{\zeta\mu} \right)^{2/3} \left( 1 - \sqrt{\zeta/\mu}\right)^{2/3}}{1 - \zeta},$ $f_1(\mu) = 1 - \frac{(1 - \sqrt{\zeta \mu} )^2}{1 -\zeta}, $ $f_1'(\mu) =  \frac{\sqrt{\zeta}(1 - \sqrt{\zeta\mu})}{\sqrt{\mu}(1 - \zeta)},$ $ f_1''(\mu) = \frac{-\sqrt{\zeta}}{2 \mu^{3/2} (1 - \zeta)}.$ Then for any $x \in [-r - 1,r + 1]$ and $y \in \mathbb{R}$ we have
\begin{equation}\label{HLConv}
\lim_{N \rightarrow \infty} \mathbb{P}^{M, N}_{\zeta} \left( f^{HL}_N(x) \leq y\right) = F_{GUE}(y),
\end{equation}
where $F_{GUE}$ is the GUE Tracy-Widom distribution \cite{TWPaper}.
\end{theorem}
\begin{remark}
Owing to the recent work in \cite{Bor16}, the result of Theorem \ref{HLCT} can be established by reduction to the Schur process (corresponding to $t = 0$). For the Schur process a proof of the convergence in (\ref{HLConv}) for the case $s = 0$ can be found in the proof of Theorem 6.1 in \cite{Bor16}. For the sake of completeness we will present a different (more direct) proof below, relying on ideas from \cite{BCG14} and \cite{ED}, which in turn date back to \cite{BorCor}.
\end{remark}
\begin{proof}
Fix $x \in [-r - 1,r + 1]$ and $y \in \mathbb{R}$ throughout. For clarity we split the proof into several steps.\\

{\raggedleft {\bf Step 1.}}
 From Section 2.2 in \cite{BorCor} we know that for $1 \leq K \leq M$ we have
$$ \mathbb{P}_{\zeta}^{M,N}(\lambda(K) = \nu) = \mathbb{P}_{\zeta}^{K,N}(\lambda(K) = \nu) =\left( \frac{1 - \zeta}{1 - t\zeta} \right)^{KN}\hspace{-3mm}P_\nu(1^K)\cdot Q_\nu(\zeta^N),$$
In the last equality we used the homogeneity of $P_\nu$ and $Q_\nu$. Setting $\lambda_1' = \lambda_1'(K)$, we have as a consequence of Proposition 3.5 in \cite{ED} that if $\phi \in \mathbb{C} \backslash \mathbb{R}^+$ then
\begin{equation}\label{finlengthform}
\mathbb{E}_{\zeta}^{M,N} \left[ \frac{1}{(\phi t^{1-\lambda_1'}; t)_\infty} \right] = \det (I + K^{K,N}_\phi)_{L^2(C_\rho)}.
\end{equation}
The contour $C_\rho$ is the positively oriented circle of radius $\rho \in(\zeta t^{-1}, t^{-1})$, centered at $0$, and the operator $K^{K,N}_\phi$ is defined in terms of its integral kernel
$$K^{K,N}_\phi(w,w') = \frac{1}{2\pi \iota} \int_{1/2 - \iota\infty}^{1/2 + \iota\infty} ds \Gamma(-s)\Gamma(1 + s)(-\phi)^s g^{K,N}_{w,w'}(t^s),$$
where  $\Gamma$ is the Euler gamma function and 
$$g^{K,N}_{w,w'}(t^s) = \frac{1}{wt^{-s} - w'}\left( \frac{1 - \zeta(wt)^{-1}}{1 - \zeta(wt)^{-1}t^s}\right)^K \left(\frac{1 - (wt)t^{-s}}{1 - (wt)}\right)^N.$$
We also recall that $(x; t)_\infty = \prod_{i = 1}^\infty ( 1 - x t^{i-1})$ is the $t$-Pochhammer symbol and 
$$\det (I + K^{K,N}_\phi)_{L^2(C_\rho)} = 1 + \sum_{n = 1}^\infty \frac{1}{n!} \int_{C_\rho} \cdots \int_{C_\rho} \det \left[ K^{K,N}_\phi(w_i,w_j) \right]_{i,j = 1}^n \prod_{i = 1}^n dw_i$$
is the Fredholm determinant of the kernel $K^{K,N}_\phi$ (see Section 2 in \cite{ED} for details). \\

{\raggedleft {\bf Step 2.}} For the remainder of the proof we set 
$$K =  \mu N + xN^{2/3}  + O(1) \mbox{ and  }\phi(N) =( -t^{-1}) \times   t^{f_1(\mu)N + f_1'(\mu)xN^{2/3} + (1/2)f_1''(\mu)x^2N^{1/3} + y \sigma_\mu },$$
 where $\sigma_\mu$ and $f(\mu)$ are as in the statement of the theorem.
Our goal in this step is to show that
\begin{equation}\label{Lim1}
\lim_{N \rightarrow \infty}\det (I + K^{K,N}_\phi)_{L^2(C_\rho)} = F_{GUE}(y).
\end{equation}

We use the following change of variables and functional identities
$$w_i \rightarrow \frac{-1}{\tilde w_i},  \hspace{5mm} \rho \rightarrow \tilde\rho^{-1}, \hspace{5mm} \Gamma(-s)\Gamma(1 + s) = \frac{\pi}{\sin(\pi s)}$$
to rewrite
$$\det (I + K^{K,N}_\phi)_{L^2(C_\rho)} = \det (I + \tilde K^{K,N}_{ \phi})_{L^2(C_{\tilde \rho})}.$$
In the above we have that  
\begin{equation}
\begin{split}
&\tilde K^{K,N}_{ \phi}(w,w') = \frac{1}{2 \iota} \int_{1/2 - \iota\infty}^{1/2 + \iota\infty} \frac{t^{-\tilde m_{\nu}K + y \tilde \sigma_\nu K^{1/3}}}{\sin(\pi s)}\frac{g(\tilde w;\zeta, \nu K, K)}{g(t^s \tilde w; \zeta, \nu K, K)} \cdot \frac{ds}{\tilde wt^{s} -\tilde w'}, \mbox{ where }\\
&g(\tilde w; b_1,b_2, x, t) = \left( 1 + zt^{-1}\zeta \right)^x \left(\frac{1}{1 + t^{-1} \tilde z}\right)^{t} \mbox{, } \tilde \sigma_\nu = \frac{\zeta^{1/2} \nu^{-1/6}}{1 - \zeta} \left( (1 - \sqrt{\nu \zeta}) ( \sqrt{\nu  / \zeta} -1)\right)^{2/3}\\
&  \tilde m_{\nu} = \frac{(\sqrt{\nu} - \sqrt{\zeta})^2}{1- \zeta},\mbox{ and }\nu = \mu^{-1} - x \mu^{-5/3} K^{-1/3} + \frac{2x^2}{3} \mu^{-7/3} K^{-2/3} + O(1)   .
\end{split}
\end{equation}
The validity of (\ref{Lim1}) is now equivalent to Proposition 5.3 in \cite{BCG14}. To make the connection clearer we reconcile the notation from equation (65) in that paper with our own below:
$$y \leftrightarrow h, \hspace{3mm} t \leftrightarrow \tau, \hspace{3mm} \tilde \rho \leftrightarrow r, \hspace{3mm} K \leftrightarrow L, \hspace{3mm} b_1 \leftrightarrow \frac{1 - \zeta}{1 - t\zeta}, \hspace{3mm} b_2 \leftrightarrow t \frac{1 - \zeta}{1 - t\zeta}, \hspace{3mm} \zeta \leftrightarrow \kappa.$$
We remark that in \cite{BCG14} the variable $\nu$ is constant, while in our case it changes with $K$ and quickly converges to $\mu^{-1}$ -- this does not affect the validity of Proposition 5.3 and the same arguments can be repeated verbatim.\\

{\raggedleft {\bf Step 3.}} Combining (\ref{finlengthform})  and  (\ref{Lim1}) we see that
\begin{equation}\label{Lim2}
\begin{split}
&\lim_{N \rightarrow \infty} \mathbb{E}_{\zeta}^{M,N} \left[g_N(X_N - y) \right] = F_{GUE}(y), \mbox{ where } g_N(z) =  \frac{1}{( -t^{-N^{1/3} z}; t)_\infty} \mbox{ and } \\
&X_N = \sigma_\mu^{-1}N^{-1/3}(\lambda_1'(\mu N + xN^{2/3}) - f_1(\mu)N - f_1'(\mu)xN^{2/3} - (1/2)f_1''(\mu)x^2N^{1/3}).
\end{split}
\end{equation}
As discussed in the proof of Theorem 5.1 in \cite{BCG14}, we have that $g_N(z)$ satisfy the conditions of Lemma 5.2 of the same paper, which implies that $X_N$ weakly converges to a random variable $X$ such that $\mathbb{P}(X \leq y) = F_{GUE}(y)$. This suffices for the proof.

\end{proof}

\subsection{The stochastic six-vertex model in a quadrant}\label{Section2.2}
In this section we recall the definition of a stochastic inhomogeneous six-vertex model in a quadrant, considered in \cite{Spohn, BCG14,BP16}. There are several (equivalent) ways to define the model and we will, for the most part, adhere to the one presented in Section 1.1.2 in \cite{Agg16B}. We also refer the reader to Section 1 of \cite{BP16} for the definition of a more general higher spin version of this model.

A {\em six-vertex directed path ensemble} is a family of up-right directed paths drawn in the first quadrant $\mathbb{Z}^2_{\geq 1}$ of the square lattice, such that all the paths start from a left-to-right arrow entering each of the points $\{(1,m): m \geq 1 \}$ on the left boundary (no path enters from the bottom boundary) and no two paths share any horizontal or vertical edge (but common vertices are allowed); see Figure \ref{S2_3}. In particular, each vertex has six possible {\em arrow configurations}, presented in Figure \ref{S2_4}.

\begin{figure}[h]
\centering
\scalebox{0.6}{\includegraphics{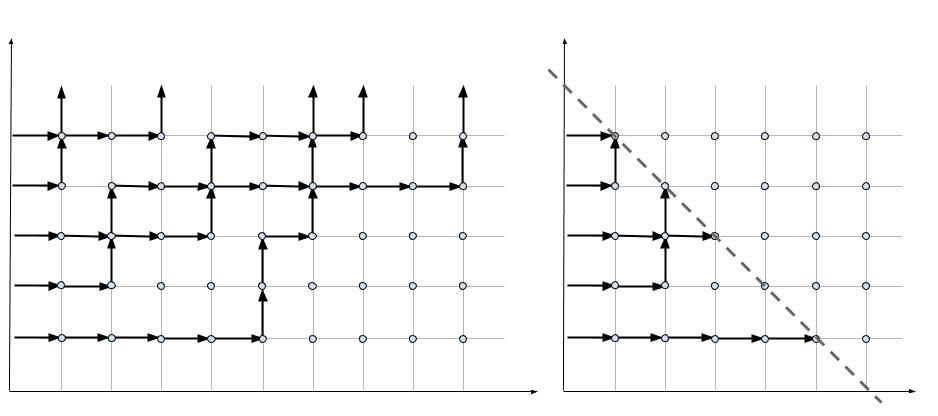}}
\caption{The left picture shows an example of a six-vertex directed path ensemble. The right picture shows an element in $P_n$ for $n = 6$. The vertices on the dashed line belong to $D_n$ and are given half of an arrow configuration if a directed path ensemble from $P_n$ is drawn. Vertices in $D_n$ with zero (two) incoming arrows from the left and bottom can be completed in a unique way - by having zero (two) outgoing arrows. Compare vertices $(4,2)$  in both pictures, also vertices $(1,5)$. Vertices in $D_n$ with a single incoming arrow can be completed by having exactly one outgoing arrow, which can go either to the right or up. Compare vertices $(5,1)$ in both pictures, also vertices $(2,4)$.  }
\label{S2_3}
\end{figure}

\begin{figure}[h]
\centering
\scalebox{0.5}{\includegraphics{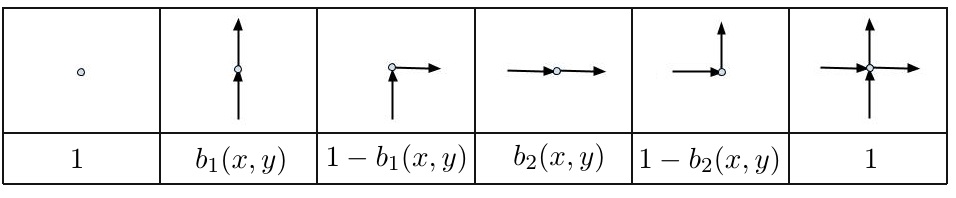}}
\caption{The top row shows the six possible arrow configurations at a vertex $(x,y)$. The bottom row shows the probabilities of top-right completion, given the bottom-left half of a configuration. The probabilities $b_1(x,y)$ and $b_2(x,y)$ depend on $\xi_x, u_y$ and $q$ and are given in (\ref{ProbSV}). }
\label{S2_4}
\end{figure}

The stochastic inhomogeneous six-vertex model is a probability distribution $\mathcal{P}$ on six-vertex directed path ensembles, which depends on a set of parameters $\{\xi_x\}_{x \geq 1}$, $\{ u_y\}_{y \geq 1}$ and $q$, which satisfy
\begin{equation}\label{SVP}
q \in (0,1), \hspace{5mm} \xi_x > 0, u_y > 0, \hspace{5mm} \xi_x u_y > q^{-1/2} \mbox{ for all $x,y \geq 1$}.
\end{equation}
It is defined as the infinite-volume limit of a sequence of probability measures $\mathcal{P}_n$, which are constructed as follows.

For $n \geq 1$ we consider the triangular regions $T_n = \{ (x,y) \in \mathbb{Z}_{\geq 1}^2: x + y \leq n\}$ and let $P_n$ denote the set of six-vertex directed path ensembles whose vertices are all contained in $T_n$. By convention, the set $P_1$ consists of a single empty ensemble. We construct a consistent family of probability distributions $\mathcal{P}_n$ on $P_n$ (in the sense that the restriction of a random element sampled from $\mathcal{P}_{n+1}$ to $T_{n}$ has law $\mathcal{P}_n$) by induction on $n$, starting from $\mathcal{P}_1$, which is just the delta mass at the single element in $P_1$. 

For any integer $n \geq 1$ we define $\mathcal{P}_{n+1}$ from $\mathcal{P}_n$ in the following Markovian way. Start by sampling a directed path ensemble $\mathcal{E}_n$ on $T_n$ according to $\mathcal{P}_n$. This gives arrow configurations (of the type presented in Figure \ref{S2_4}) to all vertices in $T_{n-1}$. In addition, each vertex in $D_n = \{ (x,y) \in \mathbb{Z}^2_{\geq 1}: x+ y = n\}$  is given ``half'' of an arrow configuration, meaning that the arrows entering the vertex from the bottom or left are specified, but not those leaving from the top or right; see the right part of Figure \ref{S2_3}. 

To extend $\mathcal{E}_n$ to a path ensemble on $T_{n+1}$, we must ``complete'' the configurations, i.e. specify the top and right arrows, for the vertices on $D_n$. Any half-configuration at a vertex $(x,y)$ can be completed in at most two ways; selecting between these completions is done independently for each vertex in $D_n$ at random according to the probabilities given in the second row of Figure \ref{S2_4}, where the probabilities $b_1(x,y)$ and $b_2(x,y)$ are defined as
\begin{equation}\label{ProbSV}
b_1(x,y) = \frac{1 - q^{1/2} \xi_x u_y}{1 - q^{-1/2} \xi_x u_y} \hspace{5mm} b_2(x,y) = \frac{q^{-1} - q^{-1/2} \xi_x u_y}{1 - q^{-1/2} \xi_x u_y}.
\end{equation}
In this way we obtain a random ensemble $\mathcal{E}_{n+1}$ in $P_{n+1}$ and we denote its law by $\mathcal{P}_{n+1}$. One readily verifies that the distributions $\mathcal{P}_n$ are consistent and then we define $\mathcal{P} = \lim_{n \rightarrow \infty} \mathcal{P}_n$. 

A particular case that will be of interest to us is setting $\xi_x = \xi$ and $u_y = u$ for all $x \geq 1$ and $y\geq 1$, where $\xi, u > 0$ are such that $\xi u > q^{-1/2}$. We refer to this model as the {\em homogeneous stochastic six-vertex model} and denote the corresponding measure as $\mathbb{P}_{\xi, u, q}$. Let us remark that (upto a reflection with respect to the diagonal $x = y$) this model was investigated in \cite{Gwa} and much more recently in \cite{BCG14} under the name ``stochastic six-vertex model''. \\

Given a six-vertex directed path ensemble on $\mathbb{Z}_{\geq 1}^2$, we define the {\em height function } $h(x,y)$ as the number of up-right paths, which intersect the 
horizontal line through $y$ at or to the right of $x$. We end this section by recalling the following important connection between the height function of the homogeneous stochastic six-vertex model and the homogeneous ascending Hall-Littewood process. The following result is a special case of Theorem 4.1 in \cite{BBW} and plays a central role in our arguments.
\begin{theorem}[Theorem 4.1 in \cite{BBW}]\label{SVHL}
Let $\xi,u,q > 0$ be given such that $q \in (0,1)$, $\zeta = \xi^{-1} u ^{-1} q^{-1/2}< 1$ and fix $\mu \in (\zeta, \zeta^{-1})$. Let $h(x,y)$ denote height function sampled from $\mathbb{P}_{\xi, u ,q}$ and $\varnothing \prec \lambda(1) \prec \cdots \prec \lambda(M)$ be distributed as $\mathbb{P}^{M,N}_\zeta$ from Definition \ref{HLPDef}, where $t = q$. Then we have the following equality in distribution of random vectors
$$ \left(N - \lambda_1'(0),...,N - \lambda'_1(M) \right) \eqd \left(h(1,N), \cdots, h(M+1,N) \right), \mbox{ where by convention $\lambda_1'(0) = 0$.}$$
\end{theorem}

\subsection{The asymmetric simple exclusion process}\label{Section2.3} The {\em asymmetric simple exclusion process} (ASEP) is a continuous time Markov process, which was introduced in the mathematical community by Spitzer in \cite{Spi}. In this paper we consider ASEP started from the so-called step initial condition, which can be described as follows. Particles are initially (at time $0$) placed on $\mathbb{Z}$ so that there is a particle at each location in $\mathbb{Z}_{\leq 0}$ and all positions in $\mathbb{Z}_{\geq 1}$ are vacant. There are two exponential clocks, one with rate $L$ and one with rate $R$, associated to each particle; we assume that $R > L \geq 0$ and that all clocks are independent. When some particle's left clock rings, it attempts to jump to the left by one; similarly when its right clock rings, it attempts to jump to the right by one. If the adjacent site in the direction of the jump is unoccupied, the jump is performed; otherwise it is not. For a more careful description of the model, as well as a proper definition of this dynamics with infinitely many particles, we refer the reader to \cite{Lig2}.

Given a particle configuration on $\mathbb{Z}$, we define the {\em height function } $\mathfrak{h}(x)$ as the number of particles at or to the right of the position $x$, when $x \in \mathbb{Z}$. For non-integral $x$, we define $\mathfrak{h}(x)$ by linear interpolation of $\mathfrak{h}(\lfloor x \rfloor)$ and $\mathfrak{h}(\lceil x \rceil)$. For $R > L \geq 0$ and $T \geq 0$ we denote by $\mathbb{P}^T_{L,R}$ the law of the height function $\mathfrak{h}$ of the random particle configuration sampled from the ASEP (started from the step initial condition) with parameters $R$ and $L$ after time $T$. 

We isolate the following one-point convergence result for future use.
\begin{theorem}\label{ASEPCT}Suppose $r > 0$, $R = 1$, $L \in (0,1)$, $\gamma = R- L$ and fix $\alpha \in (0, 1)$.  Let $\mathfrak{h}(x)$ denote height function sampled from $\mathbb{P}^{N/\gamma}_{L,R}$ and for $s \in [-r-1,r+1]$ set
\begin{equation}
f^{ASEP}_N(s) =  \sigma_\alpha^{-1}N^{-1/3} \left(f_3(\alpha)N + f'_3(\alpha)s N^{2/3} + (1/2) s^2 f_3''(\alpha)N^{1/3} - \mathfrak{h} \left( \alpha N + sN^{2/3}\right) \right),
\end{equation}
where we define $\mathfrak{h}(\cdot)$ at non-integer points by linear interpolation. The constants above are given by $\sigma_\alpha = 2^{-4/3}(1 - \alpha^2)^{2/3}$, $f_3(\alpha) = \frac{(1-\alpha)^2}{4}$, $f'_3(\alpha) = -\frac{1 - \alpha}{2}$, $f''_3(\alpha) = \frac{1}{2}$. Then for any $s \in [-r - 1,r + 1]$ and $y \in \mathbb{R}$ we have
\begin{equation}\label{ASEPConv}
\lim_{N \rightarrow \infty} \mathbb{P}^{N/\gamma}_{L,R} \left( f^{ASEP}_N(s) \leq y\right) = F_{GUE}(y),
\end{equation}
where $F_{GUE}$ is the GUE Tracy-Widom distribution \cite{TWPaper}.
\end{theorem}
\begin{proof}
The above result follows immediately from the celebrated theorem of Tracy-Widom  \cite[Theorem 3]{TWASEP}, which says that
\begin{equation}\label{TWeq} 
\lim_{T \rightarrow \infty} \mathbb{P}^{T/\gamma}_{L,R} \left( \frac{ c_1 T - x_m}{c_2 T^{1/3}} \leq y\right) = F_{GUE}(y),
\end{equation}
where $\sigma = \frac{m}{T} \in (0,1), \hspace{2mm} c_1 = 1 - 2\sqrt{\sigma}, \hspace{1mm} c_2 = \sigma^{-1/6}(1 - \sqrt{\sigma})^{2/3}.$ In the above relation $x_m$ denotes the position of the $m$-th right-most ASEP particle (notice there is a sign change with the result in \cite{TWASEP}, due to the fact that in that paper $L > R$ and the particles initially occupy the positive integers). Below we briefly explain why the above statement implies the theorem.

One observes that at each fixed time and for any positive integers $m,n$ we have the equality of events  $\{  \mathfrak{h}(n) \geq m  \} = \{ x_m \geq n \},$ which implies that for any $\tau \geq 0$ we have
\begin{equation}\label{ALim1}
 \mathbb{P}^{\tau}_{L,R} \left(  \mathfrak{h}(n)\geq m \right) =  \mathbb{P}^{\tau}_{L,R} \left(  x_m\geq n \right)
\end{equation}
Let $m(N) = \lfloor f_3(\alpha)N + f'_3(\alpha)s N^{2/3} + (1/2) s^2 f_3''(\alpha)N^{1/3} - y N^{1/3} \sigma_\alpha \rfloor$, $n(N) = \lfloor \alpha N + sN^{2/3} \rfloor$ and observe that
\begin{equation}\label{ALim2}
\mathbb{P}^{N/\gamma}_{L,R}\left( \mathfrak{h}(n(N) ) \geq  m(N) + 1 \right) \leq  \mathbb{P}^{N/\gamma}_{L,R}\left( f^{ASEP}_N(s) \leq y \right) \leq\mathbb{P}^{N/\gamma}_{L,R}\left( \mathfrak{h}(n(N) + 1) \geq  m(N) \right).
\end{equation}
 From  (\ref{ALim1}) we have
\begin{equation}\label{ALim3}
\mathbb{P}^{N/\gamma}_{L,R}\left( \mathfrak{h}(n(N) + 1) \geq  m(N) \right) = \mathbb{P}^{N/\gamma}_{L,R} \left( x_{m(N)} \geq c_1N -c_2 N^{1/3} + O(1)\right),
\end{equation}
where we also used that $\sigma = \frac{m}{N} = f_3(\alpha) + f_3'(\alpha)sN^{-1/3} + (1/2)s^2 N^{-2/3}f_3''(\alpha) - y\sigma_\alpha N^{-2/3} + O(N^{-1})$. One similarly obtains
\begin{equation}\label{ALim4}
\mathbb{P}^{N/\gamma}_{L,R}\left( \mathfrak{h}(n(N) ) \geq  m(N) + 1 \right) = \mathbb{P}^{N/\gamma}_{L,R} \left( x_{m(N)+1} \geq c_1N -c_2 N^{1/3} + O(1)\right).
\end{equation}
The right sides in (\ref{ALim3}) and (\ref{ALim4})  converge to $F_{GUE}(y)$ from (\ref{TWeq}) as $N \rightarrow \infty$, which together with (\ref{ALim2}) proves the theorem. 
\end{proof}

We end this section by recalling the following important connection between the height function of the homogeneous stochastic six-vertex model and the height function of the ASEP started from step initial condition. This connection was observed in \cite{Gwa,BCG14} and carefully proved for general initial conditions in \cite{Agg16}.
\begin{theorem}[Theorem 1 in \cite{Agg16}]\label{SVASEP}
Let $\xi(N),u(N),q > 0$ be given such that $q \in (0,1)$, $\zeta(N) = \xi(N)^{-1} u(N)^{-1} q^{-1/2}< 1$ and
$$b_1(N) = \frac{1 - q^{1/2} \xi(N) u(N)}{1 - q^{-1/2} \xi(N) u(N)} = qN^{-1} + O(N^{-2}) \hspace{3mm} b_2(N) = \frac{q^{-1} - q^{-1/2} \xi(N) u(N)}{1 - q^{-1/2} \xi(N) u(N)} =  N^{-1}+ O(N^{-2}).$$
In addition, fix $K \in \mathbb{N}$, $T > 0$ and set $N_T = \lfloor N \cdot T\rfloor$. Let $h^N(x,y)$ denote height function sampled from $\mathbb{P}_{\xi(N), u(N) ,q}$ and $\mathfrak{h}$ have law $\mathbb{P}_{L,R}^T$, where $R = 1$ and $L = q$. Then we have the following convergence in distribution of random vectors
$$ (h^N(N_T-K+1,N_T),..., h^N(N_T+K+1,N_T)) \implies (\mathfrak{h}(-K+1), ..., \mathfrak{h}(K+1)) \mbox{ as $N \rightarrow \infty$}.$$
\end{theorem}

\section{Definitions, notations and main results}\label{Section3}
In this section we introduce the necessary definitions and notations that will be used in the paper as well as our main technical result -- Theorem \ref{PropTightGood} below. Afterwards we give several applications of Theorem \ref{PropTightGood} to the three models discussed in the previous section. 

\subsection{Discrete line ensembles and the Hall-Littlewood Gibbs property}\label{Section3.1}
In this section we introduce the concept of a discrete line ensemble and the Hall-Littlewood Gibbs property. Subsequently, we state the main result of the paper.
\begin{definition}\label{DefDLE}
Let $N \in \mathbb{N}$, $T_0, T_1 \in \mathbb{Z}$ with $T_0 < T_1$ and denote $\Sigma = \{1,\dots,N\}$, $\llbracket T_0, T_1 \rrbracket = \{T_0, T_0 + 1,\dots,T_1\}$. Consider the set $Y$ of functions $f: \Sigma \times \llbracket T_0, T_1 \rrbracket \rightarrow \mathbb{Z}$ such that $f(j, i+1) - f(j,i) \in \{0, 1\}$ when $j \in \Sigma$ and $i \in\llbracket T_0, T_1 -1 \rrbracket$ and let $\mathcal{D}$ denote the discrete topology on $Y$. We call elements in $Y$ {\em up-right paths}.

A $\Sigma \times \llbracket T_0, T_1 \rrbracket$-{\em indexed (up-right) discrete line ensemble $\mathfrak{L}$ }  is a random variable defined on a probability space $(\Omega, \mathcal{B}, \mathbb{P})$, taking values in $Y$ such that $\mathfrak{L}$ is a $(\mathcal{B}, \mathcal{D})$-measurable function. 
\end{definition}
\begin{remark}
Notice that the definition of an up-right path we use here differs from the one in the six-vertex model. Namely, for the six-vertex model an up-right path is one that moves either to the right or up, while in discrete line ensembles up-right paths move to the right or with slope $1$. This should cause no confusion as it will be clear from context, which paths we mean.
\end{remark}

The way we think of discrete line ensembles is as random collections of up-right paths on the integer lattice, indexed by $\Sigma$ (see Figure \ref{S3_1}). Observe that one can view a path $L$ on $\llbracket T_0, T_1 \rrbracket \times \mathbb{Z}$ as a continuous curve by linearly interpolating the points $(i, L(i))$. This allows us to define $ (\mathfrak{L}(\omega)) (i, s)$ for non-integer $s \in [T_0,T_1]$ and to view discrete line ensembles as line ensembles in the sense of \cite{CorHamA}. In particular, we can think of $L(s), s\in [T_0,T_1]$ as a random variable in $(C[T_0,T_1], \mathcal{C})$  -- the space of continuous functions on $[T_0,T_1]$ with the uniform topology and Borel $\sigma$-algebra $\mathcal{C}$ (see e.g. Chapter 7 in \cite{Bill}).
\begin{figure}[h]
\centering
\scalebox{0.40}{\includegraphics{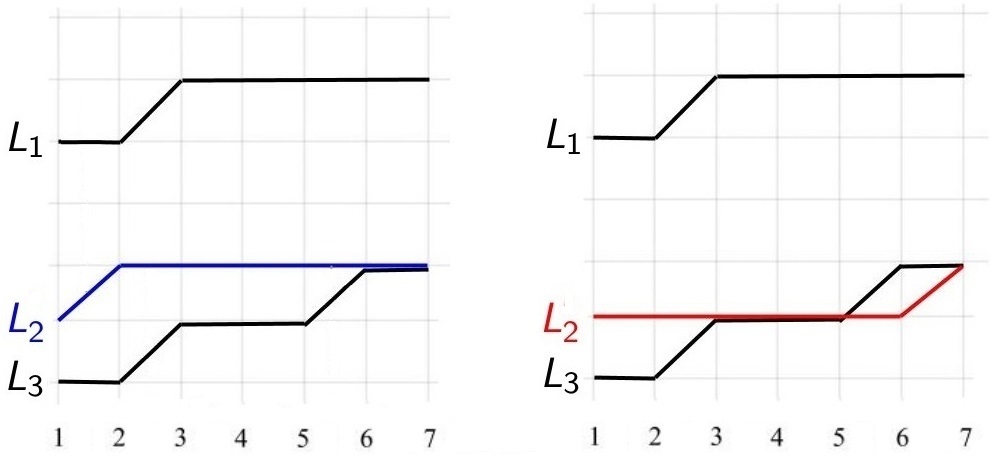}}
\caption{Two samples of $\{1,2,3\}\times \llbracket 1,7 \rrbracket$-indexed discrete line ensembles. }
\label{S3_1}
\end{figure}

We will often slightly abuse notation and write $\mathfrak{L}: \Sigma \times \llbracket T_0, T_1 \rrbracket \rightarrow \mathbb{Z}$, even though it is not $\mathfrak{L}$ which is such a function, but rather $\mathfrak{L}(\omega)$ for each $\omega \in \Omega$. Furthermore we write $L_i = (\mathfrak{L}(\omega)) (i, \cdot)$ for the index $i \in \Sigma$ path. \\

In what follows we fix a parameter $t \in (0,1)$ and make several definitions. Suppose we are given three up-right paths $f,g,L$ on $\llbracket T_0, T_1 \rrbracket \times \mathbb{Z}$. Given a (finite) subset $S \subset \llbracket T_0 +1, T_1 \rrbracket$ we define the following weight function
\begin{equation}\label{S1eq0}
\begin{split}
W_t(T_0, T_1, L, f,g; S) =& \prod_{i \in S}\left(1 -  {\bf 1}_{\{\Delta^+(i-1) - \Delta^+(i) = 1\}}\cdot t^{\Delta^+(i-1)} \right) \times \\  &\prod_{i \in S}\left(1 -  {\bf 1}_{\{\Delta^-(i-1) - \Delta^-(i) = 1\}}\cdot t^{\Delta^-(i-1)} \right) ,
\end{split}
\end{equation}
if $f(i) \geq L(i) \geq g(i)$ for $i \in S$ and $0$ otherwise. In the above $\Delta^+(s) = f(s) - L(s)$ and $\Delta^-(s) = L(s) -g(s)$. In words (\ref{S1eq0}) means that we follow the paths $f,g,L$ from left to right and any time $f - L$ (resp. $L - g$) decreases from $\Delta^+$ to $\Delta^+ - 1$ (resp. $\Delta^-$ to $\Delta^- - 1$) at a location in the set $S$ we multiply by a factor of $1 - t^{\Delta^+}$ (resp. $1 - t^{\Delta^-}$). Observe that by our assumption on $t$ we have that $W_t \in (0, 1]$ unless $L(i) > f(i)$ or $g(i) > L(i)$ for some $i \in S$, in which case the weight is $0$. Typically $S$ will be a finite union of disjoint intervals (i.e. consecutive integer points). 
\begin{remark}
Observe that (\ref{S1eq0}) makes sense even if $f = \infty$. In the latter case as $t \in (0,1)$ the product on the first line of (\ref{S1eq0}) becomes $1$ -- in fact, this will be the most common way $W_t(T_0, T_1, L, f,g; S)$ will appear in the text.
\end{remark}

{\raggedleft \bf Example.} Take the left sample in Figure \ref{S3_1}. If $S = \{2,\dots,7\}$ then we have $W_t(1,7,L_2, L_1,L_3;S) = (1-t)(1-t^2)(1-t^3)$ and $W_t(1,7,L_1, \infty, L_2;S) =  (1-t^3)$. If $S = \{3,\dots,5\}$ then $W_t(1,7,L_2, L_1,L_3;S) = (1-t^2)$ and $W_t(1,7,L_1, \infty, L_2;S) =  1$. If we take the right sample in  Figure \ref{S3_1} with $S= \{2,\dots,7\}$ then we have $W_t(1,7,L_2, L_1,L_3;S) = 0$ and $W_t(1,7,L_1, \infty, L_2;S) =  (1-t^4)$.
\vspace{2mm}

Let $t_i, z_i \in \mathbb{Z}$ for $i = 1,2$ be given such that $t_1 < t_2$ and $0 \leq z_2 - z_1 \leq t_2 - t_1$. We denote by $\Omega(t_1,t_2;z_1,z_2)$ the collection of up-right paths that start from $(t_1,z_1)$ and end at $(t_2,z_2)$, by $\mathbb{P}_{free}^{t_1,t_2; z_1, z_2}$ the uniform distribution on $\Omega(t_1,t_2;z_1,z_2)$ and write $\mathbb{E}^{t_1,t_2;z_1,z_2}_{free}$ for the expectation with respect to this measure. One thinks of the distribution $\mathbb{P}_{free}^{t_1,t_2; z_1, z_2}$ as the law of a simple random walk with i.i.d. Bernoulli increments with parameter $p \in (0,1)$ that starts from $z_1$ at time $t_1$ and is conditioned to end in $z_2$ at time $t_2$. Notice that by our assumptions on the parameters the state space is non-empty. 

The key definition of this section is the following.
\begin{definition}\label{DefHLGP}
Fix $N \geq 2$, $t \in (0,1)$, two integers $T_0 < T_1$ and set $\Sigma = \{1,\dots,N\}$. Suppose $\mathbb{P}$ is a probability distribution on $\Sigma \times \llbracket T_0, T_1 \rrbracket$-indexed discrete line ensembles $\mathfrak{L} = (L_1,\dots,L_N)$ and adopt the convention $L_0 = \infty$. We say that $\mathbb{P}$ satisfies the {\em Hall-Littlewood Gibbs property} with parameter $t$ for a subset $S \subset \llbracket T_0+1, T_1 \rrbracket$ if the following holds. Fix an arbitrary index $i \in \{1,\dots,N-1\}$ and let $\ell_{i-1}, \ell_i, \ell_{i+1}$ be three paths drawn in  $\{ (r,z) \in \mathbb{Z}^2 : T_0 \leq r \leq T_1\}$ such that $\mathbb{P}(L_{i-1} = \ell_{i-1}, L_{i+1} = \ell_{i+1} ) > 0$ (if $i = 1$ we set $\ell_0 = \infty$). Then for any path $\ell$ such that $\ell(T_0) = a = \ell_i(T_0)$ and $\ell(T_1) = b = \ell_i(T_1)$ we have
\begin{equation}\label{S2eq2}
\mathbb{P}(L_i =  \ell |L_i(T_0) = a, L_i(T_1)  = b, L_{i-1} = \ell_{i-1}, L_{i+1} = \ell_{i+1}) = \frac{W_t(T_0, T_1, \ell,\ell_{i-1}, \ell_{i+1}; S)}{Z_t(  T_0, T_1, a,b, \ell_{i-1}, \ell_{i+1};S)},
\end{equation}
where $Z_t(  T_0, T_1, a,b, \ell_{i-1}, \ell_{i+1};S)$ is a normalization constant. We refer to the measure in (\ref{S2eq2}) as $\mathbb{P}_{S}^{T_0, T_1,a,b}( \cdot | \ell_{i-1}, \ell_{i+1})$.
\end{definition}
\begin{remark}\label{RemHLGP}
An equivalent formulation of the above definition is that the law of $L_i$, conditioned on its endpoints $a = L_i(T_0)$ and $b = L_i(T_1)$, $L_{i-1} = \ell_{i-1}$ and $L_{i+1} = \ell_{i+1}$  is given by the Radon-Nikodym derivative
$$ \frac{d \mathbb{P}}{d \mathbb{P}_{free}^{T_0, T_1; a,b}} ( \ell) = \frac{W_t(T_0, T_1,\ell, \ell_{i-1}, \ell_{i+1}; S)}{Z_t(  T_0, T_1, a,b, \ell_{i-1}, \ell_{i+1};S)}.$$
With the above reformulation we get that
$$Z_t(  T_0, T_1, a,b, \ell_{i-1}, \ell_{i+1};S) = \mathbb{E}^{T_0,T_1;a,b}_{free}\left[ W_t(T_0, T_1,\ell, \ell_{i-1}, \ell_{i+1}; S)\right],$$ 
where the expectation is over $\ell$, distributed according to $\mathbb{P}^{T_0,T_1;a,b}_{free}$.

If a measure $\mathbb{P}$ satisfies the Hall-Littlewood Gibbs property, it enjoys the following sampling property. Start by (jointly) sampling $L_i(T_0), L_i(T_1)$ and $L_{j}(r)$ for $j \neq i$ and $r \in \llbracket T_0, T_1 \rrbracket$ according to $\mathbb{P}$ (i.e. according to the restriction of $\mathbb{P}$ to these random variables). Set $a = L_i(T_0)$ and $b = L_i(T_1)$ and let $ L^N_i, N \in \mathbb{N}$  be a sequence of i.i.d. up-right paths distributed according to $\mathbb{P}_{free}^{T_0, T_1; a,b}$. Let $U$ be a uniform random variable on $(0,1)$, which is independent of all else. For each $N \in \mathbb{N}$ we check if $W_t(T_0, T_1,  L^N_i, L_{i-1}, L_{i+1}; S) > U$ and set $Q$ to be the minimal index $N$ for which the inequality holds.  Observe that $Q$ is a geometric random variable with parameter $Z_t(  T_0, T_1, a,b, L_{i-1}, L_{i+1};S)$, which we call the {\em acceptance probability}. In view of the above Radon-Nikodym derivative formulation, it is clear that the random ensemble of up-right paths $(L_1,\dots,L_{i-1}, L^{Q}_i, L_{i+1}, \dots, L_N)$ is distributed according to $\mathbb{P}$. 
\end{remark}

\begin{remark} \label{restrict} We mention that the resampling property of Remark \ref{RemHLGP} for a $\{1,\dots, N\}\times \llbracket T_0, T_1 \rrbracket$-indexed line ensemble $\{L_i\}_{i = 1}^N$ only holds for the first $N-1$ lines. The latter, in particular, implies that for $M \leq N$, we have that the induced law on $\{L_i\}_{i = 1}^M$ also satisfies the Hall-Littlewood Gibbs property with parameter $t$ and subset $S$ as an $\{1,\dots,M\}\times \llbracket T_0, T_1 \rrbracket$-indexed line ensemble.
\end{remark}

In this paper, we will be primarily concerned with the case when $\Sigma = \{1, 2\}$ and the discrete line ensemble is {\em non-crossing}, meaning that $L_1(r) \geq L_2(r)$ for all $r \in  
\llbracket T_0, T_1 \rrbracket$. For brevity we will call $\{1,2\}\times\llbracket T_0, T_1 \rrbracket$-indexed non-crossing discrete line ensembles {\em simple}. These line ensembles will typically arise by restricting a discrete line ensemble with many lines to the top two lines. If the original line ensemble satisfies a Hall-Littlewood Gibbs property with parameter $t$ and set $S$, the same will be true for the restriction to the simple line ensemble at the top (see Remark \ref{restrict}). To simplify notation, whenever we are working with a simple discrete line ensemble we will omit the $i-1$ index from all of the earlier formulas and notation, as $L_0,\ell_0$ are deterministically $\infty$.

In the remainder of this section we describe a general framework that can be used to prove tightness for the top curve of a sequence of simple discrete line ensembles. We start with the following useful definition.
\begin{definition}\label{Def1} Fix $t \in (0,1)$, $\alpha > 0$, $p \in (0,1)$ and $T > 0$.  Suppose we are given a sequence $\{ T_N \}_{N = 1}^\infty$ with $T_N \in \mathbb{N}$ and that $\{\mathfrak{L}^N\}_{N = 1}^\infty$, $\mathfrak{L}^N = (L^N_1, L^N_2)$ is a sequence of simple discrete line ensembles on $ \llbracket -T_N, T_N \rrbracket$. We call the sequence $(\alpha,p,T)$-{\em good} if there exists $N_0(\alpha,p,T)$ such that for $N \geq N_0$ we have
\begin{itemize}
\item $T_N > T N^{\alpha}$ and $\mathfrak{L}^N$ satisfies the Hall-Littlewood Gibbs property with parameter $t$ for  $S = \llbracket -T_N + 1, T_N \rrbracket$;  
\item  for each $s \in [-T, T]$ the sequence of random variables $\{ N^{-\alpha/2}(L_1^N(sN^{\alpha}) - p s N^{\alpha})  \}$ is tight (i.e. we have one-point tightness of the top curves).
\end{itemize}
\end{definition}

The main technical result of the paper is as follows.
\begin{theorem}\label{PropTightGood}
Fix $\alpha, r > 0$ and $p \in (0,1)$ and let $\mathfrak{L}^N = (L^N_1, L^N_2)$ be an $(\alpha, p, r+1)$-good sequence.  For $N \geq N_0(\alpha, p, r + 1)$ (as in Definition \ref{Def1}) set 
$$f_N(s) =  N^{-\alpha/2}(L_1^N(sN^{\alpha}) - p s N^{\alpha}), \mbox{ for $s\in [-r,r]$}$$
 and denote by $\mathbb{P}_N$ the law of $f_N(s)$ as a random variable in $(C[-r,r], \mathcal{C})$. Then the sequence $\mathbb{P}_N$ is tight.
\end{theorem}

Roughly, Theorem \ref{PropTightGood} states that if a process can be viewed as the top curve of a simple discrete line ensemble and under some shift and scaling the process's one-point marginals are tight, then under the same shift and scaling the trajectory of the process is tight in the space of continuous curves. We will show later in Theorem \ref{ACBB} that any subsequential limit of the measures $\mathbb{P}_N$ in Theorem \ref{PropTightGood} is absolutely continuous with respect to a Brownian bridge of a certain variance -- see Section \ref{Section7} for the details. We also want to remark that both Theorem \ref{PropTightGood} and Theorem \ref{ACBB} do not depend strongly on any particular structure of the Hall-Littlewood Gibbs property. Indeed, the main ingredient that is used in deriving these results is a lower bound on the acceptance probability $Z_t(  T_0, T_1, a,b, L_2;S)$ (see Remark \ref{RemHLGP}), which is the content of Proposition \ref{PropMain}. It is our belief that our arguments can be extended to other (similar) discrete Gibbs properties without significant modifications. \\

\subsection{Applications to the three models}\label{Section3.2}
In this section we use Theorem \ref{PropTightGood} to prove our main results for the three models in Section \ref{Section2}, given in Theorem \ref{HLTight}, Corollary \ref{SVTight} and Theorem \ref{ASEPTight} below. In order to apply Theorem \ref{PropTightGood} we will need to rephrase the ascending Hall-Littlewood process and the ASEP in the language of discrete line ensembles, to which we first turn.

Suppose we are given a sequence $\varnothing = \lambda(0)  \prec \lambda(1)\prec \lambda(2) \prec \cdots \prec \lambda(M)$. The condition $\lambda(i) \prec \lambda(i+1)$ is equivalent to $\lambda_j'(i+1) - \lambda_j'(i) \in \{0,1 \}$ for any  $j \geq 1$. The latter implies that we can view the sequence $\varnothing =  \lambda(0)   \prec \lambda(1)\prec \lambda(2) \prec \cdots \prec \lambda(M)$ as a collection of up-right paths $\{ \lambda'_j(\cdot) \}_{j = 1}^N$ drawn in the sector $\{0,\dots,M\} \times \mathbb{Z}$ (see Figure \ref{S3_3}). In particular, this allows us to interpret the ascending Hall-Littlewood process as a probability distribution of $\{1,\dots,N\} \times \llbracket 0, M \rrbracket$-indexed discrete line ensembles in the sense of Definition \ref{DefDLE}, where $L_j(i) = \lambda'_j(i)$ for $i = 0,\dots,M$ and $j = 1,\dots,N$. 
\begin{figure}[h]
\centering
\scalebox{0.7}{\includegraphics{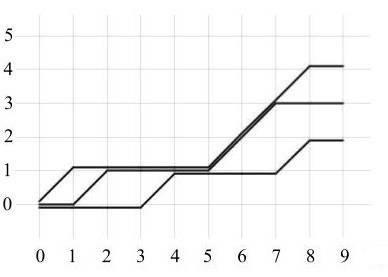}}
\caption{The up-right paths corresponding to $\lambda_1'(i), \lambda_2'(i),\lambda_3'(i)$, for $0 \leq i \leq 9$, where $\lambda(i)$ is the $i$-the element in the sequence
$\varnothing \prec (1) \prec (2) \prec (2) \prec (4) \prec (4,2) \prec (5,2,2) \prec (5,3,2) \prec (8,5,2,1)$.}
\label{S3_3}
\end{figure}

The key observation we make is that if $\varnothing = \lambda(0)  \prec \lambda(1)\prec \lambda(2) \prec \cdots \prec \lambda(M)$ is distributed according to $\mathbb{P}^{M,N}_\zeta$ from Definition \ref{HLPDef}, then the discrete line ensemble $L_j(i) = \lambda'_j(i)$ for $i = 0,\dots,M$ and $j = 1,\dots,N$ satisfies the Hall-Littlewood Gibbs property (this is the origin of the name of this property). We isolate this in the following proposition.
\begin{proposition}\label{HLGPP}
Fix $M,N \in \mathbb{N}$ and $\zeta, t \in (0,1)$. Let $\varnothing = \lambda(0)  \prec \lambda(1)\prec \lambda(2) \prec \cdots \prec \lambda(M)$ be  sampled from $\mathbb{P}^{M,N}_{\zeta}$ (see Definition \ref{HLPDef}). Then $\left(\lambda_1'(\cdot), \lambda_2'(\cdot),\dots,\lambda_N'(\cdot)\right)$ satisfies the Hall-Littlewood Gibbs property with parameter $t$ for $S = \llbracket 1,\dots,M \rrbracket$. 
\end{proposition}
\begin{proof}
By Definition \ref{HLPDef} we know that 
\begin{equation*}
\mathbb{P}^{M,N}_{\zeta} (\lambda(1),\dots,\lambda(M)) = \left( \frac{1 - \zeta}{1 - t \zeta}\right)^{NM} \times \prod_{i = 1}^M P_{\lambda(i)/ \lambda({i-1})}(1) \times Q_{\lambda(M)}(\zeta^N).
\end{equation*}
The latter equation implies that $\lambda_1'(0) = 0$ and $0 \leq \lambda_1'(M) \leq \min (M,N)$ with probability $1$. Using (\ref{HLM}) we see that 
\begin{equation}\label{EqO1}
\begin{split}
&\mathbb{P}^{M,N}_{\zeta} (\lambda(1),\dots,\lambda(M)) = Q_{\lambda(M)}(\zeta^N)  \cdot \left( \frac{1 - \zeta}{1 - t \zeta}\right)^{NM}\cdot  \prod_{i = 1}^M\psi_{\lambda(i)/\lambda(i-1)}(t), \mbox{ where } \\
&  \psi_{\lambda/\mu}(t) = {\bf 1}_{\{ \lambda \succ \mu \} } \cdot \prod_{j= 1}^\infty \left(1 - {\bf 1}_{\{\Delta(\mu,j) - \Delta(\lambda,j) = 1\}} t^{\Delta(\mu,j)} \right), \mbox{ and $\Delta(\mu,j) = \mu_j' - \mu_{j+1}'$}.
\end{split}
\end{equation}

Fix $i \in \{1,\dots,N-1\}$ and notice that (\ref{EqO1}) and (\ref{S1eq0}) imply that for any $\ell \in \Omega(0, M; 0, k)$ with $0 \leq k \leq \min(M, N)$ we have
$$\mathbb{P}^{M,N}_{\zeta}\left( \lambda_i'(\cdot) = \ell  | \mathcal{F}_{ext}\left( \{i\}, (0,M) \right) \right) = C \cdot  W_t(0, M, \ell, \lambda_{i-1}'(\cdot), \lambda_{i+1}'(\cdot);S),$$
where  $\mathcal{F}_{ext}\left( \{i\}, (0,M) \right) $ is the $\sigma$-algebra generated by $\lambda_j'(a)$ for $a= 0,\dots,M$ and $j \neq i$ as well as $\lambda_i'(0) = 0$ and $\lambda'_i(M) = k$, and $C$ is an $\mathcal{F}_{ext}\left( \{i\}, (0,M) \right) $-measurable normalization constant. Let $\mathcal{F}_1$ be the $\sigma$-algebra generated by $\lambda_i'(0) = 0$, $\lambda'_i(M) = k$, $\lambda_{i-1}'(\cdot)$ and $\lambda_{i+1}'(\cdot)$  and observe that $\mathcal{F}_1 \subset  \mathcal{F}_{ext}\left( \{i\}, (0,M) \right)$. It follows from the tower property for conditional expectation that
$$\mathbb{P}^{M,N}_{\zeta}\left( \lambda_i'(\cdot) = \ell | \lambda'_i(0) = 0, \lambda'_i(M) = k, \lambda_{i-1}'(\cdot),\lambda_{i+1}'(\cdot)  \right)  = \mathbb{P}^{M,N}_{\zeta}\left(C   W_t(0, M, \ell, \lambda_{i-1}'(\cdot), \lambda_{i+1}'(\cdot);S) | \mathcal{F}_1\right) $$
$$ = W_t(0, M, \ell, \lambda_{i-1}'(\cdot), \lambda_{i+1}'(\cdot);S) \cdot \mathbb{P}^{M,N}_{\zeta}\left(C  | \mathcal{F}_1\right),$$
where in the last equality we used that $W_t(0, M, \ell, \lambda_{i-1}'(\cdot), \lambda_{i+1}'(\cdot);S)$ is $\mathcal{F}_1$-measurable. The latter equation is equivalent to (\ref{S2eq2}), which proves the proposition.
\end{proof}

With the help of Proposition \ref{HLGPP} we deduce the following results for the homogeneous ascending Hall-Littlewood process and stochastic six-vertex model.
\begin{theorem}\label{HLTight}
Assume the same notation as in Theorem \ref{HLCT}. If $\mathbb{P}_N$ denotes the law of $f^{HL}_N(\cdot)$ as a random variable in $(C[-r,r], \mathcal{C})$, then the sequence $\mathbb{P}_N$ is tight.
\end{theorem}
\begin{proof}
Consider the $\{1,2\} \times \llbracket -T_N, T_N\rrbracket$-indexed simple discrete line ensemble with $T_N = \lfloor (r+2)N^{2/3} \rfloor$, given by
$$(L^N_1(i),L^N_2(i)) = \left(\lambda_1'(\lfloor \mu N \rfloor + i) -  \lfloor f_1(\mu)N \rfloor , \lambda_2'(\lfloor \mu N \rfloor + i) -  \lfloor f_1(\mu)N \rfloor\right).$$ 
It follows from Proposition \ref{HLGPP} that $(L^N_1,L^N_2) $ is a simple discrete line ensemble, which satisfies the Hall-Littlewood Gibbs property with parameter $t$ for $S = \llbracket -T_N + 1, T_N\rrbracket$. In addition, by Theorem \ref{HLCT} we know that for each $s \in [-r-1, r+1]$ the sequence of random variables $N^{-1/3} \left( L_1^N(sN^{2/3}) - sN^{2/3} f_1'(\mu) \right)$ is tight. The latter statements imply that the sequence $(L^N_1,L^N_2) $  is $(2/3, f_1'(\mu), r+1)$-good. It follows from Theorem \ref{PropTightGood} that if
$$g_N^{HL}(s) = N^{-1/3}\left(\lambda_1'(\lfloor \mu N \rfloor + sN^{2/3}) - \lfloor f_1(\mu)N \rfloor - f'_1(\mu) s N^{2/3} \right ), \mbox{ for $s\in [-r,r]$},$$
then $g_N^{HL}(\cdot)$ form a tight sequence of random variables in $(C[-r,r], \mathcal{C})$. The latter clearly implies the statement of the theorem.
\end{proof}

\begin{corollary}\label{SVTight}
Let $\xi,u,q,r > 0$ be given such that $q \in (0,1)$, $\zeta = \xi^{-1} u ^{-1} q^{-1/2}< 1$ and fix $\mu \in (\zeta, \zeta^{-1})$. Let $h(x,y)$ denote height function sampled from $\mathbb{P}_{\xi, u ,q}$ and set for $s\in [-r,r]$
\begin{equation}
f^{SV}_N(s) =  \sigma_\mu^{-1}N^{-1/3}\left(  f_2(\mu)N + f_2'(\mu) s N^{2/3} +  (1/2)s^2 f_2''(\mu) N^{1/3} -  h(1 + \mu N + sN^{2/3}, N)\right),
\end{equation}
where we define $h(\cdot, N)$ at non-integer points by linear interpolation. The constants above are given by  $ \sigma_\mu = \frac{(\zeta\mu)^{1/6} \left(1 - \sqrt{\zeta\mu} \right)^{2/3} \left( 1 - \sqrt{\zeta/\mu}\right)^{2/3}}{1 - \zeta},$ $f_2(\mu) = \frac{(1 - \sqrt{\zeta \mu} )^2}{1 -\zeta}, $ $f_2'(\mu) =  -\frac{\sqrt{\zeta}(1 - \sqrt{\zeta\mu})}{\sqrt{\mu}(1 - \zeta)},$ $ f_2''(\mu) = \frac{\sqrt{\zeta}}{2 \mu^{3/2} (1 - \zeta)}.$\\
If $\mathbb{P}_N$ denotes the law of $f^{SV}_N(s)$ as a random variable in $(C[-r,r], \mathcal{C})$, then the sequence $\mathbb{P}_N$ is tight.
\end{corollary}
\begin{proof}
 From Theorem \ref{SVHL} we know that the law of $f_N^{HL}$ as in the statement of Theorem \ref{HLTight} is the same as $f_N^{SV}$. The result now follows from Theorem \ref{HLTight}.
\end{proof}

Before we apply Theorem \ref{PropTightGood} to the ASEP, we need to rephrase the latter in the language of discrete line ensembles that satisfy the Hall-Littlewood Gibbs property. We achieve this in the following proposition, whose proof is deferred to the next section.
\begin{proposition}\label{HLASEP}Suppose $R = 1$, $L = t \in (0,1)$ are given, fix $K_1, K_2\in \mathbb{N}$, $T > 0$ and set $\Sigma = \{1,\dots,K_1\}$. Then there exists a probability space, on which a $\Sigma \times \llbracket -K_2,K_2 \rrbracket$-indexed discrete line ensemble $(L_1, L_2,\dots,L_{K_1})$ is defined such that
\begin{itemize}
\item the law of $(L_1, L_2,\dots,L_{K_1})$ satisfies the Hall-Littlewood Gibbs property with parameter $t$ for the set $S = \llbracket-K_2+1,K_2\rrbracket$;
\item the law of $(L_1(-K_2),\dots,L_1(K_2))$ is the same as $(-\mathfrak{h}(-K_2+1), \dots, -\mathfrak{h}(K_2+1))$, viewed as random vectors in $\mathbb{R}^{2K_2+1}$, where $\mathfrak{h}$ has law $\mathbb{P}_{L,R}^T$ (see Section \ref{Section2.3}).
\end{itemize}
\end{proposition}

With the help of Proposition \ref{HLASEP} we deduce the following results for the ASEP.
\begin{theorem}\label{ASEPTight}
Assume the same notation as in Theorem \ref{ASEPCT}. If $\mathbb{P}_N$ denotes the law of $f^{ASEP}_N(s)$ as a random variable in $(C[-r,r], \mathcal{C})$, then the sequence $\mathbb{P}_N$ is tight.
\end{theorem}
\begin{proof}
Consider the $\{1,2\} \times \llbracket -T_N, T_N\rrbracket$-indexed simple discrete line ensemble with $T_N = \lfloor (r+2)N^{2/3} \rfloor$, given by
$$(\tilde L_1^N(i),\tilde L_2^N(i)) = \left(L_1(\lfloor \alpha N \rfloor + i) +  \lfloor f_3(\alpha)N \rfloor , L_2(\lfloor \alpha N \rfloor + i) +  \lfloor f_3(\alpha)N \rfloor\right),$$ 
with $(L_1,L_2)$ defined as in Proposition \ref{HLASEP} with $K_1 = 2$, $K_2 = \alpha N + T_N$ and $T = N/\gamma$.

By construction, we have that $(\tilde L_1^N,\tilde L_2^N) $ satisfies the Hall-Littlewood Gibbs property with parameter $t$ for $S = \llbracket -T_N + 1, T_N\rrbracket$. In addition, by Theorem \ref{ASEPCT} and the fact that $L_1$ has the same law as $-\mathfrak{h}$, we know that for each $s \in [-r-1, r+1]$ the sequence of random variables $N^{-1/3} \left( \tilde L_1^N(sN^{2/3}) + sN^{2/3} f'_3(\alpha) \right)$ is tight. The latter statements imply that the sequence
$(\tilde L_1^N,\tilde L_2^N) $  is $\left(2/3, -f'_3(\alpha), r+1\right)$-good. It follows from Theorem \ref{PropTightGood} that if
$$g_N^{ASEP}(s) = N^{-1/3}\left(L_1(\lfloor \alpha N \rfloor + sN^{2/3}) + \lfloor f_3(\alpha)N \rfloor + f_3'(\alpha) s N^{2/3} \right ), \mbox{ for $s\in [-r,r]$},$$
then $g_N^{ASEP}(\cdot)$ form a tight sequence of random variables in $(C[-r,r], \mathcal{C})$. The latter clearly implies the statement of the theorem.
\end{proof}

\begin{remark} In Corollary \ref{ACC} we show that any subsequential limit of either of the sequences $f_N^{HL}$, $f_N^{SV}$ and $f_N^{ASEP}$ as in the text above, when shifted by an appropriate parabola, is absolutely continuous with respect to a Brownian bridge of appropriate variance. This, in particular, implies that the subsequential limits of these random curves are non-trivial.
\end{remark}

\subsection{Proof of Proposition \ref{HLASEP}}\label{Section3.3} In this section we present the proof of Proposition \ref{HLASEP}, which we split into several steps for clarity. Before we go into the main argument let us briefly outline the main ideas of the proof. We begin by considering a particular sequence of  $\{1,\dots,K_1\} \times \llbracket -K_2,K_2 \rrbracket$-indexed discrete line ensemble $(\Lambda^N_1,\dots, \Lambda^N_{K_1})$. The latter are defined through appropriately truncated and shifted discrete line ensembles associated to ascending Hall-Littlewood processes with parameters $\zeta(N)$ such that $\zeta(N)$ converges to $1$. In Step 1 below we carefully explain the construction of  $(\Lambda^N_1, \dots, \Lambda^N_{K_1})$ and assume that the sequence is tight and that $(\Lambda^N_1(-K_2),\dots,\Lambda^N_1(K_2))$ weakly converges to $(-\mathfrak{h}(-K_2+1), \dots, -\mathfrak{h}(K_2+1))$. Using the tightness assumption we can pick some subsequential limit $(\Lambda^\infty_1,\dots, \Lambda^\infty_{K_1})$ and show it satisfies the conditions of the proposition. The weak convergence of $(\Lambda^N_1(-K_2),\dots,\Lambda^N_1(K_2))$ to $(-\mathfrak{h}(-K_2+1), \dots, -\mathfrak{h}(K_2+1))$ is proved in Step 2 and it relies on Theorems \ref{SVHL} and \ref{SVASEP}. The tightness of $(\Lambda^N_1, \dots, \Lambda^N_{K_1})$ is demonstrated in Steps 3, 4, 5 and 6, by combining the already known tightness of $\Lambda^N_1$ and the Hall-Littlewood Gibbs property.   \\

{\raggedleft \bf Step 1.} For each $N \in \mathbb{N}$ consider the homogeneous ascending Hall-Littlewood process $\mathbb{P}^{ M, N_T}_{\zeta(N)}$ where $N_T = \lfloor N \cdot T \rfloor$, $\zeta(N) = 1 - \frac{1-t}{N}$ and $M = N_T + K$. For $N$ such that $N_T \geq K_1$ we let $(\Lambda_1^{N},\dots, \Lambda_{K_1}^{N})$  be the $\Sigma \times \llbracket -K_2,K_2 \rrbracket$-indexed discrete line ensemble, given by
\begin{equation}\label{ASLE}
\Lambda_j^N(i) = \lambda'_j(i + N_T) - N_T, \mbox{ for } i\in \{-K_2,-K_2+1,\dots,K_2\} \mbox{ and } j \in \{1,\dots,K_1\}
\end{equation}
where $(\lambda'_1(\cdot) ,\dots, \lambda'_{K_1}(\cdot))$ is sampled from $\mathbb{P}^{M, N_T}_{\zeta(N)}$. We isolate the following claims.

{\raggedleft {\bf Claims:}} \begin{itemize}
\item the sequence $(\Lambda_1^N, \dots, \Lambda_{K_1}^N)$ is tight as random vectors in $\mathbb{Z}^{K_1 \cdot (2K_2 + 1)}$
\item  the sequence $(\Lambda^N_1(-K_2),\dots,\Lambda^N_1(K_2))$ weakly converges to $(-\mathfrak{h}(-K_2+1), \dots, -\mathfrak{h}(K_2+1))$ as random vectors in $\mathbb{Z}^{2K_2+1}$ as $N \rightarrow \infty$.
\end{itemize}
The latter statements are proved in the steps below. In what follows we assume their validity and finish the proof of the proposition.

Let $(\Lambda_1^\infty,\dots, \Lambda_{K_1}^\infty)$ be any subsequential limit of $(\Lambda_1^N,\dots, \Lambda_{K_1}^N)$ and assume that $N_k$ is an increasing sequence of integers such that 
\begin{equation}
(\Lambda_1^{N_k},\dots, \Lambda_{K_1}^{N_k}) \implies (\Lambda_1^\infty,\dots,  \Lambda_{K_1}^\infty) \mbox{ as $k \rightarrow \infty$},
\end{equation}
We know that $(\Lambda_1^{N_k},\dots, \Lambda_{K_1}^{N_k}) $ is a $\Sigma \times \llbracket -K_2,K_2 \rrbracket$-indexed discrete line ensemble, which by Proposition \ref{HLGPP} satisfies the Hall-Littlewood Gibbs property with parameter $t$ on $S$ and we conclude that the same is true for $(\Lambda_1^\infty,\dots,  \Lambda_{K_1}^\infty) $. By our earlier assumptions we know that  $(\Lambda^\infty_1(-K_2),\dots,\Lambda^\infty_1(K_2))$ has the same law as $(-\mathfrak{h}(-K_2+1), \dots, -\mathfrak{h}(K_2+1))$ and so $ (\Lambda_1^\infty,\dots,  \Lambda_{K_1}^\infty) $ satisfies the conditions of the proposition.\\

{\raggedleft \bf Step 2.} We show that $(\Lambda^N_1(-K_1),\dots,\Lambda^N_1(K_1))$ weakly converges to $(-\mathfrak{h}(-K_1+1), \dots, -\mathfrak{h}(K_1+1))$. Let us put $q = t$, $\xi(N) = t^{1/2} $ and $u = t^{-1} \zeta^{-1}$. From Theorem \ref{SVHL} we have the following equality in distribution
$$(\Lambda_1^{N}(-K_2),\dots,\Lambda_1^{N}(K_2)) \eqd (- h(N_T-K_2+1,N_T ),\dots, - h(N_T+K_2+1,N_T)),$$
 where $h$ is the height function of a homogeneous stochastic six-vertex model sampled from $\mathbb{P}_{\xi(N), u(N), q}$. From (\ref{ProbSV}) we have the following formulas for the probabilities $b_1(N)$ and $b_2(N)$:
$$
b_1(N) = \frac{1 - q^{1/2} \xi(N) u(N)}{1 - q^{-1/2} \xi(N) u(N)} = tN^{-1} + O(N^{-2}) \hspace{3mm} b_2(N) = \frac{q^{-1} - q^{-1/2} \xi(N) u(N)}{1 - q^{-1/2} \xi(N) u(N)} =  N^{-1}+ O(N^{-2}).$$
As a consequence of Theorem \ref{SVASEP} we have that $(- h(N_T-K_2+1,N_T),\dots, - h(N_T+K_2+1,N_T))$ converges weakly to $(-\mathfrak{h}(-K_2+1), \dots, -\mathfrak{h}(K_2+1))$, where $\mathfrak{h}$ has law $\mathbb{P}_{L,R}^T$. \\

{\raggedleft \bf Step 3.} In this step we show that $(\Lambda_1^N, \dots, \Lambda_{K_1}^N)$ is tight, by showing that $\Lambda_k^N$ is tight for each $k = 1,\dots, K_1$. We proceed by induction on $k$ with base case $k = 1$ being true by Step 2. In what follows assume that $\Lambda^N_1,\dots, \Lambda^N_k$ are tight and want to show that $\Lambda^N_{k+1}$ is also tight. Notice that because $L^N_i(j) - L^N_i(j+1) \in \{0,1\}$ it is enought to show that $\lambda'_{k+1}( N_T) - N_T$ is tight.

Let $\epsilon > 0$ be given. Set $D_N(B) := \{| \lambda'_{k-1}( N_T) - N_T| \geq B\}$. If $k \geq 2$ we have from the tightness of the sequence $\lambda'_{k-1}(N_T) - N_T$ that there exists $B \in \mathbb{N}$ sufficiently large so that
\begin{equation}\label{W0}
\mathbb{P} \left(D^c_N(B)\right) < \epsilon/16.
\end{equation}
By convention, $\lambda_0 = \infty$ and so $D_N(B)$ is a set of full measure and (\ref{W0}) holds even if $k = 1$.

 From the tightness of the sequence $\lambda'_k( N_T) - N_T$, we know that there exists $A \in \mathbb{N}$ sufficiently large so that
\begin{equation}\label{W1}
\mathbb{P} \left(| \lambda'_k( N_T) - N_T| \geq A\right) < \epsilon (1 -t)^B/16 \mbox{ and } 1 \geq (1 - t^A)^{2A} \geq 1/2 .
\end{equation}
We make the following definitions
$$E_N := \{ \lambda'_k(N_T - 2A) - N_T > -4A \} \mbox{ and } F_N := \{ \lambda'_{k+1}(N_T ) - N_T < -8A \}.$$
Let us denote by $\mathcal{F}^k_N = \mathcal{F}_{ext} \left( \{k\} \times (N_T- 2A, N_T]\right)$ the $\sigma$-algebra generated by the up-right paths $\lambda_j'(\cdot)$ for $j \neq k$ and $\lambda_k'(\cdot)$ on the interval $[0,N_T- 2A]$. Observe that all three events $D_N(B)$, $E_N$ and $F_N$ are $\mathcal{F}^k_N$-measurable. Using the above notation we claim that for all $N$ sufficiently large we have
\begin{equation}\label{W5}    
4 \cdot \mathbb{E} \left[ {\bf 1}\{ \lambda'_k(N_T) \leq N_T - A\}| \mathcal{F}^k_N\right] \geq (1-t)^B \cdot {\bf 1}_{D_N \cap E_N \cap F_N}.
\end{equation}
The above statement will be proved in Step 4 below. For now we assume it and finish the proof.

Taking expectations on both sides  of (\ref{W5}) and using (\ref{W1}), we conclude that $\epsilon/4\geq \mathbb{P}(D_N \cap E_N \cap F_N).$
Notice that $E_N \subset \{ 0 \geq  \lambda'_k( N_T) - N_T > -2A \}$, which implies by (\ref{W1}) that $\mathbb{P}(E_N^c) \leq \epsilon/16$.
Combining the last two estimates with (\ref{W0}) we see that for all large $N$ we have
$$\mathbb{P}(F_N) \leq \mathbb{P}(D_N \cap E_N \cap F_N) + \mathbb{P}(E_N^c) + \mathbb{P}(D^c_N(B))  \leq \epsilon/4 + \epsilon/16  +  \epsilon/16< \epsilon.$$
The latter means that for all large $N$ we have
$$\mathbb{P}\left(0 \geq \lambda'_{k+1}(N_T ) - N_T \geq -8A \right) > 1 - \epsilon.$$
Since $\epsilon > 0$ was arbitrary this proves that $\lambda'_{k+1}(N_T ) - N_T$ is tight. \\

{\raggedleft \bf Step 4.} For $t_1, t_2, x \in \mathbb{Z}$ and $t_1 < t_2$ we let $\Omega_x(t_1,t_2)$ denote the set of up-right paths drawn in $\{t_1,\dots,t_2\}\times \mathbb{Z}$, which start from $(t_1,x)$. In addition, we fix two up-right path $\ell_{bot} \in \Omega_y(t_1,t_2)$ and $\ell_{top} \in \Omega_z(t_1,t_2) $, where $y <x - 4A$, $y \leq z$ and $K(\ell_{top}) \leq B$ where $K(\ell_{top}):= |\{ N_T -2A+ 1 \leq i \leq N_T: \ell_{top}(i) - \ell_{top}(i-1) = 0 \}|$. If $k = 1$ we set $\ell_{top} = \infty$ and $K(\ell_{top}) = 0$.

For $N \in \mathbb{N}$ we consider the measure $\mathbb{P}^{x, \ell_{top},\ell_{bot}}_N$ on $\Omega_x(N_T - 2A,N_T)$, given by
$$\mathbb{P}^{x, \ell_{top},\ell_{bot}}_N(\ell) = Z_N^{-1} \cdot W_t(N_T- 2A, N_T, \ell, \ell_{top}, \ell_{bot}; S_N)\cdot  \zeta(N)^{\ell(N_T) - x}, 
$$
$\mbox{where } S_N =  \llbracket N_T - 2A + 1, N_T \rrbracket $ and $Z_N \mbox{ is a normalization constant}.$ 
With the above notation we define $P( x, N, \ell_{top},\ell_{bot}) = \mathbb{P}^{x, \ell_{top}, \ell_{bot}}_N \left( \ell(N_T) \leq x + A\right)$ and claim that for all $N$ sufficiently large (depending on $t$ and $A$) we have that 
\begin{equation}\label{WS1}
P( x, N, \ell_{top},\ell_{bot}) \geq (1-t)^B/4.
\end{equation}
The latter will be proved in Step 5 below. For now we assume its true and finish the proof of (\ref{W5}).

Let $\ell^N_{k\pm 1} \in \Omega_{\lambda_{k\pm 1}'(N_T - 2A)}(N_T - 2A,N_T)$ be such that $\ell^N_{k\pm 1}(i) = \lambda_{k\pm 1}'(i)$ for $i = N_T -2A, \dots, N_T$, where $ \ell_{k-1}^N = \ell_0^N = \infty$ when $k = 1$. As a consequence of Proposition \ref{HLGPP} (see also (\ref{EqO1})) we have the following a.s. equality of $\mathcal{F}^k_N $ random variables 
\begin{equation*}
\begin{split}
&{\bf 1}_{D_N(B) \cap E_N \cap F_N} \cdot \mathbb{E} \left[ {\bf 1}\{ \lambda'_k(N_T)  \leq \lambda_k'(N_T - 2A) + A  \}| \mathcal{F}^k_N\right] = \\&{\bf 1}_{D_N(B) \cap E_N \cap F_N} \cdot P(  \lambda_k'(N_T - 2A) , N,  \ell_{k-1}^N, \ell^N_{k+1}) .
\end{split}
\end{equation*}
In deriving the above equality we used that for $\omega \in D_N(B)$ we have $K(\ell_{k-1}^N(\omega)) \leq B$ by definition of $D_N(B)$.

Notice that a.s. $ \lambda_k'(N_T - 2A) + A  \leq  N_T- A$, from which we conclude that we have the following a.s. inequality
\begin{equation}\label{W4}
\begin{split}
&{\bf 1}_{D_N(B) \cap E_N \cap F_N} \cdot \mathbb{E} \left[ {\bf 1}\{ \lambda'_k(N_T) \leq N_T - A\}| \mathcal{F}^k_N\right] \geq \\
&{\bf 1}_{D_N(B) \cap E_N \cap F_N} \cdot P(  \lambda_k'(N_T - 2A) , N,  \ell_{k-1}^N, \ell^N_{k+1}).
\end{split}
\end{equation}
From (\ref{WS1}) we have for all large $N$ that $ P(  \lambda_k'(N_T - 2A) , N,  \ell_{k-1}^N, \ell^N_{k+1}) \geq (1-t)^B/4$, which together with $1 \geq {\bf 1}_{E_N \cap F_N} $ and (\ref{W4}) imply (\ref{W5}). \\

{\raggedleft \bf Step 5.} In this step we establish (\ref{WS1}), but first we briefly explain our idea. By assumption, we know that $\ell$ is a random path that lies at least a distance $A$ above $\ell_{bot}$ and that $\ell_{top}(i)$ increases by $1$ when $i$ increases by $1$ on $[N_T-2A, N_T]$ with at most $B$ exceptions. The latter implies that
$$1 \geq W_t(N_T- 2A, N_T, \ell, \ell_{top}, \ell_{bot}; S_N) \geq (1-t)^B (1 - t^{A})^{2A} \geq (1-t)^B/2,$$
where in the last inequality we used (\ref{W1}). On the other hand, we know that $\zeta(N) \rightarrow 1$ as $N \rightarrow \infty$. This implies that $\mathbb{P}^{x, \ell_{top},\ell_{bot}}_N(\ell) $ is essentially the uniform measure on up-right paths of length $2A$ started from $x$, conditioned to stay below $\ell_{top}$ and distorted by a well-behaved Radon-Nikodym derivative. At least half of the paths that start from $x$ and have length $2A$ end in a position below $x+A$, and since each path carries roughly the same weight we can obtain the desired estimate.\\

We make the following definitions
\begin{equation*}
\begin{split}
&\Omega^+(\ell_{top}) := \{ \ell \in \Omega_x(N_T - 2A,N_T) : \ell(N_T) > x+ A \mbox{ and } \ell_{top}(i) \geq \ell(i) \mbox{ for } N_T - 2A\leq i \leq N_T\},\\
&\Omega^-(\ell_{top}) := \{ \ell \in \Omega_x(N_T - 2A,N_T) : \ell(N_T) \leq x+ A \mbox{ and } \ell_{top}(i) \geq \ell(i) \mbox{ for } N_T - 2A\leq i \leq N_T\}.
\end{split}
\end{equation*}
We claim that we have
\begin{equation}\label{W8}
|\Omega^-| \geq |\Omega^+|.
\end{equation}
The latter will be proved in Step 6 below. For now we assume it and finish the proof of (\ref{WS1}).

Write $\mathbb{P}_N$ instead of $\mathbb{P}^{x, \ell_{top}, \ell_{bot}}_N$ for brevity. We can find $N_0$ (depending on $t$ and $A$ ) such that for all $N \geq N_0$ we have $1 \geq \zeta(N)^{2A} \geq 1/2 $. The latter together with our assumption on $\ell_{top}$ implies
$$1 \geq W_t(N_T- 2A, N_T, \ell, \ell_{top}, \ell_{bot}; S_N) \geq (1-t)^B(1 - t^{A})^{2A} \geq (1-t)^B/2$$
Consequently, for any $\ell_1, \ell_2 \in \Omega_x(N_T - 2A,N_T)$ we have
$$\mathbb{P}_N(\ell_1) \geq \left[(1-t)^B/2\right]\cdot \mathbb{P}_N(\ell_2) .$$
In view of (\ref{W8}) we have
$$\mathbb{P}_N( \Omega^-) = \sum_{ \ell \in \Omega^-} \mathbb{P}_N(\ell) \geq \left[(1-t)^B/2\right] \cdot \sum_{ \ell \in \Omega^+} \mathbb{P}_N(\ell) = \left[(1-t)^B/2\right] \cdot \mathbb{P}_N( \Omega^+). $$
The latter implies that 
$$\mathbb{P}_N( \Omega^-) \geq (1/2) \cdot \mathbb{P}_N( \Omega^-) + \left[(1-t)^B/4\right]\cdot \mathbb{P}_N( \Omega^+) \geq \left[(1-t)^B/4\right].$$

{\raggedleft \bf Step 6.} In this final step we establish the validity of (\ref{W8}). It is easy to see that (\ref{W8}) is equivalent to the following purely probabilistic question:

Let $X_i$ be i.i.d. random variables such that $\mathbb{P}(X_1 = 0) =  \mathbb{P}(X_1 = 1) = 1/2$ and $S_k = X_1 + \cdots + X_k$ be a random walk with increments $X_i$. Fix an up-right path $\ell_{top}$ such that $\ell_{top}(0) \geq 0$ and $A \in \mathbb{N}$. Then we have the following inequality
\begin{equation}\label{W9}
\mathbb{P}(S_{2A} \leq A| S_k \leq \ell_{top}(k) \mbox{ for } k = 1,\dots, 2A) \geq  1/2.
\end{equation}
Observe that if $\ell_{top} = \infty$ then the above is trivial by symmetry. For finite $\ell_{top}$, conditioning the walk to stay below $ \ell_{top}$ stochastically pushes the walk lower and so the probability it ends up below $A$ only increases.

A rigorous way to prove the above is using the FKG inequality. To be more specific, let $L_{2A}$ be the set of up-right paths starting from $0$ of length $2A$. The natural partial order on $L_{2A}$ is given by 
$$\ell_1 \leq \ell_2  \iff  \ell_1(i) \leq \ell_2(i) \mbox{ for $i = 1,\dots, 2A$}.$$
 With this $L_{2A}$ has the structure of a lattice and so the FKG inequality reads
$$\left(\sum_{\ell \in L_{2A}}\frac{{\bf 1} \{ \ell \leq \ell_{top} \} {\bf 1} \{ \ell(2A) \leq A \}}{|L_{2A}|} \right)  \geq \left(\sum_{\ell \in L_{2A}} \frac{{\bf 1} \{\ell(2A) \leq A\}}{|L_{2A}|}  \right) \cdot \left(\sum_{\ell \in L_{2A}} \frac{{\bf 1} \{\ell \leq \ell_{top}\}}{|L_{2A}|} \right),$$
and clearly implies (\ref{W9}). This concludes the proof of the proposition.

\section{Basic lemmas}\label{Section4}
This section contains the primary set of results we will need to prove Theorem \ref{PropTightGood}. For the remainder of the paper we will only work with simple discrete line ensembles and as discussed in Section \ref{Section3.1} we will drop references to $\ell_0$ and $L_0$ from our notation.

\subsection{Monotone weight lemma}\label{Section4.1}
In this section we isolate the key result that allows us to analyze measures that satisfy the Hall-Littlewood Gibbs property -- Lemma \ref{LemmaMon} below. In addition, we derive two easy corollaries, which are more suitable for our arguments later in the text.

Let  $z_1,z_2, t_1, t_2 \in \mathbb{Z}$ be such that $t_1 < t_2$ and $0 \leq z_2 - z_1 \leq t_2 - t_1$ and recall from Section \ref{Section3.1} that  $\Omega(t_1,t_2; z_1 z_2)$ denotes the set of up-right paths from $(t_1,z_1)$ to $(t_2,z_2)$. Each $\ell \in \Omega(t_1,t_2; z_1 z_2)$ can be encoded by a sequence $R(\ell)$ of $t_2 - t_1$ signs: $+$'s and $-$'s indexed from $t_1 + 1$ to $t_2$, so that $R(i) = +$ if and only if $\ell(i) - \ell(i-1) = 1$. The latter is depicted in Figure \ref{S4_1}. The total number of $+$'s is fixed and equals $z_2 - z_1$ and the number of $-$'s equals $t_2 - t_1 - z_2 + z_1$. 
\begin{figure}[h]
\centering
\scalebox{0.6}{\includegraphics{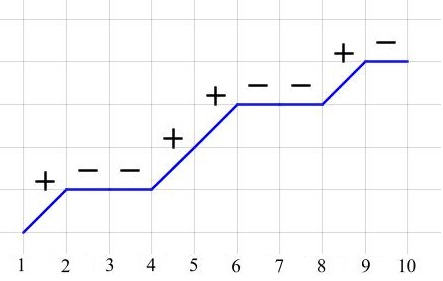}}
\caption{A path identified with a sequence of $+$ and $-$ signs. For the above path we have  $z_2 - z_1 = 4$, $t_2 - t_1 = 9$ and $R(\ell) = (+,-,-,+,+,-,-,+,-)$.  }
\label{S4_1}
\end{figure}

The main result of this section is the following.
\begin{lemma}\label{LemmaMon}
Fix $t \in (0,1)$ and let $c(t) = \prod_{i = 1}^{\infty} ( 1- t^i) \in (0,1)$. Suppose  $a,b,z_1,z_2,t_1,t_2$ are given such that $t_1 + 1 < t_2$, $0 \leq z_2 - z_1 \leq t_2 - t_1$, $0 \leq b - a \leq t_2 - t_1$, $z_1 \leq a$, $z_2 \leq b$. Fix any $\ell_{bot} \in \Omega(t_1,t_2;z_1,z_2)$, $S \subset \{t_1+1,...,t_2\}$ and $T \in \{t_1 + 1,..., t_2 - 1\}$. Let $m(T)$ and $M(T)$ denote the minimal and maximal values of the set $\{ \ell(T): \ell \in  \Omega(t_1,t_2;a,b)\}$ and let $m(T) \leq k_1 \leq k_2 \leq M(T)$. Then we have
\begin{equation}\label{monEq}
c(t)\cdot \mathbb{E}^{t_1,t_2;a,b}_{free}\left[ W_t(t_1, t_2, \ell, \ell_{bot}; S) | \ell(T) = k_1 \right] \leq   \mathbb{E}^{t_1,t_2;a,b}_{free}\left[ W_t(t_1, t_2, \ell, \ell_{bot}; S) | \ell(T) = k_2 \right].
\end{equation}
\end{lemma}
\begin{proof}
For brevity we write $W(\ell)$ for $W_t(t_1, t_2, \ell , \ell_{bot}; S).$ Let $\ell_1$ be a random path sampled according to $\mathbb{P}^{t_1,t_2;a,b}_{free}$, conditioned on $\ell_1(T) = k_1$. We identify this path with a sequence of $+$'s and $-$'s and observe that we have $(k_1 - a)$ $+$'s in the first $T - t_1$ slots and the rest are filled with $-$'s. Similarly, we have exactly $(b - k_2)$ $+$'s in the rest $t_2 - T$ slots. Let us pick uniformly at random $(k_2 - k_1)$ $-$'s in the first $T-t_1$ slots and change them to $+$, and also we pick uniformly at random $(k_2 - k_1)$ $+$'s in the last $t_2 - T$ slots and change them to $-$. In this way we build a new sequence of $+$'s and $-$'s that naturally corresponds to an element $\ell_2 \in \Omega(t_1,t_2;a,b)$ such that $\ell_2(T) = k_2$. Moreover it is clear that the random path $\ell_2$ is distributed according to $\mathbb{P}^{t_1,t_2;a,b}_{free}$, conditioned on $\ell_2(T) = k_2$. We are interested in proving the following statement
\begin{equation}\label{LemmaMon1}
W(\ell_1) \leq c(t)^{-1} \cdot W(\ell_2).
\end{equation}
The statement of the lemma is obtained by taking expectations on both sides of (\ref{LemmaMon1}).\\

Since $W(\ell_1) = 0$ otherwise (and then (\ref{LemmaMon1}) is immediate) we may assume that $\ell_1(i) \geq \ell_{bot}(i)$ for all $i \in S$. Let $r = k_2 - k_1$ and denote by $x_1< x_2 < \cdots < x_r$ and $y_1 > y_2 \cdots > y_r$ the positions of $-$'s and $+$'s respectively that we changed when we transformed $\ell_1$ to $\ell_2$. We also let $\ell^{j}$ for $j = 0,...,r$ denote the paths in $\Omega(t_1,t_2;a,b)$ obtained by flipping only the signs at locations $x_1,...,x_j$ and $y_1,...,y_j$ (in particular $\ell^0 = \ell_1$ and $\ell^r = \ell_2$). An example is depicted in Figure \ref{S4_2}. 

\begin{figure}[h]
\centering
\scalebox{0.8}{\includegraphics{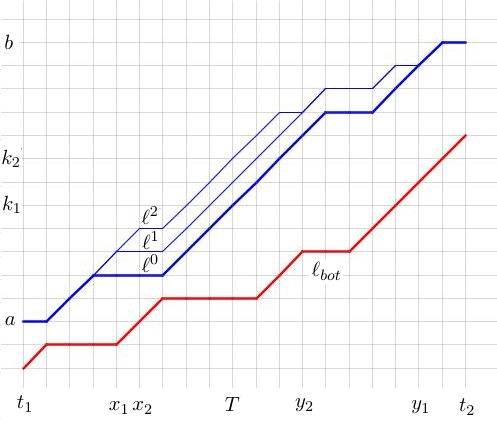}}
\caption{An example of $\ell^0, \ell^1$ and $\ell^2$ for the case $k_2 - k_1 = 2$.  }
\label{S4_2}
\end{figure}

Recall from (\ref{S1eq0}) that $W(\ell)= \prod_{j \in S}\left(1 -  {\bf 1}_{\{\Delta(j-1) - \Delta(j) = 1\}}\cdot t^{\Delta(j-1)} \right) ,\mbox{ where } \Delta(j) =  \ell(j) - \ell_{bot}(j).$ Let us explain how $W(\ell^{j+1})$ differs from $W(\ell^j)$. When we flip the signs at $x_{j+1}$ and $y_{j+1}$, we raise the path $\ell^j$ by $1$ in the interval $[x_{j+1},y_{j+1}-1]$, while outside $(x_{j+1} - 1, y_{j+1})$ it remains the same (see Figure \ref{S4_2}). The latter operation modifies the factors in $W(\ell^j)$ as follows.
\begin{itemize}
\item If $x_{j+1} \in S$ then $W(\ell^j)$ has a factor $\left(1 -  {\bf 1}_{\{\Delta(x_{j+1}) - \Delta(x_{j+1} - 1) = 1\}}\cdot t^{\Delta(x_{j+1}-1)} \right)$, which changes to $1$. 
\item All the factors $\left(1 -  {\bf 1}_{\{\Delta(i) - \Delta(i - 1) = 1\}}\cdot t^{\Delta(i-1)} \right)$ become $\left(1 -  {\bf 1}_{\{\Delta(i) - \Delta(i - 1) = 1\}}\cdot t^{\Delta(i-1) + 1} \right)$ whenever $i \in S \cap [x_{j+1}, y_{j+1}-1]$.
\item If $y_{j+1} \in S$ then $W(\ell^j)$ has a factor $\left(1 -  {\bf 1}_{\{\Delta(y_{j+1}) - \Delta(y_{j+1} - 1) = 1\}}\cdot t^{\Delta(y_{j+1}-1)} \right)$, which becomes $\left(1 -  {\bf 1}_{\{\Delta(y_{j+1}) - \Delta(y_{j+1}- 1) = 0\}}\cdot t^{\Delta(y_{j+1}-1) + 1} \right)$.
\end{itemize}
The first two changes only increase the weight $W(\ell^j)$, while the last can decrease it but at most by a factor $1 - t^{m_{j}}$, where $m_j = 1 + \min_{i \in S \cap [x_{j+1}, y_{j+1}-1 ]} \left[\ell^j(i) - \ell_{bot}(i)\right]$. This implies
$$W(\ell^j) \leq (1- t^{m_j})^{-1} \cdot W(\ell^{j+1}).$$

Notice that $m_0 \geq 1$ since we assumed that $ \ell^0(i) = \ell_1(i) \geq \ell_{bot}(i)$ for $i \in S$. In addition, since at step $j+1$ we raise the path on $[x_{j+1}, y_{j+1}-1 ]$ by $1$ it is clear that $m_{j+1} \geq 1 + m_j$, which implies that $m_j \geq j + 1$ for each $j \geq 0$. We conclude that 
$$W(\ell^0)\leq \prod_{j = 1}^r (1 - t^j)^{-1} \cdot W(\ell^r) \leq c(t)^{-1} \cdot W(\ell^r).$$
As  $\ell^0 = \ell_1$ and $\ell^r = \ell_2$ the above proves (\ref{LemmaMon1}) and hence the lemma.
\end{proof}
\begin{remark}\label{MonCoup}
If $t = 0$ the acceptance probability $W_0(t_1, t_2, \ell, \ell_{bot}; S)$ is equal to $1$ if $\ell$ does not cross $\ell_{bot}$ on the set $S$, and $0$ otherwise. In this case one can use the arguments in the proof of Lemmas 2.6 and 2.7 in \cite{CorHamA} to show that we can construct on the same probability space $\ell'$ and $\ell''$ such that 
$$\mbox{ $\mathbb{P}(\ell' = \ell) = \mathbb{P}^{t_1,t_2;a,b}_{free}(\ell  | \ell(T) = k_1)$, \hspace{5mm} $\mathbb{P}(\ell'' = \ell) = \mathbb{P}^{t_1,t_2;a,b}_{free}(\ell  | \ell(T) = k_2)$}$$
 and $\ell'(j) \leq \ell''(j)$ for $t_1 \leq j \leq t_2$ with probability $1$. The latter statement implies that we have the following almost sure inequality $ W_0(t_1, t_2, \ell', \ell_{bot}; S) \leq W_0(t_1, t_2, \ell'', \ell_{bot}; S)$, which means that higher curves are accepted with higher probability. This statement fits well with the continuous setup in \cite{CorHamA}.

For general $t \in (0,1)$ we no longer have the above inequality almost surely. For example, we can take $t_1 = 0$, $t_2 =2n$, $a = k_1 = 0$, $b = k_2 = n$, $S = \llbracket t_1 + 1, t_2 \rrbracket$, $\ell_{bot} = \ell'$ to be the path that is flat on the interval $[0,n]$ and goes up on $[n,2n]$, while $\ell''$ the path that goes up on $[0,n]$ and is flat on $[n, 2n]$. For this choice one calculates 
$$W_t(t_1, t_2, \ell', \ell_{bot}; S) = 1 > \prod_{i = 1}^n (1 -t^i)  = W_t(t_1, t_2, \ell'', \ell_{bot}; S).$$
Consequently, even though $\ell'$ is below $\ell''$ it is accepted with higher probability and the reason is that the acceptance probability depends not only on the distance between lines but also on their relative slope. In this context, the result of Lemma \ref{LemmaMon} is that the acceptance probability of the top line does increase as it is raised, although only in terms of its expected value and up to a factor of $c(t) = \prod_{i=1}^{\infty} (1-t^i)$. This monotonicity is much weaker than the almost sure monotonicity in the $t = 0$ case, but it turns out to be sufficient for our methods to work.
\end{remark}

Using the above lemma we prove two useful corollaries.
\begin{corollary}\label{CorMon1}
Assume the same notation as in Lemma \ref{LemmaMon}. Suppose $A,B$ are non-empty subsets of $\{m(T), m(T) + 1,..., M(T)\}$, such that $\alpha \geq \beta$ for all $\alpha \in A$ and $\beta \in B$. Then we have
\begin{equation}\label{monEq2}
c(t)\cdot \mathbb{E}^{t_1,t_2;a,b}_{free}\left[ W_t(t_1, t_2, \ell, \ell_{bot}; S) | \ell(T) \in B \right] \leq   \mathbb{E}^{t_1,t_2;a,b}_{free}\left[ W_t(t_1, t_2, \ell, \ell_{bot}; S) | \ell(T) \in A \right].
\end{equation}
\end{corollary}
\begin{proof}
For brevity we write $W(\ell)$ for $W_t(t_1, t_2, \ell , \ell_{bot}; S)$. We have that 
$$c(t) \cdot \mathbb{E}^{t_1,t_2;a,b}_{free}\left[ W(\ell) | \ell(T) \in B \right] = c(t)\cdot\sum_{ \beta \in B} \mathbb{E}^{t_1,t_2;a,b}_{free}\left[ W(\ell)| \ell(T) = \beta \right] \cdot \frac{\mathbb{P}^{t_1,t_2;a,b}_{free}(\ell(T) = \beta)}{\mathbb{P}^{t_1,t_2;a,b}_{free}(\ell(T) \in B)} = $$
$$  c(t)\cdot\sum_{ \beta \in B}  \sum_{ \alpha \in A}  \mathbb{E}^{t_1,t_2;a,b}_{free}\left[ W(\ell) | \ell(T) =\beta \right] \cdot\frac{\mathbb{P}^{t_1,t_2;a,b}_{free}(\ell(T) = \alpha)}{\mathbb{P}^{t_1,t_2;a,b}_{free}(\ell(T) \in A)} \cdot \frac{\mathbb{P}^{t_1,t_2;a,b}_{free}(\ell(T) = \beta)}{\mathbb{P}^{t_1,t_2;a,b}_{free}(\ell(T) \in B)} \leq $$
$$ \sum_{ \beta \in B}  \sum_{ \alpha \in A}  \mathbb{E}^{t_1,t_2;a,b}_{free}\left[ W(\ell) | \ell(T) = \alpha \right] \cdot\frac{\mathbb{P}^{t_1,t_2;a,b}_{free}(\ell(T) = \alpha)}{\mathbb{P}^{t_1,t_2;a,b}_{free}(\ell(T) \in A)} \cdot \frac{\mathbb{P}^{t_1,t_2;a,b}_{free}(\ell(T) = \beta)}{\mathbb{P}^{t_1,t_2;a,b}_{free}(\ell(T) \in B)} = $$
$$   \sum_{ \alpha \in A}  \mathbb{E}^{t_1,t_2;a,b}_{free}\left[ W(\ell) | \ell(T) = \alpha \right] \cdot\frac{\mathbb{P}^{t_1,t_2;a,b}_{free}(\ell(T) = \alpha)}{\mathbb{P}^{t_1,t_2;a,b}_{free}(\ell(T) \in A)}  =  \mathbb{E}^{t_1,t_2;a,b}_{free}\left[ W(\ell)| \ell(T) \in A \right].$$
The middle inequality follows from Lemma \ref{LemmaMon}.
\end{proof}

\begin{corollary}\label{CorMon2}
Assume the same notation as in Lemma \ref{LemmaMon} and let $\alpha \leq M(T)$. Denote by $\mathbb{P}$ the probability distribution $\mathbb{P}_S^{t_1,t_2,a,b}(\cdot | \ell_{bot})$ from Definition \ref{DefHLGP}.
Then we have
\begin{equation}\label{monEq3}
\mathbb{P}(\ell(T) \geq \alpha) \geq c(t) \cdot \mathbb{P}^{t_1,t_2;a,b}_{free}(\ell(T) \geq  \alpha).
\end{equation}
\end{corollary}
\begin{proof}
If $\alpha \leq  m(T)$ then (\ref{monEq3}) becomes $1 \geq c(t)$, which is clearly true. We thus may assume that $M(T) \geq \alpha > m(T)$. Let $A = [\alpha,M(T)]$ and $B = [m(T), \alpha)$. Define $D_1 := \{ \ell \in \Omega(t_1,t_2;a,b) : \ell(T) \in A\}$  and $D_2 := \{ \ell \in \Omega(t_1,t_2;a,b) : \ell(T) \in B\}$. Observe that $A$ and $B$ satisfy the conditions of Corollary \ref{CorMon1} and hence
$$ \sum_{\ell \in D_1}  W_t(t_1, t_2, \ell, \ell_{bot}; S) \geq c(t) \cdot \frac{|D_1|}{|D_2|} \sum_{\ell \in D_2}  W_t(t_1, t_2,\ell, \ell_{bot}; S).$$
Dividing both sides by $ \sum_{\ell \in \Omega(t_1,t_2;a,b) }  W_t(t_1, t_2,\ell, \ell_{bot}; S)  $ we see that
$$\mathbb{P}(\ell(T) \geq \alpha) \geq c(t) \cdot \frac{|D_1|}{|D_2|} ( 1 -\mathbb{P}(\ell(T) \geq \alpha)) \mbox{ or equivalently } \mathbb{P}(\ell(T) \geq \alpha)  \geq c(t) \cdot \frac{ |D_1|}{|D_2| + c(t) |D_1|}. $$
Since $c(t) \in (0,1)$ we can increase the denominator by replacing it with $|D_1| + |D_2|$, which makes the RHS above precisely $c(t) \cdot \mathbb{P}^{t_1,t_2;a,b}_{free}(\ell(T) \geq  \alpha)$ as desired.
\end{proof}

\subsection{Properties of random paths}\label{Section4.2}
In this section we derive several lemmas about random paths distributed as $\mathbb{P}^{0,n;0,z}_{free}$ for $z \in \{0,...,n\}$, which are essential for the proof of our main results. Recall that if $L$ is such a path, we define $L(s)$ for non-integral $s$ by linear interpolation (see Section \ref{Section3.1}). The key ingredient we use to derive the lemmas below is a strong coupling between random walk bridges and Brownian bridges, which is presented as Theorem \ref{KMT} below. 

If $W_t$ denotes a standard one-dimensional Brownian motion and $\sigma > 0$, then the process
$$B^{\sigma}_t = \sigma^2 (W_t - t W_1), \hspace{5mm} 0 \leq t \leq 1,$$
is called a {\em Brownian bridge (conditioned on $B_0 = 0, B_1 = 0$)  with variance $\sigma^2$.}  With this notation we state the main result we use and defer its proof to Section \ref{Section8}.

\begin{theorem}\label{KMT}
Let $p \in (0,1)$. There exist constants $0 < C, a, \alpha < \infty$ (depending on $p$) such that for every positive integer $n$, there is a probability space on which are defined a Brownian bridge $B^\sigma$ with variance $\sigma^2 = p(1-p)$ and a family of random paths $\ell^{(n,z)} \in \Omega(0,n; 0, z)$ for $z = 0,...,n$ such that $\ell^{(n,z)}$ has law $\mathbb{P}^{0,n;0,z}_{free}$ and
\begin{equation}\label{KMTeq}
\mathbb{E}\left[ e^{a \Delta(n,z)} \right] \leq C e^{\alpha (\log n)^2}e^{|z- p n|^2/n}, \mbox{ where $\Delta(n,z):=  \sup_{0 \leq t \leq n} \left| \sqrt{n} B^\sigma_{t/n} + \frac{t}{n}z - \ell^{(n,z)}(t) \right|.$}
\end{equation}
\end{theorem}
\begin{remark} When $p = 1/2$ the above theorem follows (after a trivial affine shift) from Theorem 6.3 in \cite{LF}. The proof we present in Section \ref{Section8} for the more general $p \in(0,1)$ case is based on (suitably adapted) arguments from the same paper.
\end{remark}

We will also need the following monotone coupling lemma for random walks, which can readily be established from the arguments used in the proof of Lemma 2.6 in \cite{CorHamA}. 
\begin{lemma}\label{LemmaMonRW}
Suppose  $a_1,b_1,a_2,b_2,t_1,t_2$ are given such that $t_1 < t_2$, $0 \leq b_2 - a_2 \leq t_2 - t_1$, $0 \leq b_1 - a_1 \leq t_2 - t_1$, $a_1 \leq a_2$, $b_1 \leq b_2$. Then there exists a probability space on which are defined random paths $\ell_1$ and $\ell_2$ such that the law of $\ell_i$ is $\mathbb{P}_{free}^{t_1, t_2, a_i,b_i}$ for $i = 1,2$ and 
$\mathbb{P}(\ell_1(s) \leq \ell_2(s), \mbox{ for } s = t_1,...,t_2) = 1$.
\end{lemma}

Using facts about Brownian motion and the above coupling results we establish the following statements for random paths.
\begin{lemma}\label{LemmaHalf}
Let $M > 0$ and $p \in (0,1)$ be given. Then we can find $N_0(M,p)$ such that for $N \geq N_0$, $N \geq z \geq p N +  MN^{1/2}$ and $s \in [0,N]$ we have 
\begin{equation}\label{halfEq1}
\mathbb{P}^{0,N;0,z}_{free}\left( \ell(s)  \geq \frac{s}{N}( pN+ MN^{1/2}) - N^{1/4} \right) \geq 1/3.
\end{equation}
\end{lemma}
\begin{proof}
Assume that $N_0 \geq 2M^2$ and $N \geq N_0$. In view of Lemma \ref{LemmaMonRW}, we know that 
$$ \mathbb{P}^{0,N;0,z_2}_{free}\left( \ell(s)  \geq \frac{s}{N}( pN+ MN^{1/2}) - N^{1/4} \right) \geq   \mathbb{P}^{0,N;0,z_1}_{free}\left( \ell(s)  \geq \frac{s}{N}( pN+ MN^{1/2}) - N^{1/4} \right),$$
whenever $z_2 \geq z_1$ and so it suffices to prove the lemma when $z =  \lceil pN +  MN^{1/2} \rceil$. Suppose we have the same coupling as in Theorem \ref{KMT} and let $\mathbb{P}$ denote the probability measure on the space afforded by the theorem. Then we have for $\sigma^2 = p(1-p)$ that
$$ \mathbb{P}^{0,N;0,z}_{free}\left( \ell(s) \geq \frac{s}{N}( pN + MN^{1/2}) - N^{1/4} \right) = \mathbb{P}\left( \ell^{(N,z)}(s) \geq \frac{s}{N}(pN + MN^{1/2} ) - N^{1/4} \right) \geq$$
$$\geq \mathbb{P}\left( N^{1/2}B^{\sigma}_{s/N} \geq 0 \mbox{ and } \Delta(N,z) \leq (N^{1/4} - 1)\right ) \geq 1/2 - \mathbb{P}\left(\Delta(N,z) > N^{1/4}-1\right).$$
In the next to last inequality we used that $|z - (pN + MN^{1/2} )| \leq 1$ and in last inequality we used that $\mathbb{P}(B^{v}_{s/N} \geq 0) = 1/2$ for every $v > 0$ and $s \in [0,N]$. Next by Theorem \ref{KMT} and Chebyshev's inequality  we know
$$\mathbb{P}\left(\Delta(N,z) > N^{1/4}-1\right) \leq Ce^{\alpha (\log N)^2} e^{M^2} e^{-a N^{1/4}}.$$
The latter is at most $1/6$ if we take $N_0$ sufficiently large and $N \geq N_0$, which would imply that $\mathbb{P}^{0,N;0,z}_{free}\left( \ell(s) \geq (s/N)( pN + MN^{1/2}) - N^{1/4} \right)  \geq 1/3$ for such $N$, as desired.
\end{proof}

\begin{lemma}\label{LemmaTail}
Let $M_1,M_2 > 0$ and $p \in (0,1)$ be given. Then we can find $N_0(M_1,M_2,p)$ such that for $N \geq N_0$, $ z_1 \geq -M_1N^{1/2}$, $ z_2 \geq pN -  M_1N^{1/2}$ we have 
\begin{equation}\label{halfEq2}
\mathbb{P}^{0,N;z_1,z_2}_{free}\left( \ell(N/2)  \geq \frac{M_2N^{1/2} + p N}{2} - N^{1/4} \right) \geq (1/2) (1 - \Phi^{p(1-p)/2}(M_1 + M_2) ),
\end{equation}
where $\Phi^{v}$ is the cdf of a Gaussian random variable with mean $0$ and variance $v$.
\end{lemma}
\begin{proof}
Assume that $N_0 \geq 2(M_1 + M_2)^2$ and $N \geq N_0$. In view of Lemma \ref{LemmaMonRW} it suffices to prove the lemma when $z_1 = \lceil -M_1N^{1/2} \rceil$ and $ z_2  = \lceil p N -  M_1N^{1/2} \rceil$. Set $\Delta z = z_2 - z_1$ and observe that
$$ \mathbb{P}^{0,N;z_1,z_2}_{free}\left(\ell(N/2)  \geq \frac{M_2N^{1/2} + p N}{2}- N^{1/4} \right) =  \mathbb{P}^{0,N;0,\Delta z}_{free}\left(\ell(N/2)  \geq \frac{M_2N^{1/2} + p N}{2}-z_1 - N^{1/4} \right).$$
Suppose we have the same coupling as in Theorem \ref{KMT} and let $\mathbb{P}$ denote the probability measure on the space afforded by the theorem. Then we have
$$\mathbb{P}^{0,N;0,\Delta z}_{free}\left(\ell(N/2)  \geq \frac{M_2N^{1/2} + p N}{2}- z_1 - N^{1/4} \right) = \mathbb{P}\left( \ell^{(N,\Delta z)}(N/2) \geq   \frac{M_2N^{1/2} + p N}{2} - z_1 - N^{1/4} \right) $$
$$ \geq \mathbb{P}\left( \ell^{(N,\Delta z)}(N/2) \geq   \frac{(2M_1 + M_2)N^{1/2} + \Delta z}{2}- N^{1/4} + 2 \right), $$ 
where we used that $|z_1 + M_1N^{1/2}| \leq 1$ and $|z_2 + M_1N^{1/2} - pN| \leq 1$. We now note that the expression in the second line above is bounded from below by
$$\mathbb{P}\left( B^{\sigma}_{1/2} \geq  \frac{M_2 + 2M_1}{2} \mbox{ and } \Delta(N,z) \leq N^{1/4} - 2 \right) , \mbox{ where } \sigma^2 = p (1-p).$$
Since $B^{\sigma}_{1/2}$ has the distribution of a normal random variable with mean $0$ and variance $\sigma^2/2$, and $\Phi^{v}$ is decreasing on $\mathbb{R}_{ > 0}$  we conclude that the last expression is bounded from below by
$$ 1 - \Phi^{p(1-p)/2}(M_1 + M_2) - \mathbb{P}\left(\Delta(N,z) > N^{1/4}-2\right) \geq 1 - \Phi^{p(1-p)/2}(M_1 + M_2)  - Ce^{\alpha(\log N)^2} e^{M^2} e^{-a N^{1/4}}.$$
In the last inequality we used Theorem \ref{KMT} and Chebyshev's inequality. The above is at least $(1/2) (1 - \Phi^{p(1-p)/2}(M_1 + M_2) ) $ if $N_0$ is taken sufficiently large and $N \geq N_0$. 
\end{proof}

\begin{lemma}\label{LemmaAway}
Let $p \in (0,1)$ be given. Then we can find $N_0(p)$ such that for $N \geq N_0$, $ z_1 \geq N^{1/2}$, $ z_2 \geq pN +  N^{1/2}$ we have 
\begin{equation}\label{away}
\mathbb{P}^{0,N;z_1,z_2}_{free}\left( \min_{s \in [0,N]} \left[ \ell(s) -ps \right]+ N^{1/4} \geq 0 \right) \geq \frac{1}{2} \left(1 - \exp\left(\frac{-2}{p(1-p)}\right)\right).
\end{equation}
\end{lemma}
\begin{proof}
In view of Lemma \ref{LemmaMonRW} it suffices to prove the lemma when $z_1 = \lceil N^{1/2} \rceil$ and $z_2 =  \lceil  pN +  N^{1/2} \rceil$. Set $\Delta z = z_2 - z_1$ and observe that
$$\mathbb{P}^{0,N;z_1,z_2}_{free}\left( \min_{s \in [0,N]}\left[ \ell(s) -ps \right]+ N^{1/4} \geq 0 \right) = \mathbb{P}^{0,N;0,\Delta z}_{free}\left( \min_{s \in [0,N]} \left[ \ell(s) -ps \right]+ N^{1/4} \geq -z_1 \right).$$
Suppose we have the same coupling as in Theorem \ref{KMT} and let $\mathbb{P}$ denote the probability measure on the space afforded by the theorem. Then we have
$$\mathbb{P}^{0,N;0,\Delta z}_{free}\left( \min_{s \in [0,N]} \left[ \ell(s) -ps \right] + N^{1/4} \geq -z_1 \right) = \mathbb{P}\left( \min_{s \in [0,N]} \left[ \ell^{(N, \Delta z)}(s) - ps \right]\geq  - N^{1/4}  -z_1 \right) \geq $$
$$\mathbb{P}\left( \min_{s \in [0,N]} \left[ \ell^{(N, \Delta z)}(s) - \frac{s}{N} \Delta z \right]  \geq  - N^{1/4} - N^{1/2} + 2 \right),$$
where in the last inequality we used that $|z_1 - N^{1/2}| \leq 1$ and $|z_2 - pN - N^{1/2}| \leq 1$. We now note that the expression in the second line above is bounded from below by
$$\mathbb{P}\left( \min_{s \in [0,1]} B^{\sigma}_s \geq - 1 \mbox{ and } \Delta(N,z) \leq N^{1/4} -2  \right), \mbox{ where } \sigma^2 = p(1-p).$$
We can lower-bound the above expression by $\mathbb{P}\left( \min_{s \in [0,1]} B^{\sigma}_s \geq - 1 \right) - \mathbb{P}\left(\Delta(N,z) \leq N^{1/4} -2  \right)$. By basic properties of Brownian bridges we know that 
$$ \mathbb{P}\left( \min_{s \in [0,1]} B^{\sigma}_s \geq - 1 \right)  = \mathbb{P}\left( \min_{s \in [0,1]} B^{1}_s \geq - \sigma^{-1} \right) =  \mathbb{P}\left( \max_{s \in [0,1]} B^{1}_s \leq \sigma^{-1} \right) = 1 - e^{-2\sigma^{-2}},$$
where the last equality can be found for example  in (3.40) of Chapter 4 of \cite{KS}. On the other hand, by Theorem \ref{KMT} and Chebyshev's inequality we have
$$\mathbb{P}\left(\Delta(N,z) > N^{1/4}-2\right)  \leq Ce^{\alpha(\log N)^2} e^{M^2} e^{-a N^{1/4}},$$
and the latter is at most $(1/2)(1 - e^{-2\sigma^{-2}})$ if $N_0$ is taken sufficiently large and $N \geq N_0$. Combining the above estimates we conclude that if $N_0$ is sufficiently large and $N \geq N_0$, we have $\mathbb{P}^{0,N;z_1,z_2}_{free}\left( \min_{s \in [0,N]} \left[ \ell(s) -ps \right] + N^{1/4} \geq 0 \right)  \geq (1/2)(1 - e^{-2\sigma^{-2}})$ as desired.

\end{proof}

\subsection{Modulus of continuity for random paths}\label{Section4.3}
For a function $f \in C[a,b]$ we define the {\em modulus of continuity} by
\begin{equation}\label{MOC}
w(f,\delta) = \sup_{\substack{x,y \in [a,b]\\ |x-y| \leq \delta}} |f(x) - f(y)|.
\end{equation}
In this section we derive estimates on the modulus of continuity of paths distributed according to $\mathbb{P}^{0,n;0,z}_{free}$ for $z \in \{0,...,n\}$, which are essential for the proof of Theorem \ref{PropTightGood}. Recall that if $L$ is such a path, we define $L(s)$ for non-integral $s$ by linear interpolation (see Section \ref{Section3.1}). The main result we want to show is as follows.
\begin{lemma}\label{MOCLemma}
Let $M > 0$ and $p \in (0,1)$ be given. For each positive $\epsilon$ and $\eta$, there exist a $\delta > 0$ and an $N_0 \in \mathbb{N}$ (depending on $\epsilon, \eta, M$ and $p$) such that  for $N \geq N_0$ we have
\begin{equation}\label{MOCeq}
\mathbb{P}^{0,N;0,z}_{free}\left( w({f^\ell},\delta) \geq \epsilon \right) \leq \eta,
\end{equation}
where $f^\ell(x) = N^{-1/2}(\ell(xN) - pxN)$  for $x \in [0,1]$ and $|z - pN| \leq MN^{1/2}$. 
\end{lemma}
\begin{proof}
 The strategy is to use the strong coupling between $\ell$ and a Brownian bridge afforded by Theorem \ref{KMT}. This will allow us to argue that with high probability the modulus of continuity of $f^\ell$ is close to that of a Brownian bridge, and since the latter is continuous a.s., this will lead to the desired statement of the lemma. We now turn to providing the necessary details.

Let $\epsilon, \eta > 0$ be given and fix $\delta \in (0,1)$, which will be determined later. Suppose we have the same coupling as in Theorem \ref{KMT} and let $\mathbb{P}$ denote the probability measure on the space afforded by the theorem. Then we have
\begin{equation}\label{MOC1}
\mathbb{P}^{0,N;0,z}_{free}\left( w({f^\ell},\delta) \geq \epsilon \right) = \mathbb{P}\left(w({f^{\ell^{(N,z)}}} ,\delta) \geq \epsilon \right).
\end{equation}
By definition, we have
\begin{equation*}
w({f^{\ell^{(N,z)}}} ,\delta) = N^{-1/2} \sup_{\substack{x,y \in [0,1]\\ |x-y| \leq \delta}} \left|\ell^{(N,z)}(xN) - pxN - \ell^{(N,z)}(yN) + pyN \right|.
\end{equation*}
From Theorem \ref{KMT} and  the above we conclude that for $\sigma^2 = p(1-p)$ we have
\begin{equation}\label{MOC2}
w({f^{\ell^{(N,z)}}} ,\delta) \leq N^{-1/2} \sup_{\substack{x,y \in [0,1]\\ |x-y| \leq \delta}} \left|N^{1/2}B^\sigma_{x}  -N^{1/2}B^\sigma_{y} +(z- pN)(x-y)  \right| +2 N^{-1/2}\Delta(N,z) .
\end{equation}
From (\ref{MOC1}), (\ref{MOC2}), the triangle inequality and our assumption that $|z - pN| \leq MN^{1/2}$ we see that 
\begin{equation}\label{MOC3}
\mathbb{P}^{0,N;0,z}_{free}\left( w({f^\ell},\delta) \geq \epsilon \right) \leq  \mathbb{P}\left(w({B^\sigma},\delta) + \delta M +2 N^{-1/2}\Delta(N,z) \geq \epsilon \right).
\end{equation}

Let $(I) = \mathbb{P}\left(w({B^\sigma},\delta) \geq \epsilon /3\right)$ , $(II) =  \mathbb{P}\left(\delta M  \geq \epsilon /3\right) $ and $(III) = \mathbb{P}\left(2 N^{-1/2}\Delta(N,z)  \geq \epsilon /3 \right) $, then we have
$$\mathbb{P}\left(w({B^\sigma},\delta) + \delta M +2 N^{-1/2}\Delta(N,z) \geq \epsilon \right) \leq (I) + (II) + (III).$$
By Theorem \ref{KMT} and Chebyshev's inequality  we know
$$\mathbb{P}\left(\Delta(N,z) > N^{1/4}\right) \leq Ce^{\alpha(\log N)^2} e^{M^2} e^{-a N^{1/4}}.$$
Consequently, if we pick $N_0$ sufficiently large and $N \geq N_0$ we can ensure that $2N^{-1/4} < \epsilon/3$ and $ Ce^{\alpha(\log N)^2} e^{M^2} e^{-a N^{1/4}} < \eta/3$, which would imply $(III) \leq \eta /3$. 

Since $B^{\sigma}$ is a.s. continuous we know that $w({B^\sigma},\delta) $ goes to $0$ as $\delta$ goes to $0$, hence we can find $\delta_0$ sufficiently small so that if $\delta < \delta_0$, we have $(I) < \eta/3$. Finally, if $\delta M < \epsilon /3$ then $(II) = 0$. Combining all the above estimates with (\ref{MOC3}) we see that for $\delta$ sufficiently small, $N_0$ sufficiently large and $N \geq N_0$, we have $\mathbb{P}^{0,N;0,z}_{free}\left( w_{f^\ell}(\delta) \geq \epsilon \right) \leq (2/3) \eta < \eta$ as desired.

\end{proof}


\section{Proof of Theorem \ref{PropTightGood}}\label{Section5}

The goal of this section is to prove Theorem \ref{PropTightGood} and for the remainder we assume that $\mathfrak{L}^N = (L^N_1,L^N_2)$ is an $(\alpha,p,r + 1)$-good sequence for some $r > 0$, defined on a probability space with measure $\mathbb{P}$. The main technical result we will require is contained in Proposition \ref{PropMain} below and its proof is the content of Section \ref{Section5.1}. The proof of Theorem \ref{PropTightGood} is given in Section \ref{Section5.2} and relies on Proposition \ref{PropMain} and Lemma \ref{MOCLemma}.

\subsection{Bounds on acceptance probabilities}\label{Section5.1}
 The main result in this section is the following.
\begin{proposition}\label{PropMain}
Fix $r  >0 $ and  denote $s_1 = \lfloor r N^{\alpha} \rfloor$. Then for any $\epsilon > 0$ there exist $\delta > 0$ and $N_1$ (both depending on $r,\epsilon, t, \alpha, p)$ such that for all $N \geq N_1$ we have
$$\mathbb{P}\left(Z_t( -s_1,s_1,L_1^N(- s_1),L_1^N(s_1), L_2; S') < \delta\right) < \epsilon,$$
where $S' = \llbracket - s_1+ 1, s_1\rrbracket$ and $Z_t$ is the acceptance probability defined in Definition \ref{DefHLGP} (see also Remark \ref{RemHLGP}).
\end{proposition}

The general strategy we use to prove Proposition \ref{PropMain} is inspired by the proof of Proposition 6.5 in \cite{CorHamK}. We begin by stating three key lemmas that will be required. Their proofs are postponed to Section \ref{Section6}. All constants in the statements below will, in addition, depend on $\alpha, t$ and $p$, which are fixed throughout. We will not list this dependence explicitly.

\begin{lemma}\label{PropSup} For each $\epsilon > 0$ there exist $R(r, \epsilon) > 0$ and $N_1(r,\epsilon)$ such that for all $N \geq N_1$ we have 
$$\mathbb{P}\left( \sup_{s \in [-(r+1)N^{\alpha},(r+1)N^{\alpha}]}\left[ L^N_1(s) - p s \right] \geq  RN^{\alpha/2} \right) < \epsilon.$$
\end{lemma}

Set  $s_1 =\lfloor r N^{\alpha} \rfloor$ and $t_1 = \lfloor (r+1)N^{\alpha} \rfloor$ and assume $a,b,z_1,z_2,t_1$ satisfy, $0 \leq z_2 - z_1 \leq 2t_1$, $0 \leq b- a \leq 2t_1$, $z_1 \leq a$, $z_2 \leq b$. Let $\ell_{bot}$ be a fixed path in $\Omega(-t_1,t_1; z_1,z_2)$ and denote $S = \llbracket -t_1 +1 , t_1\rrbracket$, $\tilde{S} = \llbracket -t_1 + 1, -s_1\rrbracket \cup \llbracket s_1 + 1, t_1\rrbracket$. Let $L$ and $\tilde{L}$ be two random paths in $\Omega(-t_1,t_1; a,b)$ , with laws $\mathbb{P}_L$ and $\mathbb{P}_{\tilde L}$ respectively such that
$$\mathbb{P}_L (L = \ell) = \mathbb{P}_S^{-t_1, t_1, a,b} (\ell | \ell_{bot} ) \mbox{ and } \mathbb{P}_{\tilde L }(\tilde L = \ell) = \mathbb{P}_{\tilde  S}^{-t_1, t_1, a,b} (\ell | \ell_{bot} ).$$ 
where the definition of $ \mathbb{P}_S^{T_0, T_1, a,b} (\cdot | \ell_{bot} )$ was given in Definition \ref{DefHLGP}. From (\ref{S1eq0}) we know that $L$ will not cross $\ell_{bot}$ with probability $1$. On the other hand, $\tilde{L}$ can cross $\ell_{bot}$ multiple times in the interval $(-s_1, s_1 +1)$  but it will stay above it on $[-t_1, -s_1] \cup [s_1 + 1 ,t_1]$.

\begin{lemma}\label{LemmaAP1} Fix $M_1, M_2 > 0$, $S' = \llbracket -s_1 + 1,s_1\rrbracket $ and suppose
\begin{enumerate}
\item $\sup_{s \in [-t_1,t_1]}\left[ \ell_{bot}(s)  - ps \right]  \leq M_2N^{\alpha/2}$,
\item  $ a \geq \max( \ell_{bot}(-t_1), -pt_1- M_1 N^{\alpha/2}),$ 
\item $ b \geq \max( \ell_{bot}(t_1), p t_1- M_1N^{\alpha/2}).$
\end{enumerate}
There exists $N_2 \in \mathbb{N}$ and explicit functions $g$ and $h$ (depending on $r,  M_1, M_2$) such that for $N \geq N_2$ 
\begin{equation}\label{eqn57}
\mathbb{P}_{\tilde{{L}}} \left( {Z_t}\left(  - s_1, s_1, \tilde{L}(-s_1) ,\tilde{L}(s_1), \ell_{bot}; S'\right) \geq g    \right) \geq h.
\end{equation}
The functions $g$ and $h$ are given by 
$$g =  \frac{1}{4} \left( 1 - \exp \left( \frac{-2}{p (1-p)} \right) \right) \mbox{ and } h = (c(t)^3/18)(1 - \Phi^{p(1-p)/2}\left(10(1 +r)^2 (M_1 + M_2 + 1) \right),$$
where $c(t) = \prod_{i = 1}^\infty(1 -t^i)$ and $\Phi^v$ is the cdf of a Gaussian random variable with mean zero and variance $v$. 
\end{lemma}

\begin{lemma}\label{LemmaAP2} Fix $M_1, M_2 > 0$, $S' = \llbracket -s_1 + 1,s_1 \rrbracket$  and suppose
\begin{enumerate}
\item $\sup_{s \in [-t_1,t_1]}\left[ \ell_{bot}(s)  - ps \right] \leq M_2N^{\alpha/2}$,
\item  $ a \geq \max( \ell_{bot}(-t_1), -pt_1- M_1 N^{\alpha/2}),$ 
\item $ b \geq \max( \ell_{bot}(t_1), p t_1- M_1N^{\alpha/2}).$
\end{enumerate}
Let $N_2, g, h$ be as in Lemma \ref{LemmaAP1} and for any $\tilde{\epsilon}  > 0$ set $\delta(\tilde \epsilon)  = \tilde\epsilon \cdot g \cdot h$. Then for $N \geq N_2$ we have
\begin{equation}\label{eqn60}
\mathbb{P}_{{{L}}} \left( {Z_t}\left(  - s_1, s_1, {L}(-s_1) ,{L}(s_1), \ell_{bot}; S'\right) \leq  \delta (\tilde \epsilon) \right) \leq  \tilde{\epsilon}.
\end{equation}
\end{lemma}

In the remainder we prove Proposition \ref{PropMain} assuming the validity of Lemmas \ref{PropSup} and \ref{LemmaAP2}. The arguments we present are similar to those used in the proof of Proposition 6.5 in \cite{CorHamK}.
\begin{proof}(Proposition \ref{PropMain}) Define the event 
\begin{equation*}
\begin{split}
&E_N =  \left\{    L_1^N(-t_1) \geq \max( L^N_2(-t_1), - p t_1 - M_1 N^{\alpha/2})\right\} \cap\\
&\left\{   L_1^N(t_1) \geq \max( L^N_2(t_1), p t_1- M_1 N^{\alpha/2})\right\} \cap \left\{ \sup_{s \in [-t_1, t_1]} \left[ {L}^N_2(s) - p s \right]\leq M_2 N^{\alpha/2} \right\},
\end{split}
\end{equation*}
where $M_1$ and $M_2$ are sufficiently large so that for all large $N$ we have $\mathbb{P}(E_N^c) <  \epsilon / 2$. The existence of such $M_1$ and $M_2$ is assured from Lemma \ref{PropSup} (since ${L}^N_1$ dominates ${L}^N_2$ pointwise) and the fact that $\mathfrak{L}^N$ is $(\alpha, p, r+1)$ - good. 

Let $ \delta(\tilde \epsilon)$ be as in Lemma \ref{LemmaAP2} for the values $\tilde \epsilon = \epsilon/2$, $r, M_1, M_2$ in the statement of the lemma. Consider the probability
\begin{equation}\label{eqn63}
\begin{split}
&\mathbb{P}\left(  \left\{Z_t( -s_1, s_1,{L}_1^N(-s_1),{L}_1^N(s_1),{L}^N_2; S' ) < \delta( \tilde \epsilon)\right\} \cap E_N \right) = \\
&= \mathbb{E} \left[ {\bf 1}_{E_N} \mathbb{E}\left[ {\bf 1}\{Z_t( -s_1, s_1,{L}_1^N(-s_1),{L}_1^N(s_1),{L}^N_2;S' ) < \delta( \tilde\epsilon)\} \Big{|} \mathcal{F}_{ext} \left( \{1\} \times (-t_1,t_1)\right)\right] \right].
\end{split}
\end{equation}
In the above equation we have $ \mathcal{F}_{ext} \left( \{1\} \times (-t_1, t_1)\right)$ is the $\sigma$-algebra generated by the up-right paths ${L}_2^N$ and ${L}_1^N$ outside the interval $(-t_1,t_1)$. The equality in (\ref{eqn63}) is justified by the tower property since $E_N$ is measurable with respect to $\mathcal{F}_{ext} \left( \{1\} \times (-t_1,t_1)\right)$ . We next notice that we have the following a.s. equality of $\mathcal{F}_{ext} \left( \{1\} \times (-t_1,t_1)\right)$-measurable random variables
\begin{equation*}
\begin{split}
&\mathbb{E}\left[ {\bf 1}\{Z_t( -s_1, s_1,{L}_1^N(-s_1),{L}_1^N(s_1),{L}^N_2 ; S') < \delta( \tilde\epsilon) \} \Big{|} \mathcal{F}_{ext} \left( \{1\} \times(-t_1,t_1)\right)\right] = \\
&\mathbb{P}_{{L}} \left( Z_t(-s_1, s_1,{L}(-s_1),{L}(s_1),{L}^N_2 ;S' ) < \delta( \tilde\epsilon)\right),
\end{split}
\end{equation*}
where $\mathbb{P}_{{L}}$ is specified as in the setup after Lemma \ref{PropSup} with respect to $a = {L}^N_1(-t_1)$, $b = {L}^N_1(t_1)$, $\ell_{bot} = {L}^N_2$ on $[-t_1,t_1]$. 

When the $\mathcal{F}_{ext} \left( \{1\} \times (-t_1, t_1)\right)$-measurable event $E_N$ holds we have that $\sup_{s \in [-t_1, t_1]} \left[ \ell_{bot}(s) - p s \right] \leq M_2N^{\alpha/2}$ and $a \geq \max( \ell_{bot}(-t_1), -p t_1- M_1 N^{\alpha/2})$, $b \geq \max( \ell_{bot}(t_1), p t_1- M_1N^{\alpha/2})$ (recall that $\mathfrak{L}^N$ is a simple discrete line ensemble by definition so that ${L}^N_1$ lies above ${L}^N_2$). Thus we may apply Lemma \ref{LemmaAP2} on $E_N$ and obtain that 
$$ \mathbb{P}_L \left( Z_t(S', -s_1,s_1,{L}(-s_1),{L}(s_2),{L}^N_2 ) < \delta( \tilde\epsilon)\right) \leq \tilde\epsilon {\bf 1}_{E_N} + {\bf 1}_{E_N^c},$$
where the inequality is understood in the a.s. sense. Putting this into (\ref{eqn63}) we conclude that 
$$\mathbb{P}\left( \{ Z_t(S', -s_1,s_1,{L}_1^N(-s_1),{L}_1^N(s_1),{L}^N_2 ) < \delta( \epsilon/2) \} \cap E_N \right) \leq  \epsilon/2.$$
Using this and $\mathbb{P}(E_N^c) <  \epsilon/2$, we see that for all large $N$ we have
$$\mathbb{P}\left( Z_t(S', -s_1,s_1,{L}_1^N(-s_1),{L}_1^N(s_1),{L}^N_2 ) < \delta( \epsilon/2) \right) < \epsilon.$$
\end{proof}

\subsection{Concluding the proof of Theorem \ref{PropTightGood} }\label{Section5.2}
For clarity we split the proof of Theorem \ref{PropTightGood} into several steps. In the first two steps we reduce the statement of the theorem to establishing a certain estimate on the modulus of continuity of the paths $L_1^N$. In the next two steps we show that it is enough to establish these estimates under the additional assumption that $(L^N_1, L^N_2)$ are well-behaved (in particular, well-behaved implies that the acceptance probability  $Z_t (-s_1,s_1,L_1^N(- s_1),L_1^N(s_1), L_2;S')$ is lower bounded and it is here that we use Proposition \ref{PropMain}). The fact that the acceptance probability is lower bounded is exploited in Step 5, together with the resampling property of Remark \ref{RemHLGP}, to effectively reduce the estimates on the modulus of continuity of $L_1^N$ to those of a uniform random path. The latter estimates are then derived in Step 6, by appealing to Lemma \ref{MOCLemma}.
\vspace{2mm}

{\raggedleft \bf Step 1.} Recall from (\ref{MOC}) that the modulus of continuity of $f \in C[-r,r]$ is defined by
$$ w(f,\delta) = \sup_{\substack{x,y \in [-r,r]\\ |x-y| \leq \delta}} |f(x) - f(y)|.$$
As an immediate generalization of Theorem 7.3 in \cite{Bill}, in order to prove the theorem it suffices for us to show that the sequence of random variables $f_N(0)$ is tight and that for each positive $\epsilon$ and $\eta$ there exist $\delta' > 0 $ and $N_1 \in \mathbb{N}$ such that for $N \geq N_1$, we have
\begin{equation}\label{PTG2}
\mathbb{P} ( w(f_N, \delta') \geq \epsilon ) \leq \eta.
\end{equation}
The tightness of $f_N(0)$ is immediate from our assumption that $\{\mathfrak{L}^N\}_{N = 1}^\infty$ is an $(\alpha, p, r+1)$-good sequence (in fact we know from Definition \ref{Def1} that $f_N(s)$ is tight for each $s \in [-r -1, r+1]$). Consequently, we are left with verifying (\ref{PTG2}).\\

{\raggedleft \bf Step 2.} Suppose $\epsilon,\eta > 0$ are given and also denote $s_1 = \lfloor rN^{\alpha} \rfloor$. We claim that we can find $  \delta > 0$ such that for all $N$ sufficiently large we have
\begin{equation}\label{PTG2.5}
\mathbb{P} \left(  \sup_{\substack{ x,y \in [-s_1, s_1] \\  |x - y| \leq  2  \delta s_1 }}  \left| L_1^N(x) - L_1^N(y) - p(x - y) \right| \geq \frac{\epsilon (2s_1)^{1/2}}{2 (2r)^{1/2}}  \right) \leq \eta.
\end{equation}
Let us assume the validity of (\ref{PTG2.5}) and deduce (\ref{PTG2}).

Let $\delta' = r \delta$. Suppose that $x, y \in [-r,r]$ are such that $|x - y| \leq \delta'$ and without loss of generality assume that $x < y$. Let $X = \lceil xN^{\alpha}\rceil$ and $Y = \lfloor yN^{\alpha} \rfloor$. One readily observes that if $N$ is sufficiently large then $|X - Y| \leq 2  \delta s_1$, and $X,Y \in [-s_1,s_1]$. In addition, we have that 
$$ |f_N(x) - f_N(y)| =  N^{-\alpha/2} \left|  L_1^N(xN^{\alpha}) - L_1^N(yN^{\alpha}) - pN^{\alpha}(x - y) \right| \leq  $$
$$N^{-\alpha/2}  \left|  L_1^N(X) - L_1^N(Y) - p(X - Y) \right| + 2N^{-\alpha/2}(1+p),$$
where we used that $| X - xN^{\alpha}| < 1$, $| Y - yN^{\alpha}| < 1$, the slope of $L_1$ is in absolute value at most $1$, and the triangle inequality. The above inequality shows that for all $N$ sufficiently large we have
$$\mathbb{P} \left( w(f_N, \delta') \geq \epsilon \right) \leq \mathbb{P} \left(  \sup_{\substack{ x,y \in [-s_1, s_1] \\  |x - y| \leq 2  \delta s_1 }}  \left| L_1^N(x) - L_1^N(y) - p(x - y) \right| \geq \epsilon N^{\alpha/2} - 2(1+p) \right).$$
Since $s_1 = \lfloor rN^{\alpha} \rfloor$ we see that $\frac{\epsilon (2s_1)^{1/2}}{2 (2r)^{1/2}} \sim (\epsilon /2) N^{\alpha/2}$ as $N$ becomes large and so we conclude that for all sufficiently large $N$ we have $\frac{\epsilon (2s_1)^{1/2}}{2 (2r)^{1/2}} \leq \epsilon N^{\alpha/2} - 2(1+p) $. This together with (\ref{PTG2.5}) implies that the RHS in the last equation is  bounded from above by $\eta$, which is what we wanted. \\

{\raggedleft \bf Step 3.} The first two steps above reduce the proof of the theorem to establishing (\ref{PTG2.5}), which is the core statement we want to show. In order to prove it we will need additional notation that we summarize in this step.

From the tightness of $N^{-\alpha/2}\left[L^N_1(xN^{\alpha}) - xp N^{\alpha}\right]$ at $x = -r$ and $x = r$ we can find $M_1 > 0$ sufficiently large so that for all large $N$ we have 
$$\mathbb{P} \big{(} (E_1 (M_1,N) \big{)} \geq 1 - \eta/4 \mbox{, where }  E_1(M_1,N) = \left\{ \max \left(\left| L^N_1(-s_1) + ps_1 \right| , \left| L^N_1(s_1) - ps_1 \right| \right) \leq M_1N^{\alpha/2} \right\}.$$
In addition, we know from Proposition \ref{PropMain} that we can find $\delta_1 > 0$ such that for all sufficiently large $N$ we have
$$\mathbb{P} \big{(} E_2 (\delta_1,N) \big{)} \geq 1 - \eta/4 \mbox{, where }  E_2(\delta_1,N) = \left\{ Z_t ( -s_1,s_1,L_1^N(- s_1),L_1^N(s_1), L_2 ;S') > \delta_1  \right\}.$$
Suppose  $a,b,z_1,z_2$ are given such that $0 \leq z_2 - z_1 \leq 2s_1$, $0 \leq b - a \leq 2s_1$, $z_1 \leq a$, $z_2 \leq b$. For a given $\ell_{bot} \in \Omega(-s_1,s_1;z_1,z_2)$, we let 
$$E(a,b,\ell_{bot}, N) := \left\{ (L_1^N, L_2^N) : L_2^N = \ell_{bot}\mbox{ on $[-s_1, s_1]$, $L_1^N(-s_1) = a$ and $L_1^N(s_1) = b$ }\right\}.$$
Observe that $E_1(M_1,N) \cap E_2(\delta_1,N)$ can be written as a {\em countable disjoint} union of sets of the form $E(a,b,\ell_{bot}, N)$, where the triple $(a,b,\ell_{bot})$ satisfies:
\begin{enumerate}
\item $0 \leq b-a \leq 2s_1$, $| a +ps_1| \leq M_1 N^{\alpha/2}$ and $|b - ps_1| \leq M_1 N^{\alpha}$,
\item $z_1 \leq a$, $z_2 \leq b$ and $\ell_{bot} \in \Omega(-s_1, s_1, z_1, z_2)$,
\item $Z_t (S', -s_1,s_1,a,b, \ell_{bot}) > \delta_1$.
\end{enumerate}
Clearly, there are only finitely many choices for $a,b$ that satisfy the conditions above. Then the number of $z_1, z_2$ for each given pair $(a,b)$ is countable, while the cardinality of $\Omega(-s_1, s_1, z_1, z_2)$ is finite. This means that the number of triplets $(a,b,\ell_{bot})$ is indeed countable. The fact that $E(a,b,\ell_{bot}, N)$ are disjoint is again clear, while the first and third condition above show that their union is indeed $E_1(M_1,N) \cap E_2(\delta_1,N)$. Let us denote by $F(\delta_1, M_1, s_1, N)$ the set of triplets $(a,b,\ell_{bot})$ that satisfy the three conditions above.\\

{\raggedleft \bf Step 4.} Let us write $L_1^N( [-s_1, s_1])$ as the restriction of $L_1^N$ to $[-s_1, s_1]$. For $\delta > 0$ and $\ell \in   \Omega(-s_1, s_1; a,b)$ we define 
$$V(\delta, \ell) = \sup_{\substack{ x,y \in [-s_1, s_1] \\  |x - y| \leq  2\delta s_1 }}  \left| \ell(x) - \ell(y) - p(x - y) \right|.$$
We assert that we can find $\delta > 0$ such that for all large $N$ and $(a,b, \ell_{bot}) \in F(\delta_1, M_1, s_1, N)$, we have
\begin{equation}\label{PTG7}
\mathbb{P}\Big{(} V(\delta, L_1^N([-s_1, s_1]) )\geq A \Big{|} E(a,b,\ell_{bot}, N) \Big{)} \leq \eta/4, \mbox{ where } A = \frac{\epsilon (2s_1)^{1/2}}{2 (2r)^{1/2}} .
\end{equation}

Let us assume the validity of (\ref{PTG7}) and deduce (\ref{PTG2.5}). We have
$$\mathbb{P}\Big{(} V(\delta, L_1^N([-s_1, s_1]) )\geq A \Big{)} \leq \mathbb{P} \Big{(}\left\{V(\delta, L_1^N([-s_1, s_1]) )\geq A \right\} \cap E_1(M_1,N) \cap E_2(\delta_1,N) \Big{)} + \eta/2,$$
where we used that $\mathbb{P}( E^c_1(M_1,N) ) \leq \eta/4$ and $\mathbb{P}( E^c_2(\delta_1,N) ) \leq \eta/4$. In addition, we have
$$\mathbb{P} \Big{(}\left\{V(\delta, L_1^N([-s_1, s_1]) )\geq A \right\} \cap E_1(M_1,N) \cap E_2(\delta_1,N) \Big{)} = $$
$$\sum_{(a,b,\ell_{bot}) \in F(\delta_1, M_1, s_1, N)} \mathbb{P} \Big{(}\left\{V(\delta, L_1^N([-s_1, s_1]) )\geq A \right\} \cap E(a,b,\ell_{bot}, N) \Big{)},$$
where we used that $E_1(M_1,N) \cap E_2(\delta_1,N) $ is a disjoint union of $E(a,b,\ell_{bot}, N)$. Finally, we have from (\ref{PTG7}) above that
$$\mathbb{P} \Big{(}\left\{V(\delta, L_1^N([-s_1, s_1]) )\geq A \right\} \cap E(a,b,\ell_{bot}, N) \Big{)} = $$
$$\mathbb{P}\Big{(} V(\delta, L_1^N([-s_1, s_1]) )\geq A\Big{|} E(a,b,\ell_{bot}, N) \Big{)} \mathbb{P} \big{(} E(a,b,\ell_{bot}, N) \big{)} \leq (\eta/4) \cdot  \mathbb{P} \big{(} E(a,b,\ell_{bot}, N) \big{)}.$$
Summing the latter over $(a,b,\ell_{bot}) \in F(\delta_1, M_1, s_1, N)$ and combining it with the earlier inequalities we see that
$$\mathbb{P}\Big{(} V(\delta, L_1^N([-s_1, s_1]) )\geq A \Big{)} \leq \eta/2 + \eta/4 \cdot \hspace{-13mm} \sum_{(a,b,\ell_{bot}) \in F(\delta_1, M_1, s_1, N)}  \hspace{-13mm}\mathbb{P} \big{(} E(a,b,\ell_{bot}, N) \big{)}  = $$
$$ = \eta/2 + (\eta/4) \cdot  \mathbb{P} \big{(}  E_1(M_1,N) \cap E_2(\delta_1,N)\big{)} < \eta,$$
where in the middle equality we again used that $E_1(M_1,N) \cap E_2(\delta_1,N) $ is a disjoint union of $E(a,b,\ell_{bot}, N)$. The last equation implies (\ref{PTG2.5}).\\

{\raggedleft \bf Step 5.} In this step we establish (\ref{PTG7}) and begin by fixing $(a,b, \ell_{bot}) \in F(\delta_1, M_1, s_1, N)$. Since $\mathfrak{L}^N$ satisfies the Hall-Littlewood Gibbs property on $\llbracket -s_1, s_1\rrbracket$ with respect to $S' = \llbracket -s_1 + 1, s_1\rrbracket$ for $N$ sufficiently large we know that  
\begin{equation}\label{PTG3}
\mathbb{P}\big{(} L_1^N([-s_1, s_1]) = \ell \big{|} E(a,b,\ell_{bot}, N) \big{)} =  \mathbb{P}^{-s_1, s_1, a,b}_{S'}\left(\ell| \ell_{bot} \right) \mbox{ for any $\ell \in \Omega(-s_1, s_1; a,b)$} .
\end{equation}

We now recall the sampling property we explained in Remark \ref{RemHLGP}. Let $\ell^K$ be a sequence of i.i.d. up-right paths distributed according to $\mathbb{P}_{free}^{-s_1,s_1;a,b}$. Also let $U$ be a uniform random variable on $(0,1)$, independent of all else. For each $K \in \mathbb{N}$ we check if $W_t(-s_1, s_1, \ell^K, \ell_{bot}; S') > U$ and set $Q$ to be the minimal index $K$, which satisfies the inequality. Then we have that $Q$ is a geometric random variable with parameter $Z_t ( -s_1,s_1,a,b, \ell_{bot} ;S') $ and 
\begin{equation}\label{PTG4}
 \tilde{\mathbb{P}} \left( \ell^Q = \ell\right) = \mathbb{P}^{-s_1, s_1, a,b}_{S'}\left(\ell| \ell_{bot} \right) \mbox{ for any $\ell \in \Omega(-s_1, s_1; a,b)$} ,
\end{equation}
where $ \tilde{\mathbb{P}}$ is the probability measure on a space on which $\ell^K$ and $U$ are defined, we also write $\tilde{\mathbb{E}}$ for the expectation with respect to $\tilde{\mathbb{P}}$.

By our assumption that $(a,b, \ell_{bot}) \in F(\delta_1, M_1, s_1, N)$, we know that $Z_t ( -s_1,s_1,a,b, \ell_{bot};S')  > \delta_1$ and so $\tilde{\mathbb{E}}[ Q] = Z_t ( -s_1,s_1,a,b, \ell_{bot};S')^{-1} \leq \delta_1^{-1}$. It follows that if we take $R =  8 \delta_1^{-1} \eta^{-1} $, then by Chebyshev's inequality we have
\begin{equation}\label{PTG5}
\tilde{\mathbb{P}} ( Q > R) \leq \eta/8.
\end{equation}

Fix $A = \frac{\epsilon(2s_1)^{1/2}}{2 (2r)^{1/2}}$ and observe that
$$\tilde{\mathbb{P}} \left( V(\delta, \ell^Q) \geq A \right) =  \tilde{\mathbb{P}} \left( V(\delta, \ell^Q) \geq A, Q > R   \right) +  \tilde{\mathbb{P}} \left( V(\delta, \ell^Q) \geq A, Q \leq R   \right) \leq \tilde{\mathbb{P}} \left(  Q > R   \right)  + $$
$$ \tilde{\mathbb{P}} \left( \max_{1 \leq i \leq R} V(\delta, \ell^i) \geq A  \right)   = \tilde{\mathbb{P}} \left(  Q > R   \right)  + 1 -  \tilde{\mathbb{P}} \left( \max_{1 \leq i \leq R} V(\delta, \ell^i) <  A  \right) \leq 1 - \tilde{\mathbb{P}} \left(  V(\delta, \ell^1) <  A  \right)^{\lfloor R \rfloor} + \eta/8.$$
In the last inequality we used (\ref{PTG5}) and the independence of $\ell^i$. Combining the latter inequality with (\ref{PTG3}) and (\ref{PTG4}) we see that
\begin{equation}\label{PTG6}
\mathbb{P}\left( V(\delta, L_1^N([-s_1, s_1]) )\geq A| E(a,b,\ell_{bot}, N) \right) \leq 1 - \tilde{\mathbb{P}} \left(  V(\delta, \ell^1) <  A  \right)^{\lfloor R \rfloor} + \eta/8.
\end{equation}
Equation (\ref{PTG7}) would now follow from (\ref{PTG6}) if we can show that for any $\epsilon' > 0$  we can find $\delta > 0$ (depending on $M_1$, $\epsilon'$, $\eta$, $r$ and $p$),  such that for all large $N$ we have
\begin{equation}\label{PTG8}
\tilde{\mathbb{P} }\left(  V(\delta, \ell^1) <  A  \right)\geq 1 - \epsilon'.
\end{equation}

{\raggedleft \bf Step 6.} In this final step we establish (\ref{PTG8}), which is the remaining statement we require. Notice that $A = \tilde \epsilon(2s_1)^{1/2}$, where $\tilde \epsilon =  \frac{\epsilon}{2 (2r)^{1/2}}$. The key observation we make is the following 
\begin{equation}\label{PTG6.5}
\tilde{\mathbb{P}} \left(  V(\delta, \ell^1) <  A  \right) = \mathbb{P}_{free}^{0, 2s_1; 0, b-a} \left( w({f^{\ell^1}},\delta) < \tilde\epsilon \right), 
\end{equation}
where $f^\ell(x) = (2s_1)^{-1/2}(\ell(2x s_1) - 2pxs_1))$  for $x \in [0,1]$ and $w(f,\delta)$ denotes the modulus of continuity on $[0,1]$ as in (\ref{MOC}). 

Notice that since $(a,b, \ell_{bot}) \in F(\delta_1, M_1, s_1, N)$, we know that $|b - a  - 2ps_1 | \leq 2M_1 N^{\alpha/2} \leq \frac{4M_1}{ (2r)^{1/2}} (2s_1)^{1/2}$ for all large $N$. The latter means that we can apply Lemma \ref{MOCLemma}, and find $\delta > 0$ (depending on $M_1$, $\epsilon'$, $\tilde \epsilon$, $\eta$ and $p$),  such that for all large $N$ we have
$$\mathbb{P}_{free}^{0, 2s_1; 0, b-a} \left( w({f^{\ell^1}},\delta) < \tilde\epsilon \right) \geq 1 - \epsilon'.$$ 
Combining the latter with (\ref{PTG6.5}) concludes the proof of (\ref{PTG8}).

\begin{remark}
An important idea in our arguments above is to condition on $ E(a,b,\ell_{bot}, N)$ and obtain estimates on these events, where additional structure is available to us. The latter is possible because of the discrete nature of our problem and substitutes the more involved notions of {\em stopping domains} and {\em strong Brownian Gibbs properties} that were used in \cite{CorHamA} and \cite{CorHamK}.
\end{remark}

\section{Proof of three key lemmas}\label{Section6}
Here we prove the three key lemmas from Section \ref{Section5.1}. The arguments we use below heavily depend on the results from Section \ref{Section4}. 

\subsection{Proof of Lemma \ref{PropSup}}\label{Section6.1}

 Let us start by fixing notation. As in Section \ref{Section5.1} we set  $s_1 =\lfloor r N^{\alpha} \rfloor$ and $t_1 = \lfloor (r+1)N^{\alpha} \rfloor$. Define the events
$$E(M) = \left\{ \left| {L}^N_1(-t_1) + p t_1\right| > MN^{\alpha/2} \right\},F(M) = \left\{ L_1(-s_1) > -ps_1 + MN^{\alpha/2} \right\} \mbox{ and }$$
$$G(M) =  \left\{  \sup_{s \in [0,t_1]} \left[ L_1^N(s) -p s \right] >  (6r + 10)(M+1)N^{\alpha/2}\right\}.$$
For $a,b \in \mathbb{Z}$  and $s \in \{0,1,...,t_1\}$ as well as a path $\ell_{bot}$ in $\Omega(-t_1, s; z_1,z_2)$, where $z_1 \leq a$ and $z_2 \leq b$ we define $E(a,b,s, \ell_{bot})$ to be the event that $L^N_1(-t_1) = a$, $L^N_1(s) = b$, and $L^N_2$ agrees with $\ell_{bot}$ on $[-t_1, s]$. We will also write $L^N_1([m,n])$ for the restriction of $L^N_1$ to the interval $[m,n]$.

Observe that $E^c(M) \cap G(M)$ can be written as a {\em countable disjoint} union of sets of the form $E(a,b,s,\ell_{bot})$, where the quadruple $(a,b,s,\ell_{bot})$ satisfies:
\begin{enumerate}
\item $0 \leq s \leq t_1$,
\item $0 \leq b-a \leq t_1 + s$, $|a +pt_1| \leq  M N^{\alpha/2}$ and $b - ps > (6r + 10)(M +1)N^{\alpha/2}$,
\item $z_1 \leq a$, $z_2 \leq b$ and $\ell_{bot} \in \Omega(-t_1, s, z_1, z_2)$,
\end{enumerate}
Clearly, there are only finitely many choices for $s$ and for any $s$ there are countably many  $a,b$ that satisfy the conditions above. Then the number of $z_1, z_2$ for each given pair $(a,b)$ is again countable, while the cardinality of $\Omega(-t_1, s, z_1, z_2)$ is finite. This means that the number of quadruples $(a,b,s,\ell_{bot})$ is indeed countable. The fact that $E(a,b,s,\ell_{bot})$ are disjoint is again clear, while the first and second condition above show that their union is indeed $E^c(M) \cap G(M)$. Let us denote by $D( M)$ the set of quadruples $(a,b,s,\ell_{bot})$ that satisfy the three conditions above.\\

By $1$-point tightness of ${L}^N_1$ we know that there exists $M> 0$ sufficiently large so that for every $N \in \mathbb{N}$ we have
\begin{equation}\label{S3eqN1}
\mathbb{P}\big{(} E(M) \big{)}  <  \epsilon/4\mbox{ and }\mathbb{P}\big{(} F(M) \big{)}  < \frac{\epsilon c(t)}{ 12},
\end{equation}
where we recall that $c(t) = \prod_{i = 1}^\infty (1 - t^i)$. 
Suppose $(a,b,s,\ell_{bot}) \in D(M)$ and observe that we have
\begin{equation}\label{EqK1}
\begin{split}
&\mathbb{P}^{-t_1, s; a,b}_{free} \left(\ell(-s_1) \geq  - ps_1  +MN^{\alpha/2}  \right) =  \mathbb{P}^{0, t_1 + s; 0,b - a}_{free} \left(\ell(t_1-s_1) + a\geq  -p s_1  +MN^{\alpha/2}  \right) \geq \\
&\mathbb{P}^{0, t_1 + s; 0,b - a}_{free} \left(\ell(t_1-s_1) \geq  p(t_1- s_1)  +2MN^{\alpha/2}  \right),  
\end{split}
\end{equation}
where in the last inequality we used that $a +pt_1 \geq -  M N^{\alpha/2}$. Since $|a +pt_1| \leq  M N^{\alpha/2}$ and $b - ps \geq (6r + 10)(M +1)N^{\alpha/2}$, we conclude that 
$b -a \geq p(t_1 + s) + (6r+9)(M+1)N^{\alpha/2}$. It follows from Lemma \ref{LemmaHalf} that for all large $N$ we have
\begin{equation}\label{EqK2}
\mathbb{P}^{0, t_1 + s; 0,b - a}_{free} \left(\ell(t_1-s_1) \geq  \frac{t_1 - s_1}{t_1 + s} [p(t_1 + s)  + (6r+9)(M+1)N^{\alpha/2}] - (t_1 + s)^{1/4}  \right) \geq 1/3.
\end{equation}
Notice that since $s \in [0, t_1]$, $s_1 =\lfloor r N^{\alpha} \rfloor$ and $t_1 = \lfloor (r+1)N^{\alpha} \rfloor$, we have $ \frac{t_1 - s_1}{t_1 + s} \geq \frac{1}{2r + 3}$ for all large $N$. These estimates together imply that for all large $N$ we have $\frac{t_1 - s_1}{t_1 + s} [p(t_1 + s)  + (6r+9)(M+1)N^{\alpha/2}] - (t_1 + s)^{1/4} \geq p(t_1- s_1) + 2MN^{\alpha/2}$ and so from (\ref{EqK1}) and (\ref{EqK2}) we conclude that
\begin{equation}\label{EqK3}
\mathbb{P}^{-t_1, s; a,b}_{free} \left(\ell(-s_1) \geq  - ps_1  +MN^{\alpha/2}  \right)  \geq 1/3.
\end{equation}

Since the sequence $\mathcal{L}^N$ is $(\alpha, p,  r+1)$-good, we know that for any $\ell \in \Omega(-t_1, s; a,b)$ we have
$$\mathbb{P}(L^N_1([-t_1,s]) = \ell| E(a,b,s,\ell_{bot})) = \frac{W_t(-t_1, s,\ell, \ell_2; S)}{Z_t(-t_1, s, a,b,\ell_{bot}; S)},$$
where $S = \llbracket-t_1 + 1,s\rrbracket$. The latter together with (\ref{EqK3}) and Corollary \ref{CorMon2} allow us to conclude that
\begin{equation}\label{PSeq1}
\mathbb{P} \left(L_1(-s_1) + ps_1 > MN^{\alpha/2} |E(a,b,s, \ell_{bot}) \right) \geq c(t)/3.
\end{equation}
 
We now observe that 
$$\mathbb{P}\big{(} F(M) \big{)} \geq  \sum_{(a,b,s,\ell_{bot}) \in D(M)} \mathbb{P}\big{(}F(M)\cap E(a,b,s,\ell_{bot})\big{)} =  $$
$$ = \sum_{(a,b,s,\ell_{bot}) \in D(M)} \mathbb{P}\big{(}F(M) \big{|}  E(a,b,s,\ell_{bot})\big{)} \mathbb{P}\big{(} E(a,b,s,\ell_{bot})\big{)} \geq(c(t)/3) \mathbb{P}\big{(}E^c(M) \cap G(M)\big{)},$$
where in the last inequality we used (\ref{PSeq1}). Combining the above inequality with the inequalities in (\ref{S3eqN1}) we see that for all large $N$ we have
\begin{equation}\label{EqK4}
\epsilon/2 > \mathbb{P}( G(M)) =  \mathbb{P}\left( \sup_{s \in [0,t_1]} \left[L_1^N(s) -p s \right] >  (6r + 10)(M+1)N^{\alpha/2}\right).
\end{equation}
A similar argument shows that for all large $N$ we have
\begin{equation}\label{EqK5}
\epsilon/2 >  \mathbb{P}\left( \sup_{s \in [-t_1,0]} \left[ L_1^N(s) -p s \right]>  (6r + 10)(M+1)N^{\alpha/2}\right).
\end{equation}
Combining (\ref{EqK4}) and (\ref{EqK5}) we conclude the statement of the lemma for $R = (6r + 10)(M+1) $.

\subsection{Proof of Lemma \ref{LemmaAP1}}\label{Section6.2}

 For clarity we will split the proof into two steps. 

{\raggedleft \bf Step 1.} Define $F = \left\{\min\left( \tilde{L}(-s_1) + p s_1, \tilde{L}(s_1) -ps_1 \right)\geq (M_2 + 1) N^{\alpha/2} + (2s_1)^{1/2}  \right\}$. We claim that for all $N$ sufficiently large we have
\begin{equation}\label{LAPeq4}
\mathbb{P}_{\tilde{L}}\left (F \right) \geq (c(t)^3/18)\left(1 - \Phi^{p(1-p)/2}\left(10(1 +r)^2(M_1 + M_2 + 1) \right) \right).
\end{equation}
Establishing the validity of (\ref{LAPeq4}) will be done in the second step, and in what follows we assume it is true and finish the proof of the lemma. 

We assert that if $N_2$ is sufficiently large and $N \geq N_2$ we have
\begin{equation}\label{LAPeqS}
F \subset A = \left \{  Z\left( -s_1,s_1,\tilde{L}(-s_1),\tilde{L}(s_1),{\ell}_{bot}; S' \right) > \frac{1}{4} \left( 1 - \exp \left( \frac{-2}{p (1-p)}\right) \right)  \right\}.
\end{equation}
Observe that (\ref{LAPeqS}) and (\ref{LAPeq4}) prove the lemma and so it suffices to verify (\ref{LAPeqS}). The details are presented below (see also Figure \ref{S6_1}). 
\begin{figure}[h]
\centering
\scalebox{0.66}{\includegraphics{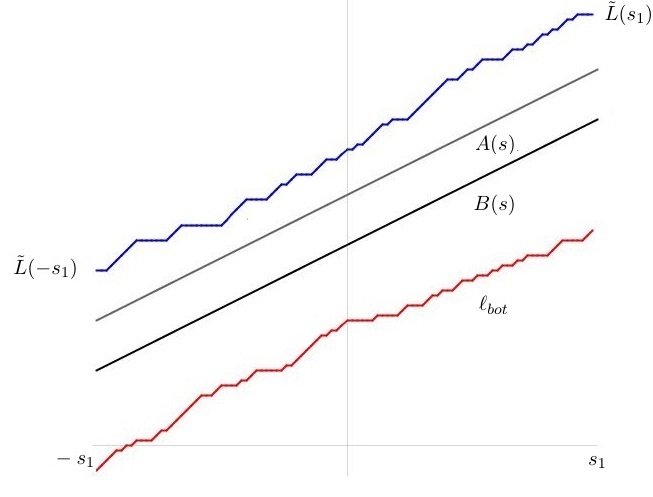}}
\caption{Overview of the arguments in Step 1: \\
 We want to prove that on the event $F$, we have a lower bound on the acceptance probability $Z_t(\tilde{L}(-s_1),\tilde{L}(s_1)) = Z_t( -s_1,s_1,\tilde{L}(-s_1),\tilde{L}(s_1),{\ell}_{bot};S')$. As explained in (\ref{away2}) the acceptance probability is just the average of the weights $W_t(-s_1,s_1, \ell, \ell_{bot}; S')$ over all up-right paths in $\Omega = \Omega(-s_1, s_1;  \tilde{L}(-s_1), \tilde{L}(s_1))$. Consequently, to show that $Z_t(\tilde{L}(-s_1),\tilde{L}(s_1))$ is lower-bounded it suffices to find a big subset $\Omega' \subset \Omega$, such that the weights $W_t(-s_1,s_1, \ell, \ell_{bot}; S')$ for $\ell \in \Omega'$ are lower-bounded.\\
 Let $A(s)$ and $B(s)$ denote the lines $ps + (M_2 + 1)N^{\alpha/2} - (2s_1)^{1/4}$ and $ps + M_2N^{\alpha/2}$, drawn in grey and black respectively above. Then $\Omega'$ denotes the set of up-right paths in $\Omega$, which lie above $A(s)$ on $[-s_1, s_1]$. On the event $F$ we have that $\tilde{L}(\pm s_1)$ are at least a distance $(2s_1)^{1/2} + (2s_1)^{1/4}$ above the points $A(\pm s_1)$ respectively. Since the endpoints of paths in $\Omega$ are well above those of $A(s)$ this means that some positive fraction of these paths will stay above $A(s)$ on the entire interval $[-s_1, s_1]$; i.e. $|\Omega'|/|\Omega|$ is lower bounded. This is what we mean by $\Omega'$ being big and the exact relation is given in (\ref{away3}).\\
To see that $W_t(-s_1,s_1, \ell, \ell_{bot}; S')$ for $\ell \in \Omega'$ are lower bounded, we notice that elements $\ell \in \Omega'$ are well-above $B(s)$, which dominates $\ell_{bot}$ by assumption. This means that $\ell$ is well above $\ell_{bot}$ and for such paths $W_t(-s_1,s_1, \ell, \ell_{bot}; S')$  is lower bounded. The exact relation is given in (\ref{away4}).
}\label{S6_1}
\end{figure}

From Definition \ref{DefHLGP} (see also Remark \ref{RemHLGP}) we have
$$Z\left( -s_1,s_1,\tilde{L}(-s_1),\tilde{L}(s_1),\ell_{bot};S'\right)  = \mathbb{E}_{free}^{-s_1, s_1; \tilde{L}(-s_1), \tilde{L}(s_1)} \left[W_t(-s_1,s_1, \cdot, \ell_{bot}; S') \right].$$
If we set $\Omega = \Omega(-s_1, s_1;  \tilde{L}(-s_1), \tilde{L}(s_1))$ and $Z_t(\tilde{L}(-s_1),\tilde{L}(s_1)) = Z\left(-s_1,s_1,\tilde{L}(-s_1),\tilde{L}(s_1),{\ell}_{bot};S'\right)$ then the above implies
\begin{equation}\label{away2}
Z_t(\tilde{L}(-s_1),\tilde{L}(s_1)) = | \Omega|^{-1} \sum_{ \ell \in \Omega} W_t(-s_1,s_1, \ell, \ell_{bot}; S').
\end{equation}
Denote $\Omega'  = \left \{ \ell \in  \Omega : \ell(s) - ps \geq (M_2 + 1) N^{\alpha/2} - (2s_1)^{1/4} \mbox{ for } s \in [-s_1, s_1]\right\}$. It follows from Lemma \ref{LemmaAway} that on the event $F$, provided $N_2$ is sufficiently large and $N \geq N_2$ we have
\begin{equation}\label{away3}
\frac{|\Omega'|}{|\Omega|} \geq \frac{1}{2} \left( 1 - \exp \left( \frac{-2}{p (1-p)}\right)\right).
\end{equation}
Since $|s_1 - rN^{\alpha}| < 1$ we know that for $N_2$ sufficiently large and $N \geq N_2$, we have that $\ell \in \Omega'$ satisfies $\ell(s) - ps \geq (M_2 + 1/2)N^{\alpha/2} \geq \ell_{bot}(s) - ps + (1/2)N^{\alpha/2}$, where the last inequality holds true by our assumption on $\ell_{bot}$. The conclusion is that for $\ell \in \Omega'$, we have that $\ell(s) - \ell_{bot}(s) \geq m$, where $m = (1/2) N^{\alpha/2}$. In view of (\ref{S1eq0}) we conclude that for $N_2$ sufficiently large, $N \geq N_2$ and $\ell \in \Omega'$, we have
\begin{equation}\label{away4}
W_t(-s_1,s_1, \ell, \ell_{bot}; S') \geq (1 - t^m)^{2s_1} \geq  (1 - t^{(1/2) N^{\alpha/2}})^{2rN^{\alpha}} \geq \frac{1}{2}.
\end{equation}
Combining (\ref{away2}), (\ref{away3}) and (\ref{away4}) we conclude that  provided $N_2$ is sufficiently large and $N \geq N_2$  on the event $F$ we have
$$Z_t(\tilde{L}(-s_1),\tilde{L}(s_1)) \geq  | \Omega|^{-1}\sum_{ \ell \in \Omega'} W_t(-s_1,s_1, \ell, \ell_{bot}; S') \geq  \frac{|\Omega'|}{|\Omega|} \cdot \frac{1}{2}  \geq \frac{1}{4} \left( 1 - \exp \left( \frac{-2}{p (1-p)}\right) \right)  .$$

{\raggedleft \bf Step 2.} In this step we prove (\ref{LAPeq4}). We refer the reader to Figure \ref{S6_2} for an overview of the main ideas in this step and a graphical representation of the notation we use.
\begin{figure}[h]
\centering
\scalebox{0.7}{\includegraphics{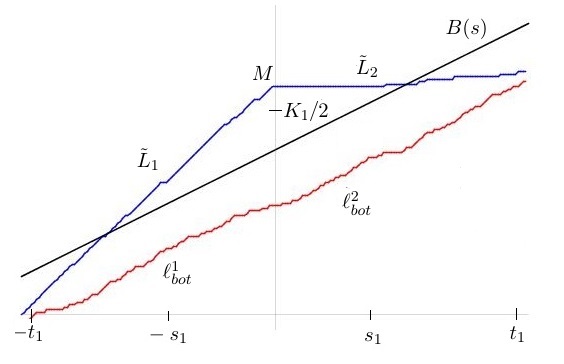}}
\caption{Overview of the arguments in Step 2: \\
$\tilde{L}_1$ and $\tilde{L}_2$ are the restrictions of $\tilde L$ to $[-t_1, 0]$ and $[0, t_1]$ respectively. $\ell^1_{bot}$ and $\ell^2_{bot}$ denote the restrictions of $\ell_{bot}$ to $[-t_1, 0]$ and $[0, t_1]$ respectively.  Let $B(s)$ denote the line $ps + M_2N^{\alpha/2}$, drawn in black above. We have that $F$ denotes the event that $\tilde{L}$ is at least a distance $ N^{\alpha/2} + (2s_1)^{1/2}$ above the line $B(s)$ at the points $\pm s_1$ and we want to find a lower bound on $\mathbb{P}_{\tilde L}(F)$. \\
We first let $E$ denote the event that $\tilde{L}(0)$ is much higher than $B(0)$, and prove that $\mathbb{P}_{\tilde L}(E)$ is lower bounded. The exact statement is given in (\ref{LAP1eq1}). Afterwards, we show that on the event that the midpoint $\tilde{L}(0)$ is very high, the points $\tilde{L}(\pm s_1)$ are also very high with positive probability. The exact statement is given in (\ref{LAPeqS3}). \\
In a sense, by conditioning on the midpoint $\tilde{L}(0)$ we split our problem into two independent subproblems for the left and right half of $\tilde L$  - see (\ref{LAPeqS2}). Establishing the required estimates for each of the subproblems is then a relatively straightforward application of Lemma \ref{LemmaHalf} and Corollary \ref{CorMon2} - see (\ref{LAP1eq2}).
}\label{S6_2}
\end{figure}

Let $K_1 = 8(1+r)^2(M_1 + M_2 + 1)N^{\alpha/2}$.  Define $E= \sqcup_{M \in X} E_M$ for
$$E_M =  \{ \tilde{L}(0)   = M \} \mbox{ and } X = \left\{ M \in \mathbb{N}: M \geq  (1/2)K_1 -  [2(r+1)N]^{\alpha/4} \mbox{ and } \mathbb{P}_{\tilde L}(E_M) > 0 \right\}.$$
It follows from Lemma \ref{LemmaTail} that we can find $N_2$, depending on $r,M_1,M_2$ such that for $N \geq N_2$ we have
$$\mathbb{P}^{-t_1,t_1;a,b}_{free}\left( \ell(0)  \geq (1/2)K_1 - [2(r+1)N]^{\alpha/4} \right) \geq (1/2) (1 - \Phi^{p(1-p)/2}(M_1 + K_1) ).$$
Then by Corollary \ref{CorMon2} we conclude 
\begin{equation}\label{LAP1eq1}
\mathbb{P}_{\tilde{L}}\left (E  \right) \geq (c(t)/2)(1 - \Phi^{p(1-p)/2}(M_1 + K_1) ).
\end{equation}

Denote by $\tilde{L}_1$ and $\tilde{L}_2$ the restriction of $\tilde L$ to $[-t_1, 0]$ and $[0, t_1]$ respectively. Similarly, we let $\ell^1_{bot}$ and $\ell^2_{bot}$ denote the restriction of $\ell_{bot}$ to $[-t_1, 0]$ and $[0, t_1]$ respectively.
The key observation we make is that if $M \in X$ then
\begin{equation}\label{LAPeqS2}
\mathbb{P}_{\tilde{L}} (\tilde{L}_1 = \ell_1, \tilde{L}_2 = \ell_2| E_M) =  \mathbb{P}_{S_1}^{-t_1, 0, a, M}(\ell_1| \ell^1_{bot})\cdot \mathbb{P}_{S_2}^{0, t_1, M, b}(\ell_2| \ell^2_{bot}),
\end{equation}
where $S_1 = \llbracket -t_1 + 1,  -s_1\rrbracket$, $S_2 =\llbracket s_1 + 1, t_1\rrbracket$ and $\ell_1 \in \Omega(-t_1, 0; a, M)$, $\ell_2 \in \Omega(0, t_1; M, b)$.

From Lemma \ref{LemmaHalf}, we know that 
$$\mathbb{P}_{free}^{-t_1,0; a,M}\left(\ell(-s_1) \geq M \frac{t_1 - s_1}{t_1} + a \frac{s_1}{t_1}   - [(r+1)N]^{\alpha/4} \right) \geq 1/3,$$
provided $N_2$ is large enough and $N \geq N_2$. Since $a \geq -p t_1- M_1 N^{\alpha/2}$, $s_1 =\lfloor r N^{\alpha} \rfloor$, $t_1 = \lfloor (r+1)N^{\alpha} \rfloor$, $M \geq (1/2) K_1 - [2(r+1)N]^{\alpha/4}$ and $K_1 = 8(1+r)^2(1 + M_1 + M_2)N^{\alpha/4}$ we conclude that if $N_2$ is sufficiently large and $N \geq N_2$ then
$$\mathbb{P}_{free}^{-t_1,0; a,M}\left(\ell(-s_1) + ps_1  \geq (M_2 + 1) N^{\alpha/2} + (2s_1)^{1/2}  \right) \geq 1/3.$$
From Corollary \ref{CorMon2} and the above inequality we conclude 
\begin{equation}\label{LAP1eq2}
 \mathbb{P}_{S_1}^{-t_1, 0, a, M}\left (\ell_1(-s_1)  + ps_1  \geq (M_2 + 1) N^{\alpha/2} + (2s_1)^{1/2}  \right) \geq c(t)/3.
\end{equation}
Similar arguments show that 
\begin{equation}\label{LAP1eq3}
 \mathbb{P}_{S_2}^{0, t_1, M, b}\left (\ell_2(s_1) - ps_1  \geq(M_2 + 1) N^{\alpha/2} + (2s_1)^{1/2} \right) \geq c(t)/3.
\end{equation}
Combining (\ref{LAPeqS2}), (\ref{LAP1eq2}) and (\ref{LAP1eq3}), we see that for $M \in X$, we have
\begin{equation}\label{LAPeqS3}
\mathbb{P}_{\tilde{L}} \left(F | E_M \right) \geq c(t)^2/9.
\end{equation}
The above inequality implies that
$$\mathbb{P}_{\tilde{L}}(F) \geq \mathbb{P}_{\tilde{L}} \left(F \cap E \right) = \sum_{M \in X} \mathbb{P}_{\tilde{L}} \left(F | E_M\right) \mathbb{P}_{\tilde{L}}(E_M)  \geq  (c(t)^2/9) \cdot \mathbb{P}_{\tilde{L}}\left (E  \right) .$$
The latter inequality together with (\ref{LAP1eq1}) and the monotonicity of $\Phi^v$ on $\mathbb{R}_{> 0}$ prove (\ref{LAPeq4}).

\subsection{Proof of Lemma \ref{LemmaAP2}}\label{Section6.3}

Define $\mathbb{P}_{{L}'}$ and $\mathbb{P}_{\tilde{{L}}'}$ as the measure on up-right paths ${L}'$ and $\tilde{{L}}': [-t_1,-s_1] \cup [s_1 ,t_1] \rightarrow \mathbb{R}$ (with ${L}'(-t_1) = \tilde{{L}}'(-t_1) = a$ and ${L}'(t_1) = \tilde{{L}}'(t_1) = b$ ) induced by the restriction of the measures $\mathbb{P}_{{L}}$ and $\mathbb{P}_{\tilde{{L}}}$ to these intervals. The Radon-Nikodym derivative between these two restricted measures is given on up-right paths $B: [-t_1,-s_1] \cup [s_1,t_1] \rightarrow \mathbb{R}$ by
\begin{equation}\label{eqn61}
\frac{d\mathbb{P}_{{L}'}}{d\mathbb{P}_{\tilde{{L}}'}} (B) = (Z')^{-1}Z_t( -s_1, s_1,B(-s_1),B(s_1),\ell_{bot}; S') ,
\end{equation}
where $Z' = \mathbb{E}_{\tilde{{L}}'} \left[ Z_t( -s_1, s_1,B(-s_1),B(s_1),\ell_{bot};S' ) \right]$. 

Observe that $Z_t(-s_1, s_1,B(-s_1),B(s_1),\ell_{bot} ;S')$ is a (deterministic) function of $\left( B(-s_1),B(s_1)\right)$. In addition, the law of $\left( B(-s_1),B(s_1)\right)$ under $\mathbb{P}_{\tilde L'}$ is the same as $(\tilde{L}(-s_1), \tilde{L}(s_1))$ under $\mathbb{P}_{\tilde L}$ (this is because $\mathbb{P}_{\tilde L'}$ is a restriction of $\mathbb{P}_{\tilde L}$ to intervals that contain $\pm s_1$).  The latter and Lemma \ref{LemmaAP1} imply
$$Z' = \mathbb{E}_{\tilde{{L}}'} \left[ Z_t( -s_1, s_1,B(-s_1),B(s_1),\ell_{bot};S' ) \right] = \mathbb{E}_{\tilde{{L}}} \left[ Z_t( -s_1, s_1,\tilde{L}(-s_1),\tilde{L}(s_1),\ell_{bot};S' ) \right] \geq g h.$$

Similarly, the law of $\left( B(-s_1),B(s_1)\right)$ under $\mathbb{P}_{ L'}$ is the same as $\left({L}(-s_1), {L}(s_1)\right)$ under $\mathbb{P}_{ L}$ (this is because $\mathbb{P}_{ L'}$ is a restriction of $\mathbb{P}_{ L}$ to intervals that contain $\pm s_1$). Since $Z_t( -s_1, s_1,B(-s_1),B(s_1),\ell_{bot};S' )$ is a (deterministic) function of $\left( B(-s_1),B(s_1)\right)$, we conclude that 
$$\mathbb{P}_{{{L}}} \left( Z_t( -s_1, s_1,L(-s_1),L(s_1),\ell_{bot} ;S') \leq \delta(\tilde \epsilon) \right) = \mathbb{P}_{{{L}'}} \left( Z_t( -s_1, s_1,B(-s_1),B(s_1),\ell_{bot};S' )\leq \delta(\tilde \epsilon)\right).$$
Let us denote $E = \left\{ Z_t( -s_1, s_1,B(-s_1),B(s_1),\ell_{bot};S' ) \leq \delta( \tilde \epsilon) \right\} \subset \Omega$ (here $\Omega$ is the space of paths $B$). Then we have that 
$$\mathbb{P}_{{{L}'}} \left( E\right) = \int_\Omega {\bf 1}_E \cdot d\mathbb{P}_{{L}'}(B) = (Z')^{-1}\int_\Omega {\bf 1}_E \cdot Z_t(-s_1, s_1,B(-s_1),B(s_1),\ell_{bot} ;S') \cdot d\mathbb{P}_{\tilde{{L}}'}(B).$$
The above is immediate from (\ref{eqn61}). On $E$ we have that $Z_t(-s_1, s_1,B(-s_1),B(s_1),\ell_{bot};S' ) \leq \delta(\tilde \epsilon)$ and so the above is bounded by
$$(Z')^{-1}\int_\Omega {\bf 1}_E \cdot\delta(\tilde\epsilon) \cdot d\mathbb{P}_{\tilde{{L}}'}(B) \leq  \frac{1}{gh} \int_\Omega {\bf 1}_E \cdot\delta(\tilde\epsilon) \cdot d\mathbb{P}_{\tilde{{L}}'}(B)  \leq \tilde\epsilon.$$
The first inequality used that $Z'  \geq  gh $ and the second one that $\delta(\tilde\epsilon) = \tilde\epsilon \cdot gh$ and ${\bf 1}_E \leq 1$. This concludes the proof of the lemma.

\section{Absolute continuity with respect to Brownian bridges }\label{Section7}
In Theorem \ref{PropTightGood} we showed that under suitable shifts and scalings $(\alpha, p, r+1)$-good sequences give rise to tight sequences of continuous random curves. In this section, we aim to obtain some qualitative information about their subsequential limits and we will show that any subsequential limit is absolutely continuous with respect to a Brownian bridge with appropriate variance. In particular, this demonstrates that we have non-trivial limits and do not kill fluctuations with our rescaling. In Section \ref{Section7.1} we present the main result of the section -- Theorem \ref{ACBB} and explain how it relates to the other results in the paper. The proof of Theorem \ref{ACBB} is given in Section \ref{Section7.2} and for the most part relies on our control of the acceptance probability in Proposition \ref{PropMain} and the Hall-Littlewood Gibbs property.

\subsection{Formulation of result and applications}\label{Section7.1} We begin by introducing some relevant notation and defining what it means for a random curve to be absolutely continuous with respect to a Brownian bridge.

\begin{definition} Let $X = C([0,1])$ and $Y = C([-r,r])$ be the spaces of continuous functions on $[0,1]$ and $[-r,r]$ respectively with the uniform topology. Denote by $d_X$ and $d_Y$ the metrics on the two spaces and by $\mathcal{B}(X)$ and $\mathcal{B}(Y)$ their Borel $\sigma$-algebras. Given $z_1,z_2 \in \mathbb{R}$ we define
$F_{z_1, z_2} : X \rightarrow Y$ and $G_{z_1,z_2}: Y \rightarrow X$ by
\begin{equation}\label{G1}
[F_{z_1,z_2} (g)] (x) =  z_1 + g \left( \frac{x + r}{2r} \right) + \frac{x + r}{2r} (z_2 - z_1) \hspace{3mm} [G_{z_1,z_2} (h) ](\xi) =  h \left( 2r\xi - r \right) -z_1 - (z_2 - z_1)  \xi,
\end{equation}
for $x \in [-r,r]$ and $\xi \in [0,1]$. 
\end{definition} 

One observes that $F_{z_1,z_2}$ and $G_{z_1,z_2}$ are bijective homomorphisms between $X$ and $Y$ that are mutual inverses. Let $X_0 = \{ f \in X : f(0) = f(1) = 0\}$ with the subspace topology and define $G: Y \rightarrow X$ through $G(h) = G_{h(-r), h(r)} ( h)$. Let us make some observations.
\begin{enumerate}
\item $G$ is a continuous function. Indeed, from the triangle inequality we have \\$d_X \left(G_{h_1(-r), h_1(r) }(h_1),  G_{h_2(-r), h_2(r) }(h_2) \right) \leq 2d_Y(h_1, h_2).$
\item If $L$ is a random variable in $(Y, \mathcal{B}(Y))$ then $G(L)$ is a random variable in $(X, \mathcal{B}(X))$, which belongs to $X_0$ with probability $1$. The measurability of $G(L)$ follows from the continuity of $G$, everything else is clearly true.
\end{enumerate}
Recall from Section \ref{Section4.2} that $B^\sigma$ stands for the Brownian bridge on $[0,1]$, with variance $\sigma^2$ -- this is a random variable in $(X, \mathcal{B}(X))$, which belongs to $X_0$ with probability $1$.

With the above notation we make the following definition. 
\begin{definition}\label{DACB}
Let $L$ be a random variable in $(Y, \mathcal{B}(Y))$ with law $\mathbb{P}_L$. We say that $L$ is {\em absolutely continuous } with respect to a Brownian bridge with variance $\sigma^2$ if for any $K \in \mathcal{B}(X)$ we have
$$\mathbb{P}(B^\sigma \in K) = 0 \implies \mathbb{P}_L(G(L) \in K) = 0.$$
\end{definition}

The main result of this section is as follows. 
\begin{theorem}\label{ACBB}
Assume the same notation as in Theorem \ref{PropTightGood} and let $\mathbb{P}_\infty$ be any subsequential limit of $\mathbb{P}_N$. If $f_\infty$ has law $\mathbb{P}_\infty$ then it is absolutely continuous with respect to a Brownian bridge with variance $2rp(1-p)$ in the sense of Definition \ref{DACB}.
\end{theorem}

We have the following immediate corollary to Theorem \ref{ACBB} about the three stochastic models of Section \ref{Section2}.
\begin{corollary}\label{ACC}
Assume the same notation as in Theorem \ref{HLCT}, Corollary \ref{SVTight} and Theorem \ref{ASEPCT} respectively and define for $x \in [-r,r]$
$$g^{HL}_N(x) =  \sigma_{\mu}f^{HL}_N(x) + \frac{x^2f_1''(\mu)}{2}, \hspace{2mm} g^{SV}_N(x) =  \sigma_{\mu}f^{SV}_N(x) -  \frac{x^2 f_2''(\mu)}{2}, \hspace{2mm} g^{ASEP}_N(x) =\sigma_{\alpha} f^{ASEP}_N(x) -  \frac{x^2 f_3''(\alpha)}{2 }.$$
If $g^{HL}_\infty$, $g^{SV}_\infty$ and $g^{ASEP}_\infty$ are any subsequential limits of $g^{HL}_N $, $g^{SV}_N $ and $g^{ASEP}_N $ respectively as $N \rightarrow \infty$ then $g^{HL}_\infty$, $g^{SV}_\infty$ and $g^{ASEP}_\infty$ are absolutely continuous with respect to a Brownian bridge of variance $2r  f_1'(\mu) [1 - f_1'(\mu)]$,$- 2r f'_3(\mu) [1 + f'_2(\mu)]$ and  $-2r   f'_3(\alpha) [1 + f'_3(\alpha)]$ respectively in the sense of Definition \ref{DACB}.
\end{corollary}
\begin{proof}
From the proof of Theorem \ref{HLTight} we know that 
$$g_N^{HL}(s) = N^{-1/3}\left(L^N_1(s N^{2/3}) - f'_1(\mu) s N^{2/3} \right ), \mbox{ for $s\in [-r,r]$},$$
where the sequence $(L_1^N,L_2^N) $  is $(2/3, f_1'(\mu), r+1)$-good. By Theorem \ref{ACBB} we conclude the statement for $g^{HL}_\infty$. In addition, by Theorem \ref{SVHL} we know that for each $N \in \mathbb{N}$,  $f^{HL}_N$ has the same distribution as $f^{SV}_N$ and so we conclude the statement for $g^{SV}_\infty$ as well.

From the proof of Theorem \ref{ASEPTight} we know that
$$g_N^{ASEP}(s) = N^{-1/3}\left(\tilde L_1^N(s N^{2/3})  + f_3'(\alpha) s N^{2/3} \right ), \mbox{ for $s\in [-r,r]$},$$
where the sequence $(\tilde L_1^N,\tilde L_2^N) $  is $\left(2/3, -f'_3(\alpha), r+1\right)$-good. By Theorem \ref{ACBB} we conclude the statement for $g^{ASEP}_\infty$. 
\end{proof}
\begin{remark} Conjecturally, $f^{HL}_N$, $f^{SV}_N$ and $f^{ASEP}_N$ should converge to the Airy$_2$ process. Corollary \ref{ACC} provides further evidence for this result as it is known that the Airy$_2$ process minus a parabola has Brownian paths \cite{CorHamA}. See also the discussion at the end of Section \ref{sec.HLGLE}.
\end{remark}

\subsection{Proof of Theorem \ref{ACBB}}\label{Section7.2} In this section we give the proof of Theorem \ref{ACBB}, which for clarity is split into several steps. Before we go into the main argument we introduce some useful notation and give an outline of our main ideas.\\

Throughout we assume we have the same notation as in the statement of Theorem \ref{PropTightGood} as well as the notation from Section \ref{Section7.1} above. We  Denote $\sigma_1^2 = 2r p(1-p)$, $s_1 = \lfloor r N^{\alpha} \rfloor$, $r_N = s_1N^{-\alpha}$ and $S' = \llbracket - s_1+ 1, s_1\rrbracket$.  In addition, we define three probability spaces $\mathbb{P}^1, \mathbb{P}^2, \mathbb{P}^3$ as well as a big probability space $\tilde{\mathbb{P}}$, which is the product space of $\mathbb{P}^1, \mathbb{P}^2$ and $ \mathbb{P}^3$.  The three spaces will carry different stochastic objects and we will use the superscript to emphasize, which properties we are using in different steps of the proof. We also reserve $\mathbb{P}$ to refer to the law of universal probabilistic objects like a Brownian bridge of a fixed variance.

From Theorem \ref{KMT} in the paper, we know that for each $n \in \mathbb{N}$ we have a probability space, on which we have a Brownian bridge $B^\sigma$ with variance $\sigma^2 = p(1-p)$ and a family of random paths $\ell^{(n,z)} \in \Omega(0,n; 0, z)$ for $z = 0,\dots,n$ such that $\ell^{(n,z)}$ has law $\mathbb{P}^{0,n;0,z}_{free}$ and
\begin{equation*}
\mathbb{E}\left[ e^{a \Delta(n,z)} \right] \leq C e^{\alpha (\log n)^2}e^{|z- p n|^2/n}, \mbox{ where $\Delta(n,z)=  \sup_{0 \leq t \leq n} \left| \sqrt{n} B^\sigma_{t/n} + \frac{t}{n}z - \ell^{(n,z)}(t) \right|,$}
\end{equation*}
where the constants $C,a, \alpha$ depend only on $p$ and are fixed. By taking products of countably many of the above spaces we can construct a probability space $(\Omega^1, \mathcal{F}^1, \mathbb{P}^1)$, on which we have defined independent Brownian bridges $B^{\sigma,k,n}$ and independent families of random paths $\ell^{(n,k,z)} \in \Omega(0,n; 0, z)$ for $z = 0,\dots,n$ such that $\ell^{(n,k,z)}$ has law $\mathbb{P}^{0,n;0,z}_{free}$ for each $k$ and
\begin{equation*}
\mathbb{E}_{\mathbb{P}^1}\left[ e^{a \Delta(n,k,z)} \right] \leq C e^{\alpha (\log n)^2}e^{|z- p n|^2/n}, \mbox{ where $\Delta(n,k,z):=  \sup_{0 \leq t \leq n} \left| \sqrt{n} B^{\sigma,k,n}_{t/n} + \frac{t}{n}z - \ell^{(n,k,z)}(t) \right|.$}
\end{equation*}
In words, for each pair $(k,n) \in \mathbb{N} \times \mathbb{N}$ we have an independent copy of the probability space afforded by Theorem \ref{KMT} sitting inside $(\Omega^1, \mathcal{F}^1, \mathbb{P}^1)$.  In addition, we assume that $(\Omega^1, \mathcal{F}^1, \mathbb{P}^1)$ carries a uniform random variable $U \in (0,1)$, which is independent of all else.

Since $\mathbb{P}_\infty$ is a subseqential limit of $\mathbb{P}_N$ we know that we can find an increasing sequence $N_j$ such that $\mathbb{P}_{N_j}$ weakly converge to $\mathbb{P}_\infty$. By Skorohod's representation theorem (see e.g. Theorem 6.7 in \cite{Bill}) we can find a probability space $(\Omega^2, \mathcal{F}^2, \mathbb{P}^2)$, on which are defined random variables $\tilde{f}_{N_j}$ and $\tilde{f}_\infty$ that take values in $(Y, \mathcal{B}(Y))$ such that the laws of $\tilde{f}_{N_j}$ and $\tilde{f}_{\infty}$ are $\mathbb{P}_{N_j}$ and $\mathbb{P}_\infty$ respectively and such that $d_Y\left(\tilde{f}_{N_j}(\omega^2),\tilde{f}_{\infty}(\omega^2) \right) \rightarrow 0 $ as $j \rightarrow \infty$ for each $\omega^2 \in \Omega^2$.

We consider a probability space $(\Omega^3, \mathcal{F}^3, \mathbb{P}^3)$, on which we have defined the original $(\alpha, p, r+1)$-good sequence $\mathfrak{L}^N = (L_1^N, L_2^N)$ and so
$$f_N(s) = N^{-\alpha/2}(L_1^N(sN^{\alpha}) - p s N^{\alpha}), \mbox{ for $s\in [-r,r]$}$$
has law $\mathbb{P}_{N}$ for each $N \geq 1$. Let us briefly explain the difference between $\mathbb{P}^2$ and $\mathbb{P}^3$ and why we need both. The space  $(\Omega^2, \mathcal{F}^2, \mathbb{P}^2)$ carries the random variables $\tilde{f}_{N_j}$ of law $\mathbb{P}_{N_j}$ and what is crucial is that the latter converge {\em almost surely} to $\tilde{f}_\infty$, whose law is $\mathbb{P}_\infty$. The space $(\Omega^3, \mathcal{F}^3, \mathbb{P}^3)$ carries the {\em entire} discrete line ensembles  $\mathfrak{L}^N = (L_1^N, L_2^N)$ (and not just the top curve), which is needed to perform the resampling procedure of Section \ref{Section3.1}. Finally, we define  $(\tilde \Omega, \tilde{ \mathcal{F}}, \tilde{\mathbb{P}})$ as the product of the three probability spaces we defined above.\\

At this time we give a brief outline of the steps in our proof. In the first step we fix $K \in \mathcal{B}(X)$ such that $\mathbb{P}(B^{\sigma_1} \in K) = 0$ and find an open set $O$, which contains $K$, and such that $B^{\sigma_1}$ is {\em extremely} unlikely to belong to $O$. Our goal is then to show that $G(\tilde{f}_\infty)$ is also unlikely to belong to $O$, the exact statement is given in (\ref{Z3}) below. Using that $O$ is open and that $\tilde{f}_{N_j}$ converge to $\tilde{f}_\infty$ almost surely we can reduce our goal to showing that it is unlikely that $G(\tilde{f}_{N_j})$ belongs to $O$ {\em and } $\tilde{f}_{N_j}$ is at least a small distance away from the complement of $G^{-1}(O)$ for large $j$. Our gain from the almost sure convergence is that  we have bounded ourselves away from $G^{-1}(O)^c$, which implies that by performing small perturbations we do not leave $G^{-1}(O)$. As the laws of $\tilde{f}_{N_j}$ and $f_{N_j}$ are the same we can switch from $(\Omega^2, \mathcal{F}^2, \mathbb{P}^2)$ to $(\Omega^3, \mathcal{F}^3, \mathbb{P}^3)$, reducing the goal to showing that it is unlikely that $G(f_{N})$ belongs to $O$ and $f_{N}$ is at least a small distance away from the complement of $G^{-1}(O)$ for large $N$. The exact statement is given in (\ref{Z4}) and the reduction happens in Step 2. The benefit of this switch is that we can perform the resampling of Section \ref{Section3.1} in $(\Omega^3, \mathcal{F}^3, \mathbb{P}^3)$ as the latter carries an entire line ensemble.

In the third step we use $U$ and $\ell^{(2s_1,k,z)}$ for $k = 1,2,3...$ to resample $f_N$ on the interval $[-s_1, s_1]$. If we denote by $Q$ the index $k$ of the first line we accept from the resampling, we can rephrase our statements for $f_N$ to equivalent statements that involve the path $\ell^{(2s_1,Q,z)}$ -- this is (\ref{Z5}). The benefit of working with $\ell^{(2s_1,k,z)}$ is that they are already strongly coupled with Brownian bridges by construction. In Step 4 we construct an event $F(N)$, on which our coupling of $\ell^{(2s_1,Q,z)}$ and the Browniand bridge $B^{\sigma, Q, 2s_1}$ is good and on which $B^{\sigma, Q, 2s_1}$ is well-behaved (its supremum and modulus of continuity are controlled). Provided we are on $F(N)$ (where the coupling is good) we see that $\ell^{(2s_1,Q,z)}$ belonging to a certain set (we want to show is unlikely) implies that $\sqrt{2r}B^{\sigma, Q, 2s_1}$ belongs to $O$. Here it is crucial, that we have the extra distance to the complement of $G^{-1}(O)$ so that when we approximate our discrete paths with Brownian bridges we do not leave the set $G^{-1}(O)$. 

The above steps reduce the problem to showing that  it is unlikely that $\sqrt{2r}B^{\sigma, Q, 2s_1}$ belongs to $O$ or that we are outside the event $F(N)$ -- the exact statement is in (\ref{Z7}). The control of $\sqrt{2r}B^{\sigma, Q, 2s_1}$ is obtained by arguing that with high probability $Q$ is bounded -- this requires our estimate on the acceptance probability from Proposition \ref{PropMain} and is the focus of Step 5. By having $Q$ bounded we reduce the question to a regular Brownian bridge, for which the event it belongs to $O$ is unlikely by definition of $O$. We demonstrate that $F(N)^c$ is unlikely in Step 6. As before  we use the estimate on the acceptance probability to reduce the question to one involving a regular Brownian bridge. In addition, we use that with high probability we have uniform control of the coupling of our paths with Brownian bridges for all large $N$.

We now turn to the proof of the theorem.\\

{\raggedleft {\bf Step 1.}} Suppose that $K \in \mathcal{B}(X)$ is given such that $\mathbb{P}(B^{\sigma_1} \in K) = 0$. We wish to show that 
\begin{equation}\label{Z1}
\mathbb{P}^2\left( G(\tilde{f}_\infty) \in K \right) = 0.
\end{equation}
Let $\epsilon \in (0,1)$ be given and note that by Proposition \ref{PropMain} and Theorem \ref{PropTightGood} , we can find $\delta \in (0,1)$ and $ M > 0$ such that for all large $N$ one has
\begin{equation}\label{Z2}
\begin{split}
\mathbb{P}^3\left( E(\delta, M, N) \right) < \epsilon, \mbox{ where } E(\delta, M, N) =  &\Big\{  Z_t( -s_1,s_1,L_1^N(- s_1),L_1^N(s_1), L_2; S') < \delta   \Big\} \cup \\
&  \Big\{ \sup_{s \in [-rN^{\alpha},rN^{\alpha}]}\left| L^N_1(s) - p s \right| \geq   MN^{\alpha/2} \Big\}.
\end{split}
\end{equation}
We observe that since $C([-r,r])$ is a metric space we have by Theorem II.2.1 in \cite{Parth}  that the measure of $B^{\sigma_1}$ is outer-regular. In particular, we can find an open set $O$ such that $K \subset O$ and $\mathbb{P}(B^{\sigma_1} \in O) < \epsilon \cdot \frac{\log (1 - \delta)}{\log (\epsilon)}$. The set $O$ will not be constructed explicitly and we will not require other properties from it other than it is open and contains $K$.

We will show that 
\begin{equation}\label{Z3}
 \mathbb{P}^2\left( G(\tilde{f}_\infty) \in O \right) \leq 6\epsilon.
\end{equation}
Notice that the above implies that $\mathbb{P}^2\left( G(\tilde{f}_\infty) \in K \right) \leq 6\epsilon$ and hence we reduce the proof of the theorem to establishing (\ref{Z3}).\\

{\raggedleft {\bf Step 2.}} Our goal in this step is to reduce (\ref{Z3}) to a statement involving finite indexed curves. 

We first observe $G^{-1}(O)$ is open since $G$ is continuous (this was proved in Section \ref{Section7.1}). The latter implies that 
\begin{equation*}
\begin{split}
 &\mathbb{P}^2\left( G(\tilde{f}_\infty) \in O \right) = \mathbb{P}^2\left( \tilde{f}_\infty \in G^{-1}(O) \right)  =\lim_{j \rightarrow \infty} \mathbb{P}^2 \left( \left\{\tilde{f}_{N_j} \in G^{-1}(O) \right\} \cap  \left\{ d_Y(\tilde{f}_{N_j}, G^{-1}(O)^c) > r_j\right\} \right),
\end{split}
\end{equation*}
where $r_j$ is any sequence that converges to $0$  as $j \rightarrow \infty$. The first equality is by definition. The second one follows from the fact that $\tilde{f}_{N_j}$ converge to $\tilde{f}_\infty$  in the uniform topology $\mathbb{P}^2$-almost surely and that $G^{-1}(O)$ is open. To be more specific we take $r_j = N_j^{-\alpha/8}$ for the sequel.

Since $f_N$ has law $\mathbb{P}_{N}$ for each $N \geq 1$,  we observe that to get (\ref{Z3}) it suffices to show that 
\begin{equation}\label{Z4}
\limsup_{N \rightarrow \infty}  \mathbb{P}^3\left(  \left\{ f_{N} \in G^{-1}(O) \right\} \cap \left \{ d_Y(f_{N}, G^{-1}(O)^c) > N^{-\alpha/8} \right\} \right) \leq 6\epsilon .
\end{equation}

{\raggedleft {\bf Step 3.}} At this time we recall the resampling procedure from Remark \ref{RemHLGP} in the setting of our probability spaces $\mathbb{P}^i$. The goal of this step is to rephrase (\ref{Z4}) into a statement involving the paths $\ell^{(n,k,z)}$ that are defined on $(\Omega^1, \mathcal{F}^1, \mathbb{P}^1)$. 

Denote by $a = L^N_1(-s_1)$, $b = L^N_1(s_1)$, $z = b-a$, $n = 2s_1$ and $\ell_{bot} = L^N_2$ restricted to $[-s_1, s_1]$. We resample the top curve $L^N_1$ as follows. We start by erasing the curve in the interval $[-s_1,s_1]$. For $k = 1,2,3,\dots$ we take $\ell^{(n,k,z)}$ (these were defined on the space $(\Omega^1, \mathcal{F}^1, \mathbb{P}^1)$), check if $W_t(-s_1, s_1, (-s_1,a) + \ell^{(n,k,z)}, \ell_{bot}; S') > U$ and set $Q$ to be the minimal index $k$, which satisfies the inequality. Here $(-s_1,a) + \ell^{(n,k,z)}$ is just the up-right path $\ell^{(n,k,z)}$ shifted so that it starts from the point $(-s_1,a)$. 

Notice that by construction the path $(-s_1,a) + \ell^{(n,k,z)}$ are independent identically distributed as $\mathbb{P}_{free}^{-s_1,s_1;a,b}$.  Because $\mathfrak{L}^N$ satisfies the Hall-Littlewood Gibbs property we have
\begin{equation}
 \tilde{\mathbb{P} }\left((-s_1, a) + \ell^{(2s_1,Q,b-a)}= \ell \right) = \mathbb{P}^{3}\left(  L^N_1[-s_1,s_1] = \ell \right),
\end{equation}
for every $\ell \in \cup_{z_1 \leq z_2} \Omega(-s_1, s_1; z_1, z_2)$ where $L^N_1[-s_1,s_1] $ stands for the restriction of $L_1^N$ to the interval $[-s_1,s_1]$. If we denote 
$$h_N(s) = \begin{cases} N^{-\alpha/2}\left( a+ \ell^{(2s_1,Q,b-a)}(sN^{\alpha} + s_1)\right), &\mbox{ for $s\in [-r_N, r_N ] $} \\ f_N(s) &\mbox{ for $ s \in [-r,r] \backslash [-r_N, r_N ]$,}\end{cases}$$
we have that $h_N$ has the same law as $f_N$. Consequently it suffices to show that 
\begin{equation}\label{Z5}
\limsup_{N \rightarrow \infty}  \tilde{\mathbb{P}}\left(  \left\{ h_{N} \in G^{-1}(O) \right\} \cap \left \{ d_Y(h_{N}, G^{-1}(O)^c) > N^{-\alpha/8} \right\} \right) \leq 6\epsilon .
\end{equation}

{\raggedleft {\bf Step 4.}} Let $B^k(s):= B^{\sigma,k,2s_1}_s$ for $s \in[0,1]$ and consider the event
\begin{equation}\label{DefF}
F(N) = \left\{ \Delta\left(2s_1,Q, b - a\right) < N^{\alpha/4}\right\} \cap \left\{ \sup_{s \in [0,1]} \left|B^Q(s) \right| \leq N^{\alpha/4} \right\} \cap \left\{ w(B^Q, N^{-\alpha}) \leq N^{-\alpha/4} \right\}.
\end{equation}
In the above $w$ stands for the modulus of continuity of a function on $[0,1]$ as defined in (\ref{MOC}). \\

{\raggedleft In this step we verify the following statement: There exists $N_0 \in \mathbb{N}$ and $C$ both depending on $r$ such that for $N \geq N_0$ and on the event $F(N)$ we have }
\begin{equation}\label{AZ1}
d_Y\left( h_N, H^Q_1\right) \leq C N^{-\alpha/4}, \mbox{ where }H^Q_1 = F_{h_N(-r), h_N(r)} \left( \sqrt{2r}B^Q\right)
\end{equation}
Before we prove (\ref{AZ1}) we give a brief summary of the ideas. By definition, we have that $H_1^Q$ is given by an appropriate shift and rescaling of $B^Q$, which interpolates the points $(-r, h_N(-r))$ and $(r, h_N(r))$. To better understand how $H^Q_1$ differs from $h_N$ we first do an auxillary rescaling $H_2^Q$ by erasing the part of $h_N$ on the interval $[-r_N, r_N]$ and interpolating the points $\left(-r_N, h_N(-r_N)\right), \left(r_N, h_N(r_N)\right)$ with an appropriate shift and rescaling of $B^Q$. The distance $d_Y(h_N, H_2^Q)$ is easily shown to be $O\left(N^{-\alpha/4}\right)$ using only the strong coupling of $B^Q$ and $\ell^{(2s_1, Q, b-a)}$  on $F(N)$ (this is the first event in (\ref{DefF})). Since $r_N$ is close to $r$ and $h_N(\pm r_N)$ is close to $h_N(\pm r)$ one can show that $d_Y(H_2^Q, H_1^Q) = O\left(N^{-\alpha/4}\right)$. The latter estimate uses the bounds on $B^Q$ and $w\left(B^Q, N^{-\alpha} \right)$ from the second and third event in (\ref{DefF}), since what is involved is a certain stretching of the Brownian bridge $B^Q$. In what follows we supply the details of the above strategy.

We start by defining
$$H^Q_2(s)= \begin{cases} N^{-\alpha/2}\left( a +\sqrt{2s_1} B^Q\left(\frac{sN^{\alpha} + s_1}{2s_1}\right) + \frac{sN^{\alpha} + s_1}{2s_1}(b- a)\right) &\mbox{ for $s\in [-r_N, r_N ], $} \\ f_N(s) &\mbox{ for $ s \in [-r,r] \backslash [-r_N, r_N ]$,}\end{cases}$$
where we recall that $r_N = s_1 N^{-\alpha}$. Notice that 
$$d_Y\left( h_N, H^Q_2\right) = N^{-\alpha/2}  \cdot \Delta\left(2s_1,Q, b - a\right),$$
and so on the event $F(N)$ for all $N \geq 1$ we have
\begin{equation}\label{AZ2}
d_Y\left( h_N, H^Q_2\right)  < N^{-\alpha/4}.
\end{equation}

We next estimate $H^Q_2(s) - H^Q_1(s) $ on the interval $[-r,r]$. Whenever $s\in [-r_N, r_N]$ and we are on the event $F(N)$ we have
\begin{equation}\label{AZ3}
\begin{split}
H^Q_2(s) - H^Q_1(s) =   N^{-\alpha/2}\left( a +\sqrt{2s_1} B^Q\left(\frac{sN^{\alpha} + s_1}{2s_1}\right) + \frac{sN^{\alpha} + s_1}{2s_1}(b - a)\right) -  \\
 - \left( h_N(-r) + \sqrt{2r}B^Q \left( \frac{s + r}{2r} \right) + \frac{s + r}{2r} (h_N(r) - h_N(-r)) \right) = O(N^{-\alpha/4}),
\end{split}
\end{equation}
where the constant in the big $O$ notation depends on $r$. In obtaining the second equality above we used that:
\begin{enumerate}
\item $s_1 = \lfloor rN^{\alpha} \rfloor = rN^{\alpha} + O(1)$, $ b -a \leq 2rN^{\alpha}$,
\item $\left| h_N(-r) - N^{-\alpha/2} \cdot a \right| = \left|h_N(-r) - h_N \left(-r_N\right) \right|  \leq N^{-\alpha/2}$, 
\item$|h_N(r) - N^{-\alpha/2} \cdot b| = \left|h_N(r) - h_N \left(r_N\right) \right| \leq N^{-\alpha/2}$, 
\item on $F(N)$ we have $\sup_{s \in [0,1]} \left|B^Q(s) \right| \leq N^{\alpha/4}$ and $ w(B^Q, N^{-\alpha}) \leq N^{-\alpha/4} $.
\end{enumerate}
For $s \in  [-r,r] \backslash [-r_N, r_N]$, we know that 
\begin{equation*}
\begin{split}
\left| H^Q_1(s) - H^Q_2(s)  \right| \leq   \left| H^Q_1(\pm r_N) - H^Q_2(\pm r_N)  \right| + \left|H^Q_1(s) - H^Q_1\left(\pm r_N\right) \right|+\left|H^Q_2(s) - H^Q_2\left(\pm r_N\right) \right|,
\end{split}
\end{equation*}
where we choose the top sign if $s  > r_N$ and the bottom sign otherwise. Note that the first term above is $O\left(N^{-\alpha/4}\right)$ by (\ref{AZ3}). Substituting the definitions of $H^Q_1$ and $H^Q_2$ we get for $s \in  [-r,r] \backslash [-r_N, r_N ]$
\begin{equation}\label{AZ4}
\begin{split}
\left| H^Q_1(s) - H^Q_2(s)  \right| \leq  O\left(N^{-\alpha/4}\right) +   N^{-\alpha/2} \left| \sqrt{2s_1} B^Q \left(\frac{sN^{\alpha} + s_1}{2s_1}\right) + \frac{sN^{\alpha} \mp s_1}{2s_1}(b - a)  \right| + \\
 \left| \sqrt{2r}\left[ B^Q \left( \frac{s + r}{2r} \right) - B^Q \left( \frac{\pm s_1 N^{-\alpha} + r}{2r} \right) \right]+ \frac{s \mp s_1 N^{-\alpha}   }{2r} (h_N(r) - h_N(-r))\right| =  O\left(N^{-\alpha/4}\right),
\end{split}
\end{equation}
where again we take the top sign if $s  > r_N$ and the bottom sign otherwise and the constant in the big $O$ notation depends only on $r$. In obtaining the last equality we used the same estimates above together with the inequality $|h_N(r) - h_N(-r)| \leq 2r N^{\alpha/2}$. Combining (\ref{AZ3}) and (\ref{AZ4}) we deduce that
\begin{equation}\label{AZ5}
d_Y\left( H^Q_1, H^Q_2\right)  = O\left( N^{-\alpha/4}\right),
\end{equation}
where the constant in the big $O$ notation depends on $r$. Combining (\ref{AZ2}) and (\ref{AZ5}) we deduce (\ref{AZ1}).\\

{\raggedleft {\bf Step 5.}} In this step we first show we have the following  inclusion of events for all large $N$ (depending on $r$) 
\begin{equation}\label{Z6}
 I(N):= F(N) \cap \left\{ h_{N} \in G^{-1}(O) \right\} \cap \left \{ d_Y(h_{N}, G^{-1}(O)^c) > N^{-\alpha/8}  \right\}\subset \left\{ \sqrt{2r}B^Q \in O \right\}. 
\end{equation}
Recall from (\ref{AZ1}) that there exists $N_0$ and $C$ depending on $r$ such that for $N \geq N_0$ and on the event $F(N)$ 
$$H_1^Q =  F_{h_N(-r), h_N(r)} \left( \sqrt{2r}B^Q\right) \mbox{ and }d_Y\left( h_N, H^Q_1\right) \leq C N^{-\alpha/4}.$$ 
By increasing $N_0$ we can also ensure that $CN^{-\alpha/4} < N^{-\alpha/8}$ for $N \geq N_0$. 

Fix $N \geq N_0$ and assume we are on the event $I(N)$. Since $ h_{N} \in G^{-1}(O)$ and $d_Y(h_{N}, G^{-1}(O)^c) > N^{-\alpha/8}$, we see that $H_1^Q \in G^{-1}(O)$. Observe that $G(H_1^Q) = \sqrt{2r} B^Q$ by definition and so we conclude that $\sqrt{2r}B^Q \in O$ on $I(N)$. This proves (\ref{Z6}).\\

From (\ref{Z6}) we know that the LHS of (\ref{Z5}) is bounded by
$$ \limsup_{N \rightarrow \infty} \left[ \tilde{\mathbb{P}}\left( \sqrt{2r}B^Q \in O  \right) +\tilde{\mathbb{P}} \left(F(N)^c \right)\right].$$
In order to finish the proof it suffices to show 
\begin{equation}\label{Z7}
\limsup_{N \rightarrow \infty} \tilde{\mathbb{P}}\left(  \sqrt{2r}B^Q\in O \right) \leq 3\epsilon \mbox{ and } \limsup_{N \rightarrow \infty}\tilde{\mathbb{P}} \left(F(N)^c \right) \leq 3\epsilon.
\end{equation}

In the second part of this step we verify the first inequality in (\ref{Z7}) and for brevity we set $W = \frac{ \log (\epsilon)}{\log (1 - \delta)}$. Observe that 
$$\tilde{\mathbb{P}}\left( \sqrt{2r}B^Q \in O \right) \leq \mathbb{P}^3 (E(\delta, M, N)) +  \tilde{\mathbb{P}} \left( E(\delta,M,N)^c \cap \left\{ \sqrt{2r}B^Q\in O \right\}\right),$$
where we recall that $E(\delta, M, N)$ was defined in Step 1. By assumption on $E(\delta, M, N)$ it suffices to show that
\begin{equation}\label{Z8}
\limsup_{N \rightarrow \infty}  \tilde{\mathbb{P}} \left( E(\delta,M,N)^c \cap \left\{ \sqrt{2r}B^Q\in O \right\}\right) \leq 2\epsilon.
\end{equation}
Notice that on $E(\delta,M,N)^c$ we have that $Q$ is a geometric random variable with parameter \\$Z_t(-s_1,s_1, a ,b, \ell_{bot}; S') \geq \delta$. In particular, we have the a.s. inequality
\begin{equation}\label{Qbound}
\tilde{\mathbb{P}} \left(Q > W | E(\delta,M,N)^c \right) \leq (1 - \delta)^W.
\end{equation}
The above suggests that 
\begin{equation*}
\begin{split}
&\tilde{\mathbb{P}} \left( E(\delta,M,N)^c \cap \left\{ \sqrt{2r}B^Q\in O \right\}\right)  \leq \\
&\tilde{\mathbb{P}} \left( E(\delta,M,N)^c\right) \left[ ( 1 - \delta)^W +  \tilde{\mathbb{P}} \left( \left\{ Q \leq W \right\} \cap \left\{ \sqrt{2r}B^Q\in O \right\} \Big{|}E(\delta,M,N)^c\right)\right] \leq \\
&  (1 - \delta)^W +  \tilde{\mathbb{P}}  \left( \bigcup_{i = 1}^W \left\{\sqrt{2r} B^{i} \in O \right\} \right) \leq  (1 - \delta)^W +  W \cdot \tilde{\mathbb{P}}  \left(  \sqrt{2r} B^{1} \in O \right),
\end{split}
\end{equation*}
where we used in the last inequality that $B^{k} $ are identically distributed. 

Now notice that $ \tilde{\mathbb{P}}  \left(  \sqrt{2r} B^{1} \in O  \right) = \mathbb{P}( B^{\sigma_1} \in O) \leq \epsilon \cdot \frac{\log (1 - \delta)}{ \log (\epsilon)}$ by our choice of $O$ and so we conclude that 
$$(1 - \delta)^W + W \cdot \tilde{\mathbb{P}}  \left( \sqrt{2r} B^{1} \in O  \right) \leq (1 - \delta)^W + \epsilon \leq 2\epsilon.$$
This establishes (\ref{Z8}).\\

{\raggedleft {\bf Step 6.}} In this final step we establish the second inequality in (\ref{Z7}) and as in Step 5 set $W = \frac{ \log (\epsilon)}{\log (1 - \delta)}$. Observe that 
$$\tilde{\mathbb{P}}\left( F(N)^c \right) \leq \mathbb{P}^3 (E(\delta, M, N)) +  \tilde{\mathbb{P}} \left( E(\delta,M,N)^c \cap F(N)^c\right)$$
and so by assumption on $E(\delta, M, N)$ it suffices to show that
\begin{equation}\label{Z9}
\limsup_{N \rightarrow \infty}  \tilde{\mathbb{P}} \left( E(\delta,M,N)^c \cap F(N)^c\right) \leq 2\epsilon.
\end{equation}
Using (\ref{Qbound}) we see that 
$$\tilde{\mathbb{P}} \left( E(\delta,M,N)^c \cap F(N)^c\right)  \leq  \tilde{\mathbb{P}} \left( E(\delta,M,N)^c\right) \left[ ( 1 - \delta)^W +  \tilde{\mathbb{P}} \left( \left\{ Q \leq W \right\} \cap F(N)^c\Big{|}E(\delta,M,N)^c\right)\right] \leq $$
$$ ( 1 - \delta)^W + \sum_{i = 1}^W \tilde{\mathbb{P}} \left( \left\{ Q = i \right\} \cap F(N)^c \cap E(\delta,M,N)^c\right).$$
Since $ ( 1 - \delta)^W \leq \epsilon$, we reduce (\ref{Z9}) to establishing
\begin{equation}\label{Z10}
\limsup_{N \rightarrow \infty}\sum_{i = 1}^W \tilde{\mathbb{P}} \Big{(} \left\{ Q = i \right\} \cap F(N)^c \cap E(\delta,M,N)^c\Big{)} \leq \epsilon.
\end{equation}
One clearly has that 
\begin{equation*}
\begin{split}
\tilde{\mathbb{P}} \Big{(} \left\{ Q = i \right\} \cap F(N)^c \cap E(\delta,M,N)^c\Big{)} \leq \hspace{2mm} &\tilde{\mathbb{P}} \Big{(}  A_i^N \cap E(\delta,M,N)^c\Big{)} + \tilde{\mathbb{P}} \Big{(}  B_i^N\cap E(\delta,M,N)^c\Big{)} + \\
& \tilde{\mathbb{P}} \Big{(}  C_i^N \cap E(\delta,M,N)^c\Big{)},
\end{split}
\end{equation*}
where 
$$A_i^N= \left\{ \Delta\left(2s_1,i, b - a\right) \geq N^{\alpha/4}\right\},  B_i^N=\left\{ \sup_{s \in [0,1]} \left|B^i(s) \right| < N^{\alpha/4} \right\}, C_i^N= \left\{ w(B^i, N^{-\alpha}) > N^{-\alpha/4} \right\}.$$
In addition, we know that since $B^i$ are identically distributed 
$$\sum_{i = 1}^W \tilde{\mathbb{P}} \left( \left\{ Q = i \right\} \cap F(N)^c \cap E(\delta,M,N)^c\right)  \leq $$
$$W \cdot \left[  \tilde{\mathbb{P}} \left(  A_1^N \cap E(\delta,M,N)^c\right) + \tilde{\mathbb{P}} \left(  B_1^N\cap E(\delta,M,N)^c\right) + \tilde{\mathbb{P}} \left(  C_1^N \cap E(\delta,M,N)^c\right)\right] .$$
The above inequality reduces (\ref{Z10}) to showing that 
\begin{equation}\label{ZF}
\begin{split}
&\limsup_{N \rightarrow \infty} \tilde{\mathbb{P}} \left(  A_1^N \cap E(\delta,M,N)^c\right) = 0,\\
&\limsup_{N \rightarrow \infty} \tilde{\mathbb{P}} \left(  B_1^N \cap E(\delta,M,N)^c\right) = 0, \hspace{5mm}
\limsup_{N \rightarrow \infty} \tilde{\mathbb{P}} \left(  C_1^N \cap E(\delta,M,N)^c\right) = 0.
\end{split}
\end{equation}

Notice that by construction 
$$ \tilde{\mathbb{P}} \left(  B_1^N \right) = \sum_{a \leq b} \mathbb{P}^1 \left( \sup_{s \in [0,1]} \left|B^{\sigma, b-a, 2s_1}_s\right| < N^{\alpha/4}\right) \mathbb{P}^3 \left( L^N_1(-s_1) =a, L^N_1(s_1) =b \right)  $$
$$ =  \sum_{a \leq b} \mathbb{P} \left( \sup_{s \in [0,1]} \left|B^{\sigma} (s)\right| < N^{\alpha/4}\right) \mathbb{P}^3 \left( L^N_1(-s_1) =a, L^N_1(s_1) =b \right) = \mathbb{P} \left( \sup_{s \in [0,1]} \left|B^{\sigma} (s)\right| < N^{\alpha/4}\right) ,$$
and the latter clearly converges to $0$ as $N \rightarrow \infty$.

A similar argument shows that 
$$  \tilde{\mathbb{P}} \left(  C_1^N \right)  = \mathbb{P} \left( w(B^\sigma, N^{-\alpha})> N^{-\alpha/4}\right) ,$$
and the latter converges to $0$ as $N \rightarrow \infty$ by the almost sure H{\" o}lder-$1/3$ continuity of the Brownian bridge (see e.g. Proposition 7.8 in Chapter 8 of \cite{Cinlar}). The above estimates establish the second line in (\ref{ZF}).\\ 

In the remainder we study $\tilde{\mathbb{P}} \left(  A_1^N \cap E(\delta,M,N)^c\right) $ and notice that by assumption on $E(\delta,M,N)$ we have that on the event $E(\delta,M,N)^c$ the values $a = L_1(-s_1)$ and $b = L(s_1)$ satisfy 
$$|a + ps_1| \leq MN^{\alpha/2} \mbox{ and } |b - ps_1| \leq MN^{\alpha/2}.$$
The latter implies that 
$$\tilde{\mathbb{P}} \left(  A_1^N \cap E(\delta,M,N)^c\right) \leq \sum_{|z| \leq 2MN^{\alpha/2}} \tilde{\mathbb{P}} \left(  A_1^N \cap \left\{ b -a = \lfloor 2ps_1 + z \rfloor \right\} \right) = $$
$$ = \sum_{|z| \leq 2MN^{\alpha/2}} {\mathbb{P}}^1 \left(  \Delta\left(2s_1,1,  \lfloor 2ps_1 + z \rfloor \right) \geq N^{\alpha/4} \right) \mathbb{P}^3(\left\{ b -a = \lfloor 2ps_1 + z \rfloor \right\}).$$
By Chebyshev's inequality and Theorem \ref{KMT} we know that 
$${\mathbb{P}}^1 \left(  \Delta\left(2s_1,1, \lfloor 2ps_1 + z \rfloor  \right)\geq N^{\alpha/4} \right) \leq C'  N^{-\alpha/4}  e^{c' (\log N)^2},$$
for some constants $C'$ and $c'$ that are independent of $N$ but depend on $M$. The latter inequalities show that 
$$\tilde{\mathbb{P}} \left(  A_1^N \cap E(\delta,M,N)^c\right) \leq   C'  N^{-\alpha/4} e^{c' (\log N)^2} \hspace{-4mm} \sum_{|z| \leq 2MN^{\alpha/2}}\hspace{-2mm}  \mathbb{P}^3(\left\{ b -a = \lfloor 2ps_1 + z \rfloor \right\}) \leq C'   N^{-\alpha/4} e^{c' (\log N)^2} \hspace{-1mm} .$$
Since the latter clearly converges to $0$ as $N \rightarrow \infty$, we conclude (\ref{ZF}), which finishes the proof.

\section{Appendix: Strong coupling of random walks and Brownian bridges}\label{Section8}
In this section we prove a certain generalization of Theorem 6.3 in \cite{LF}, given in Theorem \ref{KMTA} below, which we will use to prove Theorem \ref{KMT} in the main text.

\subsection{Proof of Theorem \ref{KMT}}\label{Section8.1}

Fix $p \in (0,1)$ throughout this and the next sections. Let $X_i$ be i.i.d. random variables with $\mathbb{P}(X_1 = 1) = p$ and $\mathbb{P}(X_1 = 0) = 1- p$. We also let $S_n = X_1 + \cdots + X_n$ denote the random walk with increments $X_i$. For $z \in L_n= \{0,...,n\}$ we let $S^{(n,z)} = \{S_m^{(n,z)} \}_{m = 0}^n$ denote the process with the law of $\{S_m\}_{m = 0}^n$, conditioned so that $S_n = z$. Finally, recall from Section \ref{Section4.2} that $B^\sigma$ stands for the Brownian bridge (conditioned on $B_0 = 0, B_1 = 0$)  with variance $\sigma^2$. We are interested in proving the following result.

\begin{theorem}\label{KMTA}
 For every $b > 0$, there exist constants $0 < C, a, \alpha < \infty$ (depending on $b$ and $p$) such that for every positive integer $n$, there is a probability space on which are defined a Brownian bridge $B^\sigma$ with variance $\sigma^2 = p(1-p)$ and the family of processes $S^{(n,z)} $ for $z \in L_n$ such that
\begin{equation}\label{KMTeqA}
\mathbb{E}\left[ e^{a \Delta(n,z)} \right] \leq C e^{\alpha (\log n)^2} e^{b|z- p n|^2/n},
\end{equation}
where $\Delta(n,z) = \Delta(n,z,B^{\sigma}, S^{(n,z)})=  \sup_{0 \leq t \leq n} \left| \sqrt{n} B^\sigma_{t/n} + \frac{t}{n}z - S^{(n,z)}_t \right|.$
We define $S_t^{(n,z)}$ for non-integer $t$ by linear interpolation.
\end{theorem}

We observe that conditional on $S_n = z$ the law of the path determined by $S_n$ is precisely $\mathbb{P}^{0,n;0,z}_{free}$. Consequently, Theorem \ref{KMTA} implies Theorem \ref{KMT} and in the remainder we focus on establishing the former. Our arguments will follow closely those in Section 6 of \cite{LF}.\\

The proof of Theorem \ref{KMTA} relies on two lemmas, which we state below and whose proofs are deferred to Section \ref{Section8.2}. We begin by introducing some necessary notation. Suppose that $Z$ is a continuous random variable with strictly increasing cumulative distribution function $F$ and $G$ is the distribution function of a discrete random variable, whose support is $\{a_1, a_2,...\}$. Then $(Z,W)$ are {\em quantile-coupled (with distribution functions $(F,G)$)} if $W$ is defined by  
$$W = a_j  \hspace{3mm} \mbox{ if } \hspace{3mm} r_{j-} < Z \leq r_j,$$
where $r_{j-}, r_j$ are defined by 
$$F(r_{j-}) = G(a_j -), \hspace{5mm} F(r_j) = G(a_j).$$
The quantile-coupling has the following property. If 
$$F(a_k-x) \leq G(a_k-) < G(a_k) \leq F(a_k + x),$$
then
\begin{equation}\label{Q1}
|Z - W| = |Z - a_k| \leq x \hspace{5mm} \mbox{ on the event } \{W = a_k \}.
\end{equation}
With the above notation we state the following two lemmas.
\begin{lemma}\label{LQ1}
There exists $\epsilon_0$ (depending on $p$) such that for every $b_1 > 0$ there exist constants $0 < c_1,a_1 < \infty$ such that the following holds. Let $N$ be an $N(0,1)$ random variable. For each integers $m,n$ such that $n > 1$ and $|2m - n| \leq 1$ and every $z \in L_n$, let
$$Z = Z^{(m,n,z)} = \frac{m}{n} z + \sqrt{p(1-p)m \left( 1 - \frac{m}{n}\right) }N, \mbox { so that } Z \sim N\left( \frac{m}{n}z, p(1-p)m \left( 1 - \frac{m}{n}\right) \right).$$
Let $W = W^{(m,n,z)}$ be the random variable, whose law is the same as that of $S_m^{(n,z)}$ and which is quantile-coupled with $Z$. Then if $|z - pn| \leq \epsilon_0n$ and $\mathbb{P}(W = w) > 0$,
\begin{equation}\label{Q2}
\mathbb{E} \left[e^{a_1 |Z - W|} \Big{|} W= w \right] \leq c_1 \cdot \sqrt{n} \cdot  \exp \left(b_1 \frac{(w - pm)^2 + (z - pn)^2}{n}\right).
\end{equation}
\end{lemma}
\begin{lemma}\label{LQ2}
 There exist positive constants $\epsilon_0, c_2,b_2$ (depending on $p$) such that for every integers $m,n$ such that $n \geq 2$ and $|2m - n| \leq 1$, every $z \in L_n$ with $| z - pn|\leq \epsilon_0 n$ and every $w \in \mathbb{Z}$,
$$\mathbb{P}( S_m = w | S_n = z) \leq c_2 n^{-1/2} \exp \left( -b_2 \frac{(w - (z/2))^2}{n}\right).$$
\end{lemma}

\begin{proof}(Theorem \ref{KMTA}) It suffices to prove the theorem when $b$ is sufficiently small. For the remainder we fix $b > 0$ such that $b < b_2/ 37$, where $b_2$ is the constant from Lemma \ref{LQ2}. Let $\epsilon_0$ be the smaller of the two values of $\epsilon_0$ in Lemmas \ref{LQ1} and \ref{LQ2}.

 In this proof, by an {\em $n$-coupling} we will mean a probability space on which are defined a Brownian bridge $B^\sigma$ and the family of processes $\{ S^{(n,z)}: z \in L_n\}$. Notice that for any $n$-coupling if $z \in L_n$, $S_t = S^{(n,z)}_t$ then 
$$\Delta(n,z) =  \sup_{0 \leq t \leq n} \left| \sqrt{n} B^\sigma_{t/n} + \frac{t}{n}z - S^{(n,z)}_t \right| \leq 2n + \sup_{0 \leq t \leq n} |\sqrt{n} B^\sigma_{t/n}|.$$ 
The above together with the fact that there are positive constants $\tilde c$ and $u$ such that \\$\mathbb{E}\left[\exp \left( \sup_{0 \leq t \leq 1 } y|B^\sigma_t| \right) \right] \leq \tilde c e^{uy^2}$ for any $y > 0$ (see e.g. (6.5) in \cite{LF}) imply that 
$$\mathbb{E}\left[ e^{a \Delta(n,z)} \right]  \leq \tilde c e^{(2a + ua^2) n}.$$
Clearly, there exists $a_0 = a_0(b)$ such that if $0 < a < a_0$ then $2a + ua^2 \leq b \epsilon_0^2$.

The latter has the following implications. Firstly, (\ref{KMTeqA}) will hold for any $n$-coupling with $C = \tilde c$, $\alpha = 0$ and $a \in (0,a_0)$ if $z \in L_n$ satisfies $|z - pn| \geq \epsilon_0n$. For the remainder of the proof we assume that $a < a_0$. Let $b_1 = b/20$ and let $a_1, c_1$ be as in Lemma \ref{LQ1} for this value of $b_1$. We assume that $a < a_1$ and show how to construct the $n$-coupling so that (\ref{KMTeqA}) holds for some $C,\alpha$.

We proceed by induction and note that we can find $C \geq \max(1, \tilde c)$ sufficiently large so that for any $n$-coupling with $n \leq 2$ we have
$$ \mathbb{E}\left[ e^{a \Delta(n,z)} \right] e^{-b|z- p n|^2/n} \leq C, \hspace{5mm} \forall z \in L_n, n \leq 2.$$
With the above we have fixed our choice of $a$ and $C$.

We will show that for every positive integer $s$, we have that there exist $n$-couplings for all $n \leq 2^s$ such that
\begin{equation}\label{Q3}
\mathbb{E}\left[ e^{a \Delta(n,z)} \right] e^{-b|z- p n|^2/n} \leq A_n^{s-1} \cdot C, \hspace{5mm} \forall{z \in L_n},
\end{equation}
where $A_n =  2c_1 c_2 n + 2c_1 \sqrt{n}$. The theorem clearly follows from this claim.\\

We proceed by induction on $s$ with base case $s = 1$ being true by our choice of $C$ above. We suppose our claim is true for $s$ and let $2^s < n \leq 2^{s+1}$. We will show how to construct a probability space on which we have a Brownian bridge and a family of processes $\{S^{(n,z)}: |z -p n| \leq \epsilon_0 n \}$, which satisfy (\ref{Q3}). Afterwards we can adjoin (after possibly enlarging the probability space) the processes for $|z| > n \epsilon_0$. Since $C > \tilde c$ and $a < a_0$ we know that (\ref{Q3}) will continue to hold for these processes as well. Hence, we assume that $|z - pn | \leq \epsilon_0 n$. For simplicity we assume that $n = 2k$, where $k$ is an integer such that $2^{s-1} < k \leq 2^s$ (if $n$ is odd we write $n = k + (k+1)$ and do a similar argument).

We define the $n$-coupling as follows:
\begin{itemize}
\item Choose two independent $k$-couplings 
$$\left( \{S^{1 (k,z))}\}_{z \in L_k}, B^1 \right), \hspace{5mm} \left( \{S^{2 (k,z))}\}_{z \in L_k}, B^2 \right), \mbox{ satisfying (\ref{Q3})}.$$
Such a choice is possible by the induction hypothesis.
\item Let $N \sim N(0,1)$ and define the translated {\em normal} variables $Z^z = \frac{z}{2} + \sqrt{\frac{p(1-p)n}{4}}N$ as well as the quantile-coupled random variables $W^z$ as in Lemma \ref{LQ1}. Assume, as we may, that all of these random variables are independent of the two $k$-couplings chosen above. Observe that by our choice of $a$ we have that
\begin{equation}\label{Q4}
\mathbb{E} \left[e^{a |Z^z - W^z|} \Big{|} W^z= w \right] \leq c_1 \cdot \sqrt{n} \cdot  \exp \left(\frac{b}{20} \cdot \frac{(w - kp)^2 + (z - np)^2}{n}\right).
\end{equation}
\item Let
\begin{equation}\label{Q5}
B_t = \begin{cases} 2^{-1/2} B_{2t}^1 + t\sqrt{p(1-p)}N & 0 \leq t \leq 1/2,\\2^{-1/2} B_{2(t-1/2)}^2 + (1-t)\sqrt{p(1-p)}N & 1/2 \leq t \leq 1. \end{cases}
\end{equation}
By Lemma 6.5 in \cite{LF}, $B_t$ is a Brownian bridge with variance $\sigma^2$.
\item Let $S^{(n,z)}_k = W^z$, and 
$$S_{m}^{(n,z)} = \begin{cases}S^{1 (k,W^z)}_m &0 \leq m \leq k, \\ W^z + S_{m-k}^{2 (k,z-W^z)}, &k \leq m \leq n.\end{cases}$$
What we have done is that we first chose the value of $S_k^{(n,z)}$ from the conditional distribution of $S_k$, given $S_n = z$. Conditioned on the midpoint $S^{(n,z)}_k = W^z$ the two halves of the random walk bridge are independent and upto a trivial shift we can use $S^{1 (k,W^z)}$ and $S^{2 (k,z -W^z)}$ to build them.
\end{itemize}
The above defines our coupling and what remains to be seen is that it satisfies (\ref{Q3}) with $s + 1$.

Note that 
$$\Delta(n,z,S^{(n,z)}, B) \leq |Z^z - W^z| + \max \left(\Delta(k, W^z, S^{1 (k,W^z)}, B^1),\Delta(k, z- W^z, S^{2 (k,z - W^z)}, B^2)  \right)$$
and therefore for any $w$ such that $\mathbb{P}(W^z = w) > 0$ we have
$$\mathbb{E}\left[ e^{a\Delta(n,z)} \Big{|} W^z = w\right] \leq \mathbb{E}\left[ e^{a|Z^z - W^z|} \Big{|} W^z = w\right] \times  C A_n^{s-1}\left( e^{b|w - kp|^2/k} + e^{b|z-w - kp|^2/k}\right).$$
In deriving the last expression we used that our two $k$-couplings satisfy (\ref{Q3}) and the simple inequality $\mathbb{E}[e^{\max (Z_1,Z_2)}] \leq\mathbb{E}[e^{Z_1}]  + \mathbb{E}[e^{Z_2}]  $. Taking expectation on both sides above we see that
\begin{equation}\label{Q6}
\mathbb{E}\left[ e^{a\Delta(n,z)} \right] \leq C\cdot (2c_1 \sqrt{n}) \cdot A_n^{s-1} \sum_{w = 0}^k \mathbb{P}(W^z = w) \exp \left(\frac{9}{4} \cdot \frac{b\max(|w - kp|^2,|z-w - kp|^2) }{n} \right) .
\end{equation}
In deriving the last expression we used (\ref{Q4}) and the simple inequality $x^2 + y^2 \leq 5 \max (x^2, (x-y)^2)$ as well as that $k = n/2$.\\

We finally estimate the sum in (\ref{Q6}) by splitting it over the $w$ such that $|w - z/2| > |z-pn|/6$ and $|w - z/2| \leq |z-pn|/6$. Notice that if $|w  - z/2| \leq |z-pn|/6$ we have $\max( |w - pk|^2,|z-w - pk|^2) \leq (2|z - pn|/3)^2$; hence
\begin{equation}\label{Q7} 
\sum_{w: |w - z/2| \leq |z-pn|/6} \mathbb{P}(W^z = w) \exp \left(\frac{9}{4} \cdot \frac{\max( |w - kp|^2,|z-w - kp|^2) }{n} \right) \leq \exp \left(\frac{|z-pn|^2}{n} \right).
\end{equation}
To handle the case $|w - z/2| > |z-pn|/6$  we use Lemma \ref{LQ2}, from which we know that 
$$\mathbb{P}(W^z = w) = \mathbb{P}(S_k = w | S_n = z) \leq c_2 n^{-1/2} \exp \left(- b_2 \frac{(w - (z/2))^2}{n} \right).$$
Using the latter together with the fact that for $|w - z/2| > |z-pn|/6$ we have that $(w- z/2)^2 > \frac{1}{16} \max \left( (w - kp)^2 , |z-w - kp|^2 \right)$ we see that
\begin{equation}\label{Q8} 
\begin{split}
&\sum_{w: |w - z/2| > |z-pn|/6} \mathbb{P}(W^z = w) \exp \left(\frac{9}{4} \cdot \frac{b \max( |w - kp|^2,|z-w - kp|^2) }{n} \right) \leq \\
&\sum_{w = 1 }^kc_2n^{-1/2} \exp \left(- \frac{b}{16} \cdot \frac{(w - kp)^2}{n} \right) \leq c_2 \sqrt{n}.
\end{split}
\end{equation}
Combining the above estimates we see that
$$\mathbb{E}\left[ e^{a\Delta(n,z)} \right] \leq C\cdot (2c_1 \sqrt{n}) \cdot A_n^{s-1} \left[ \exp \left(\frac{|z-pn|^2}{n} \right) + c_2 \sqrt{n} \right] \leq C \cdot A_n^{s} \exp \left(\frac{|z-pn|^2}{n} \right) .$$
The above concludes the proof.

\end{proof}

\subsection{Proof of Lemmas \ref{LQ1} and \ref{LQ2}}\label{Section8.2}
Our proofs of Lemma \ref{LQ1} and \ref{LQ2} will mostly follow (appropriately adapted) arguments from Sections 6.4 and 6.5 in \cite{LF}. We begin with two technical lemmas.

\begin{lemma}\label{LQ3}
 There is a constant $c > 0$ (depending on $p$) such that for integers $m,n,z$ and real $w$ with $n \geq 2$, $|2m - n| \leq 1$, $| z - pn| \leq c n$, $|w | \leq cn$ and $ w + \frac{m}{n}z  \in \mathbb{N}$ one has
\begin{equation}\label{Q9}
\mathbb{P} \left( S_m = w + \frac{m}{n}z  \Big{|} S_n= z \right) = \frac{1}{\sqrt{2\pi \sigma_{n,z}^2}} \exp \left( - \frac{w^2}{2 \sigma_{n,z}^2} + O\left(\frac{1}{\sqrt{n}} + \frac{|w|^3}{n^2} \right) \right),
\end{equation}
where $\sigma_{n,z}^2 = (n/4)(z/n)(1 - z/n)$.
\end{lemma}
\begin{proof}
The above result is very similar to Lemma 6.7 in \cite{LF} and so we only sketch the main ideas that go into the proof. The statement of the lemma will follow if we can show that if $|j| \leq cn$ we have
$$p(j,m,n,z) = \mathbb{P} \left( \lfloor S_m -  \frac{m}{n}z \rfloor = j  \Big{|} S_n= z \right) = \frac{1}{\sqrt{2\pi \sigma_{n,z}^2}} \exp \left( - \frac{j^2}{2 \sigma_{n,z}^2} + O\left(\frac{1}{\sqrt{n}} + \frac{j^3}{n^2} \right) \right).$$
Using Stirling's approximation formula $A! = \sqrt{2 \pi} A^{A + 1/2} e^{-A} [1 + O(A^{-1})],$
we see that 
$$p(0,m,n,z) = \mathbb{P} \left( \lfloor S_m - \frac{m}{n}z \rfloor = 0  \Big{|} S_n= z \right) =  \frac{1}{\sqrt{2\pi \sigma_{n,z}^2}} \left( 1 +  O\left(n^{-1}  \right) \right) .$$
Let us remark that in order to apply Stirling's approximation, we needed to choose $c$ sufficiently small so that $ \frac{m}{n}z$, $z$, $m - \frac{m}{n}z$, $n- z$ all tend to infinity faster than $\epsilon n$  for some $\epsilon > 0$ fixed (depending on $p$) as $n \rightarrow \infty$. For the remainder we assume such a $c$ is chosen and the constant in the big $O$ notation above depends on it. 

Let us focus on the case $j > 0$ (if $j < 0$ a similar argument can be applied). For $j > 0$  and $A(j,m,n,z) = \frac{\left( m + z - 2 \lceil \frac{m}{n} z \rceil - 2j\right)^2 - (m-z)^2 }{\left( 2 \lceil \frac{m}{n} z \rceil + 2j + 2 + m - z\right)^2 - (m-z)^2} $ we have
$$p(j + 1,m,n,z)  = p(j ,m,n,z)  \times A(j,m,n,z) $$
and so
$$p(j ,m,n,z)  = p(0,m,n,z) \times \prod_{i = 1}^jA(i,m,n,z). $$
Given our earlier result for $p(0,m,n,z)$ to finish the proof it remains to show that 
\begin{equation}\label{Q9}
\sum_{i = 1}^j \log \left[A(j,m,n,z) \right] = - \frac{j^2 }{2\sigma_{n,z}^2} + O\left(\frac{1}{\sqrt{n}} +\frac{j^3}{n^2} \right).
\end{equation}
Notice that if we choose $c$ sufficiently small, we have that 
$$A(j,m,n,z) = 1 - B(j,m,n,z), \mbox{ where } B(j,m,n,z) = \frac{8jm}{m^2 - (m-z)^2} + O\left( \frac{j}{n^2} + \frac{1}{n}\right)$$
and $0 \leq B(j,m,n,z) \leq \frac{1}{2}$. Using the latter together with the fact that $\log(1 + x) = x + O(x^2)$ for $|x| \leq 1/2$ we get
$$\sum_{i = 1}^j \log \left[A(j,m,n,z) \right] = - \sum_{i =1}^j\frac{8im}{m^2 - (m-z)^2} + O\left( \frac{j^2}{n^2} + \frac{j}{n}\right) =  - \frac{4j^2m}{m^2 - (m-z)^2} +  O\left( \frac{j^3}{n^2} + \frac{1}{\sqrt{n}}\right) .$$
To conclude the proof we observe that
$$\frac{4j^2m}{m^2 - (m-z)^2} = \frac{j^2}{m \cdot \frac{z}{2m} \cdot \left( 1 - \frac{z}{2m} \right)} = \frac{j^2}{\frac{n}{2} \cdot \frac{z}{n} \cdot \left( 1 - \frac{z}{n} \right)} +O\left( \frac{j^2}{n^2} \right) = \frac{j^2}{2\sigma_{n,z}^2} + O\left( \frac{j^3}{n^2} + \frac{1}{\sqrt{n}}\right) . $$
\end{proof}

We now state without proof an easy large deviation estmiate, which can be established in the same way one establishes large deviations for binomial random variables.
\begin{lemma}\label{LQ4}
There exists an $\eta > 0$ (depending on $p$) such that, for any $a > 0$, there exist $C = C(a) < \infty$ and $\gamma = \gamma(a) > 0$ with the following properties. For any integers $m,n,z$ with $n \geq 2$, $|2m - n| \leq 1$, $| z - pn| \leq \eta n$ one has
\begin{equation}\label{Q10}
\mathbb{P} \left( \left| S_m - \frac{m}{n} \right| > am \Big{|} S_n = z\right) \leq C e^{-\gamma m}.
\end{equation}
\end{lemma}

It is clear that Lemmas \ref{LQ3} and \ref{LQ4} imply Lemma \ref{LQ2}. What remains is to prove Lemma \ref{LQ1}, to which we now turn.

\begin{proof}(Lemma \ref{LQ1})
Notice that we only need to prove the lemma for $n$ sufficiently large. In order to simplify the notation we will assume that $n$ is even and so $m = n/2$ (the case $n$ odd can be handled similarly). 

We start by choosing $\epsilon_0 \leq \min (c,\eta)$ with $c$ and $\eta$ as in Lemmas \ref{LQ3} and \ref{LQ4} respectively. We denote
$$Z = Z_{n,z} = z/2 +  \sqrt{p(1-p)n/4} N,  \hspace{5mm} \hat Z = \hat Z_{n,z} = z/2 + \sigma_{n,z}N,$$
where we recall that $\sigma_{n,z}^2 = (n/4)(z/n)(1 - z/n)$ and let $W = W_{n,z}$ be the random variable with distribution $S^{(n,z)}_{n/2}$ that is quantile coupled with $N$. Notice that $W$ is also quantile coupled with $Z$ and $\hat Z$. We write $F = F_{n,z}$ for the distribution function of $\hat Z$ and $G = G_{n,z}$ for the distribution function of $W$. We observe that from Lemmas \ref{LQ3} and \ref{LQ4}, the random variable $W - \lfloor z/2 \rfloor$ satisfies the conditions of Lemma 6.9 in \cite{LF}, from which we deduce that there are constants $c', \epsilon' > 0$ and $N' \in \mathbb{N}$ such that for $n \geq N'$ and $|x- z/2|  \leq \epsilon' n$ we have
\begin{equation}\label{Q11}
F\left( x - c'\left[1 + \frac{(x- z/2)^2}{n}  \right]\right) \leq G(x- 1) \leq G(x + 1) \leq F\left( x + c' \left[1 + \frac{(x- z/2)^2}{n}  \right]\right).
\end{equation}
In the remainder we assume $\epsilon_0 \leq \epsilon'$ as well. It follows from (\ref{Q1}) and (\ref{Q11}) that 
\begin{equation}\label{R1}
 |\hat Z - W| \leq c'\left[1 + \frac{(W- z/2)^2}{n}\right],
\end{equation}
for all $n \geq N'$, provided that $|z - pn| \leq \epsilon_0 n$, $|W - z/2| \leq \epsilon_0 n$. In addition, we have the following string of inequalities for any $a > 0$
$$\mathbb{E}\left[ e^{a|Z - \hat Z|} \Big{|} W = w\right] \leq \mathbb{E}\left[ e^{a(Z - \hat Z)} +e^{-a(Z - \hat Z)}  \Big{|} W = w\right] \leq \frac{ \mathbb{E}\left[ e^{a(Z - \hat Z)} +e^{-a(Z - \hat Z)}  \right]}{\mathbb{P}(W = w)} = \frac{2e^{a^2\sigma(n,p)^2/2}}{\mathbb{P}(W = w)},$$
where $\sigma(n,p) = \sqrt{n/4} \cdot \left|\sqrt{p(1-p)} - \sqrt{(z/n)(1 - z/n)} \right|$. It follows from Lemma \ref{LQ3} that if $|w - z/2| \leq  \epsilon_0 n$ and $|z - pn| \leq \epsilon_0 n$ then we have for some $C > 0$ and all $n \geq 2$ that
\begin{equation}\label{R2}
\mathbb{E}\left[ e^{a|Z - \hat Z|} \Big{|} W = w\right] \leq Ce^{a^2\sigma(n,p)^2/2} \sqrt{n} \cdot \exp \left( C \frac{(w - z/2)^2}{n}\right).
\end{equation}
Combining (\ref{R1}) and (\ref{R2}) we see that for some (possibly larger than before) $C > 0$ we have
\begin{equation}\label{R3}
\mathbb{E} \left[e^{a|W - Z|} \Big{|}W = w \right] \leq \mathbb{E} \left[e^{a|W - \hat Z|} e^{a|Z - Z|} \Big{|}W = w \right] \leq Ce^{a^2\sigma(n,p)^2/2} \sqrt{n} \cdot \exp \left( C \frac{(w - z/2)^2}{n}\right),
\end{equation}
provided $ n \geq N'$, $|w - z/2| \leq\epsilon_0 n$ and $|z - pn| \leq \epsilon_0 n$. 

Notice that by possibly taking $\epsilon_0$ smaller we can make $\sigma(n,p) \leq \sqrt{n/4} \cdot c_p | z/n - p|$, where $c_p = \frac{2}{p(1-p)}$. Using the latter together with (\ref{R3}) and Jensen's inequality we have for any $k \in \mathbb{N}$ that
\begin{equation*}
 \mathbb{E} \left[e^{(1/k)|W - Z|} \Big{|}W = w \right] \leq \mathbb{E} \left[e^{|W - Z|} \Big{|}W = w \right]^{1/k} \leq  (\sqrt{n}C)^{1/k} \cdot \exp \left( \frac{c_p(z-pn)^2}{nk} + C \frac{(w - z/2)^2}{nk}\right),
\end{equation*}
and if we further use that $(x +y )^2 \leq 2x^2 + 2y^2$ above we see that
\begin{equation}\label{R4}
 \mathbb{E} \left[e^{(1/k)|W - Z|} \Big{|}W = w \right] \leq  (\sqrt{n}C)^{1/k} \cdot \exp \left( \frac{(c_p + 1/2)(z-pn)^2}{nk} + \frac{2C(w - pm)^2}{nk}\right),
\end{equation}
provided $ n \geq N'$, $|w - z/2| \leq\epsilon_0 n$ and $|z - pn| \leq \epsilon_0 n$.\\

Suppose now that $b_1$ is given, and let $k$ be sufficiently large so that 
$$\frac{c_p + 1/2}{k} \leq b_1 \mbox{ and } \frac{2C}{k} \leq b_1.$$
If $a_1 \leq 1/k$ we see from (\ref{R4}) that
\begin{equation}\label{R5}
 \mathbb{E} \left[e^{a_1|W - Z|} \Big{|}W = w \right] \leq  C^{1/k}\sqrt{n} \cdot \exp \left( \frac{b_1(z-pn)^2}{n} + \frac{b_1(w - pm)^2}{n}\right),
\end{equation}
provided $ n \geq N'$, $|w - z/2| \leq\epsilon_0 n$ and $|z - pn| \leq \epsilon_0 n$. If $|z-pn| > \epsilon_0 n$ or $|w - z/2| >  \epsilon_0 n$ we observe that 
$$\frac{b_1(z-pn)^2}{n} + \frac{b_1(w - pm)^2}{n} \geq \frac{b_1\epsilon_0^2 n}{3} .$$
One easily observes that if $a_1 \leq a_0$ with $a_0$ sufficiently small and $C \geq \tilde c$ with $\tilde c$ sufficiently large we have for any $w$ such that $\mathbb{P}(W = w) > 0$ that
$$ \mathbb{E} \left[e^{a_1|W - Z|} \Big{|}W = w \right] \leq  C^{1/k}\sqrt{n} \cdot \exp \left( \frac{b_1\epsilon_0^2 n}{3}\right).$$
The latter statements suggest that (\ref{R5}) holds for all $w$ such that $\mathbb{P}(W = w) > 0$ and $n \geq N'$, which concludes the proof of the lemma.

\end{proof} 

\bibliographystyle{amsplain}
\bibliography{PD}

\end{document}